\documentclass[11pt]{amsart} 

\usepackage[
  hmarginratio={1:1},     
  vmarginratio={1:1},     
  textwidth=420pt,        
  heightrounded,          
]{geometry}

\usepackage{amsmath,amssymb,amsfonts,amsthm,amscd,graphicx,psfrag,epsfig,fge}
\usepackage{xcolor}
\usepackage{caption}
\usepackage{subcaption}
\definecolor{blue}{rgb}{0,0,0.7}
\definecolor{red}{rgb}{0.75, 0, 0}
\definecolor{midnight}{rgb}{0.0,0.2,0.4}
\usepackage{stmaryrd}
\usepackage[titletoc,toc]{appendix}
\usepackage{hyperref}
\hypersetup{colorlinks=true, citecolor=blue,linkcolor=midnight}
\usepackage[all]{xy}           
\usepackage{marginnote}  
\usepackage{float}
\usepackage{mathtools}
\usepackage{shuffle}
\usepackage{tikz}
\usepackage{tkz-euclide}
\usetikzlibrary[topaths]
 \usetikzlibrary{arrows}

\newcount\mycount
\usepackage[T2A,T1]{fontenc}
\usepackage[utf8]{inputenc}
\usepackage[russian,english]{babel}

\numberwithin{equation}{section}

\usepackage{amsmath,amssymb,amsfonts,amsthm,amscd,graphicx,psfrag,epsfig,mathabx}
\usepackage{color}
\definecolor{blue}{rgb}{0,0,0.7}
\definecolor{red}{rgb}{0.75, 0, 0}
\usepackage[titletoc,toc]{appendix}

\newtheorem{theorem}{Theorem}[section]
\newtheorem*{theorem*}{Theorem}

\newtheorem{lemma}[theorem]{Lemma}
\newtheorem{proposition}[theorem]{Proposition}
\newtheorem{corollary}[theorem]{Corollary}

\newtheorem{definition}[theorem]{Definition}
\newtheorem{remark}[theorem]{Remark}
\newtheorem{example}[theorem]{Example}

\newcommand{\be}{\begin{equation}}
\newcommand{\bs}{\begin{split}}
\newcommand{\ee}{\end{equation}}
\newcommand{\es}{\end{split}}
\newcommand{\lra}{\longrightarrow}

\def\H{\mathcal{H}}
\def\F{\mathcal{F}}

\def\P{\mathrm{P}}

\def\L{\mathcal{L}}
\def\C{\mathbb{C}}
\def\Q{\mathbb{Q}}

\def\Li{\textup{Li}}
\def\ort{\textup{ort}}

\newcommand\rsmraise[1]{%
  \ifx#1\displaystyle .8\else
    \ifx#1\textstyle .8\else
      \ifx#1\scriptstyle .6\else
        .45%
      \fi
    \fi
  \fi}
  
\title[]{On the Goncharov depth conjecture and a formula for volumes of orthoschemes}
\author[]{Daniil Rudenko}

\date{}

\begin{document}
\dedicatory{{\it \large To Alexander Goncharov for his 60th birthday}}
\begin{abstract}
We prove a conjecture of Goncharov,  which says that any multiple polylogarithm can be expressed via polylogarithms of depth at most half of the weight. We give an explicit formula for this presentation, involving a summation over trees that correspond to decompositions of a polygon into quadrangles. 

Our second result is a formula for  volume of hyperbolic orthoschemes, generalizing the formula of Lobachevsky in  dimension three to an arbitrary dimension. We show a surprising relation between the two results, which comes from the fact that hyperbolic orthoschemes are parametrized by configurations of points on $\mathbb{P}^1.$ In particular, we derive both formulas from their common generalization.
\end{abstract}

\maketitle

\begin{figure*}[!hb]
\centering
\begin{tikzpicture}[remember picture, overlay,scale=4.5]
\pgfmathsetmacro{\radius}{1};

  \foreach \number in {1,...,8}{
        \mycount=\number
        \advance\mycount by -1
  \multiply\mycount by 45
        \advance\mycount by  11.25
      \coordinate (N-\number) at (\the\mycount:\radius) {};
    }
  \foreach \number in {9,...,16}{
        \mycount=\number
        \advance\mycount by -1
  \multiply\mycount by 45
        \advance\mycount by 33.75
      \coordinate (N-\number) at (\the\mycount:\radius) {};
    }

\tkzDefPoint(0,0){O}
\tkzDefPoint(\radius,0){A}
  
\tkzDefCircle[orthogonal through=N-1 and N-9](O,A) \tkzGetPoint{B}
\tkzInterCC(O,A)(B,N-1)
\tkzClipCircle(O,A)
\tkzDrawCircle[thick,color=midnight](B,N-1) 
 
\tkzDefCircle[orthogonal through=N-2 and N-10](O,A) \tkzGetPoint{B}
\tkzInterCC(O,A)(B,N-2)
\tkzDrawCircle[thick,color=midnight](B,N-2) 
 
\tkzDefCircle[orthogonal through=N-3 and N-11](O,A) \tkzGetPoint{B}
\tkzInterCC(O,A)(B,N-3)
\tkzDrawCircle[thick,color=midnight](B,N-3) 

\tkzDefCircle[orthogonal through=N-4 and N-12](O,A) \tkzGetPoint{B}
\tkzInterCC(O,A)(B,N-4)
\tkzDrawCircle[thick,color=midnight](B,N-4) 

\tkzDefCircle[orthogonal through=N-5 and N-13](O,A) \tkzGetPoint{B}
\tkzInterCC(O,A)(B,N-5)
\tkzDrawCircle[thick,color=midnight](B,N-5) 

\tkzDefCircle[orthogonal through=N-6 and N-14](O,A) \tkzGetPoint{B}
\tkzInterCC(O,A)(B,N-6)
\tkzDrawCircle[thick,color=midnight](B,N-6)

\tkzDefCircle[orthogonal through=N-7 and N-15](O,A) \tkzGetPoint{B}
\tkzInterCC(O,A)(B,N-7)
\tkzDrawCircle[thick,color=midnight](B,N-7)  

\tkzDefCircle[orthogonal through=N-8 and N-16](O,A) \tkzGetPoint{B}
\tkzInterCC(O,A)(B,N-8)
\tkzDrawCircle[thick,color=midnight](B,N-8)  
 
\tkzDefCircle[orthogonal through=N-1 and N-16](O,A) \tkzGetPoint{B}
\tkzInterCC(O,A)(B,N-1)
\tkzDrawCircle[thick,color=midnight](B,N-1)

\tkzDefCircle[orthogonal through=N-2 and N-9](O,A) \tkzGetPoint{B}
\tkzInterCC(O,A)(B,N-2)
\tkzDrawCircle[thick,color=midnight](B,N-2)  

\tkzDefCircle[orthogonal through=N-3 and N-10](O,A) \tkzGetPoint{B}
\tkzInterCC(O,A)(B,N-3)
\tkzDrawCircle[thick,color=midnight](B,N-3)  

\tkzDefCircle[orthogonal through=N-4 and N-11](O,A) \tkzGetPoint{B}
\tkzInterCC(O,A)(B,N-4)
\tkzDrawCircle[thick,color=midnight](B,N-4)

\tkzDefCircle[orthogonal through=N-5 and N-12](O,A) \tkzGetPoint{B}
\tkzInterCC(O,A)(B,N-5)
\tkzDrawCircle[thick,color=midnight](B,N-5)   

\tkzDefCircle[orthogonal through=N-6 and N-13](O,A) \tkzGetPoint{B}
\tkzInterCC(O,A)(B,N-6)
\tkzDrawCircle[thick,color=midnight](B,N-6)

\tkzDefCircle[orthogonal through=N-7 and N-14](O,A) \tkzGetPoint{B}
\tkzInterCC(O,A)(B,N-7)
\tkzDrawCircle[thick,color=midnight](B,N-7)     

\tkzDefCircle[orthogonal through=N-8 and N-15](O,A) \tkzGetPoint{B}
\tkzInterCC(O,A)(B,N-8)
\tkzDrawCircle[thick,color=midnight](B,N-8)  


\tkzDefCircle[orthogonal through=N-7 and N-16](O,A) \tkzGetPoint{B}
\tkzInterCC(O,A)(B,N-7)
\tkzDrawCircle[thick,color=midnight](B,N-7)  

\tkzDefCircle[orthogonal through=N-6 and N-16](O,A) \tkzGetPoint{B}
\tkzInterCC(O,A)(B,N-6)
\tkzDrawCircle[thick,color=midnight](B,N-6)  

\tkzDefCircle[orthogonal through=N-6 and N-11](O,A) \tkzGetPoint{B}
\tkzInterCC(O,A)(B,N-6)
\tkzDrawCircle[thick,color=midnight](B,N-6)  

\tkzDefCircle[orthogonal through=N-3 and N-16](O,A) \tkzGetPoint{B}
\tkzInterCC(O,A)(B,N-3)
\tkzDrawCircle[thick,color=midnight](B,N-3)  

\tkzDefCircle[orthogonal through=N-12 and N-6](O,A) \tkzGetPoint{B}
\tkzInterCC(O,A)(B,N-12)
\tkzDrawCircle[thick,color=midnight](B,N-12)

\tkzDefCircle[orthogonal through=N-3 and N-9](O,A) \tkzGetPoint{B}
\tkzInterCC(O,A)(B,N-3)
\tkzDrawCircle[thick,color=midnight](B,N-3)    

\end{tikzpicture} 
\end{figure*}

\pagebreak 

\tableofcontents

\section{Introduction}
\subsection{Depth reduction for multiple polylogarithms}
Multiple polylogarithms are multivalued functions on $a_1,\dots,a_k\in \C$ defined by Goncharov in \cite{Gon95} for $n_1,\dots,n_k\in \mathbb{N}$ by the power series 
\[
\begin{split}
&\Li_{n_1,n_2,\dots, n_k}(a_1,a_2,\dots,a_k)=\sum_{m_1>m_2>\dots>m_k>0}\frac{a_1^{m_1} a_2^{m_2}\dots a_k^{m_k}}{m_1^{n_1}m_2^{n_2}\dots m_k^{n_k}}
 \ \  \text{for} \ \  |a_1|,|a_2|,\dots, |a_k| <1.\\
\end{split}
\]
The number $n=n_1+\dots+n_k$ is called the weight of the multiple polylogarithm, and the number $k$ is called its depth. The case of $k=1$ is of particular interest and is known as the classical polylogarithm:
\[
\Li_n(a)=\sum_{m=1}^{\infty}\frac{a^m}{m^n} \ \ \text{for} \ \  |a|<1.
\]

Classical polylogarithms  appeared in the 18th and 19th centuries under different guises in the works of Leibniz, Euler, Spence,  Abel, Kummer, Lobachevsky, and many others. It was noticed early on that polylogarithms satisfy functional equations of an algebraic nature. Here is the famous five-term relation for the 
 dilogarithm obtained by Abel:
\[
\begin{split}
	&\Li_2(a_1)+\Li_2(a_2)+\Li_2\left(\frac{1-a_1}{1-a_1a_2}\right)+\Li_2\left(1-a_1a_2\right)+\Li_2\left(\frac{1-a_2}{1-a_1a_2}\right)\\
	&=\frac{\pi^2}{6}-\log(a_1)\log(1-a_1)-\log(a_2)\log(1-a_2)+\log\left(\frac{1-a_1}{1-a_1a_2}\right)\log\left(\frac{1-a_2}{1-a_1a_2}\right).\\
\end{split}
\]

The structure of this equation can be clarified with the following two observations. First, all the terms in the RHS are products of polylogarithms of lower weight. We introduce the symbol  $\stackrel{\shuffle}{=}$ for ``equal modulo products of polylogarithms of lower weight.'' The precise meaning of $\stackrel{\shuffle}{=}$ is explained in \S\ref{SectionMHTS}. Second, define the cross-ratio of four points $x_1, x_2 ,x_3,x_4 \in \mathbb{P}^1$ by the formula
\[
[x_1,x_2,x_3,x_4]=\frac{(x_{1}-x_{2})(x_{3}-x_{4})}{(x_{1}-x_{4})(x_{3}-x_{2})}.
\]
The five-term relation implies that:
\be \label{FormulaFiveTerm}
\sum_{i=0}^4 (-1)^i\Li_2([x_0,\dots,\widehat{x_i},\dots,x_4])\stackrel{\shuffle}{=}0.
\ee

Finding equations for classical polylogarithms of higher weight is notoriously difficult as their length grows very fast. The first equation for $\Li_4$ in more than two variables   was found by Gangl in \cite{Gan16} with computer-assisted search and contains more than 1000 terms; similar equations for $\Li_5$ are not known yet. It seems that a more manageable goal is to write down equations for multiple polylogarithms and deduce from them equations for  $\Li_{n}.$ Here is an example in weight two:
\[
\begin{split}
&\Li_{1,1}(a_1,a_2)+\Li_{1,1}(a_2,a_1)+\Li_{2}(a_1a_2)=\Li_{1}(a_1)\Li_{1}(a_2);\\
&\Li_{1,1}(a_1,a_2)=\Li_{2}\left(\frac{1-a_1\ \  }{1-a_2^{-1}}\right)-\Li_{2}\left(\frac{a_2}{a_2-1}\right)-\Li_{2}\left(a_1a_2\right).\\
\end{split}
\]
The first of these equations can be generalized to an arbitrary weight and is a part of a family of so-called quasi-shuffle relations (also called ``second shuffle relations'') for multiple polylogarithms. The second relation is more intricate. In combination, these relations imply the five-term relation and allow expression of any polylogarithm of weight two via the  dilogarithm  $\Li_{2}(a)$ and products of logarithms.

In \cite[Conjecture 7.6]{Gon01} Goncharov formulated the depth conjecture giving a criterion for when a sum of polylogarithms can be expressed using polylogarithms of lower depths. For polylogarithms of depth larger than half of the weight, the criterion is always satisfied. We prove this part of the depth conjecture.
 
\begin{theorem}\label{MainTheoremDepth}
Any multiple polylogarithm 
of weight $n\geq 2$ can be expressed as a linear combination of multiple polylogarithms of depth at most $\left \lfloor n/2 \right \rfloor$ 
 and products of polylogarithms of lower weight.
\end{theorem}

This result is sharp as it is easy to show using a coproduct (discussed in \S \ref{SectionHodgeIterated}) that a general multiple polylogarithm of weight $n$ cannot be expressed via  multiple  polylogarithms of depth strictly less than $\left \lfloor n/2 \right \rfloor.$ Theorem \ref{MainTheoremDepth} was known for $n\leq 5$; see \cite{Cha17}, \cite{CGR19b}, and \cite{CGR19}  for further results on the depth reduction problem for multiple polylogarithms.

The proof of Theorem \ref{MainTheoremDepth} is based on the construction of so-called quadrangular polylogarithms, which are certain multivalued functions on the moduli space $\mathfrak{M}_{0,2n+2}$ of configurations of $2n+2$ distinct points on $\mathbb{P}^1.$ We show that multiple polylogarithms can be expressed via quadrangular polylogarithms.  On the other hand, we provide an explicit formula for quadrangular polylogarithms via multiple polylogarithms of depth at most $n.$ The combination of these results implies Theorem \ref{MainTheoremDepth}.

Our work came out of an attempt to understand the results of Coxeter in \cite{Cox36}.
Coxeter found a relation between  non-Euclidean orthoschemes, which are higher-dimensional generalizations of right triangles, and configurations of points in $\mathbb{P}^1.$ Based on  Coxeter's results   B{\"o}hm outlined in \cite{Boh63} an integration procedure, from which one can deduce that the volume of a non-Euclidean orthoscheme in any dimension can be expressed via multiple polylogarithms. After a simple adjustment,  our formula for quadrangular polylogarithms gives an explicit formula for the volume of a non-Euclidean orthoscheme.

\subsection{Quadrangular polylogarithms}

The proof of Theorem \ref{MainTheoremDepth}  is based on the construction of quadrangular polylogarithms
\[
\textup{QLi}_{n,k}(x_0,\dots,x_{2n+1})
\]
for $n\geq 1,\ k\geq 0,$ which are certain iterated integrals of weight $n+k$ on the configuration space of $ 2n+2$ points on $\mathbb{P}^1.$ For $n=1$ we recover the classical polylogarithm 
\[
\textup{QLi}_{1,k}(x_0,x_1,x_2,x_3)=(-1)^{k+1}\Li_{k+1}([x_0,x_1,x_2,x_3]).
\]
We show that multiple polylogarithms can be expressed (modulo products) as linear combinations of the functions $\textup{QLi}_{n,n}$ and $\textup{QLi}_{n,n+1}$ by an explicit formula, see \S \ref{SectionQuadrangularUniversality}.

 On the other hand, the quadrangular polylogarithm $\textup{QLi}_{n,k}(x_0,\dots, x_{2n+1})$ can be expressed as a linear combination of multiple polylogarithms of depth at most $n,$ whose arguments are certain products of cross-ratios of points $x_0,\dots, x_{2n+1}$. We position points $x_0,\dots,x_{2n+1}$ at the vertices of a convex polygon $\P$.  Every quadrangle inside $\P$ with vertices $x_{i_0}, x_{i_1},x_{i_2},x_{i_3}$  determines a cross-ratio $[x_{i_0}, x_{i_1},x_{i_2},x_{i_3}],$ which is a regular function
 on $\mathfrak{M}_{0,2n+2}.$ The quadrangular polylogarithm $\textup{QLi}_{n,k}$ is equal to a sum of multiple polylogarithms of depth at most $n$, whose arguments are certain products of cross-ratios, corresponding to disjoint quadrangles in $\P.$
 
Consider a rooted tree $\mathrm{t}$ with $n$ vertices of two types (``even'' and ``odd'') labeled by complex numbers. Using quasi-shuffle relations for polylogarithms, we construct a polylogarithm $\Li_k(\mathrm{t})$ which is a certain sum of multiple polylogarithms of weight $n+k$ and depth at most $n$ evaluated at products of numbers labeling the vertices of $\mathrm{t}$. Every decomposition of $\P$ into disjoint quadrangles (we call such decompositions ``quadrangulations'') determines such a tree: $\mathrm{t}$ is the dual graph of the triangulation with vertices labeled by the cross-ratios of the corresponding points. Then the following theorem holds.
\begin{theorem}[A formula for quadrangular polylogarithms via multiple polylogarithms]\label{MainTheoremQuadrangulationCluster}
Consider a configuration of points $x_0,\dots,x_{2n+1}\in \mathbb{P}^1$.
Let $\mathcal{Q}(\P)$ be the set of quadrangulations of a convex $(2n+2)$-gon $\P$ with vertices labeled by points $x_0,\dots, x_{2n+1}\in \mathbb{P}^1.$ For a quadrangulation $Q\in \mathcal{Q}$ denote by $\mathrm{t}_Q$ the corresponding labeled tree. Then the following formula holds:
\be
\label{FormulaMainTHeorem2}
\textup{QLi}_{n,k}(x_0,\dots,x_{2n+1})\stackrel{\shuffle}{=}\sum_{Q\in \mathcal{Q}(\P)} \Li_{k}(\mathrm{t}_Q).
\ee
\end{theorem}
 
\begin{figure}
\centering
\begin{subfigure}[b]{0.3\textwidth}
 \centering
\begin{tikzpicture}[transform shape]
  \foreach \number in {1,2,3}{
        \mycount=\number
        \advance\mycount by -1
  \multiply\mycount by 120
        \advance\mycount by 60
      \node[draw, very thin, color=gray,inner sep=0.001cm] (N-\number) at (\the\mycount:2.0cm) {};
    }
  \foreach \number in {4,5,6}{
        \mycount=\number
        \advance\mycount by -1
  \multiply\mycount by 120
        \advance\mycount by 120
      \node[draw, very thin, color=gray,inner sep=0.001cm](N-\number) at (\the\mycount:2cm) {};
    }
\draw[midnight, line width=0.25mm] (N-1) -- (N-4);
\draw[midnight, line width=0.25mm] (N-4) -- (N-2);
\draw[midnight, line width=0.25mm] (N-2) -- (N-5);
\draw[midnight, line width=0.25mm] (N-5) -- (N-3);
\draw[midnight, line width=0.25mm] (N-3) -- (N-6);
\draw[midnight, line width=0.25mm] (N-6) -- (N-1);

\draw[midnight, line width=0.25mm] (N-5) -- (N-1);

\node[yshift=-0.2cm, xshift=-0.14cm] at (N-5) {$0$};
\node[yshift=-0.03 cm, xshift=-0.13cm] at (N-2) {$1$};
\node[yshift=0.13cm, xshift=-0.21cm] at (N-4) {$2$};
\node[yshift=0.13cm, xshift=0.21cm] at (N-1) {$3$};
\node[yshift=-0.03 cm, xshift=0.13cm] at (N-6) {$4$};
\node[yshift=-0.2cm, xshift=0.14cm] at (N-3) {$5$};

\end{tikzpicture}
\caption{Quadrangles $(0,3,4,5)$ and $(0,1,2,3)$ correspond to arguments of ${\scriptstyle \Li_{1,1}([x_0, x_3, x_4, x_5],[x_0, x_1, x_2, x_3])}.$}
\end{subfigure}
\hfill
\begin{subfigure}[b]{0.3\textwidth}
  \centering
\begin{tikzpicture}[transform shape]
  \foreach \number in {1,2,3}{
        \mycount=\number
        \advance\mycount by -1
  \multiply\mycount by 120
        \advance\mycount by 60
      \node[draw, very thin, color=gray,inner sep=0.001cm] (N-\number) at (\the\mycount:2cm) {};
    }
  \foreach \number in {4,5,6}{
        \mycount=\number
        \advance\mycount by -1
  \multiply\mycount by 120
        \advance\mycount by 120
      \node[draw, very thin, color=gray,inner sep=0.001cm](N-\number) at (\the\mycount:2cm) {};
    }
\draw[midnight, line width=0.25mm] (N-1) -- (N-4);
\draw[midnight, line width=0.25mm] (N-4) -- (N-2);
\draw[midnight, line width=0.25mm] (N-2) -- (N-5);
\draw[midnight, line width=0.25mm] (N-5) -- (N-3);
\draw[midnight, line width=0.25mm] (N-3) -- (N-6);
\draw[midnight, line width=0.25mm] (N-6) -- (N-1);

\draw[midnight, line width=0.25mm] (N-6) -- (N-2);

\node[yshift=-0.2cm, xshift=-0.14cm] at (N-5) {$0$};
\node[yshift=-0.03 cm, xshift=-0.13cm] at (N-2) {$1$};
\node[yshift=0.13cm, xshift=-0.21cm] at (N-4) {$2$};
\node[yshift=0.13cm, xshift=0.21cm] at (N-1) {$3$};
\node[yshift=-0.03 cm, xshift=0.13cm] at (N-6) {$4$};
\node[yshift=-0.2cm, xshift=0.14cm] at (N-3) {$5$};

\end{tikzpicture}
\caption{Quadrangles $(0,1,4,5)$ and $(1,2,3,4)$ correspond to arguments of ${\scriptstyle \Li_{1,1}([x_0, x_1, x_4, x_5],[x_1,x_2, x_3, x_4]^{-1})}.$}
\end{subfigure}
\hfill
\begin{subfigure}[b]{0.3\textwidth}
\centering
\begin{tikzpicture}[transform shape]
  \foreach \number in {1,2,3}{
        \mycount=\number
        \advance\mycount by -1
  \multiply\mycount by 120
        \advance\mycount by 60
      \node[draw, very thin, color=gray,inner sep=0.001cm] (N-\number) at (\the\mycount:2cm) {};
    }
  \foreach \number in {4,5,6}{
        \mycount=\number
        \advance\mycount by -1
  \multiply\mycount by 120
        \advance\mycount by 120
      \node[draw, very thin, color=gray,inner sep=0.001cm](N-\number) at (\the\mycount:2cm) {};
    }
\draw[midnight, line width=0.25mm] (N-1) -- (N-4);
\draw[midnight, line width=0.25mm] (N-4) -- (N-2);
\draw[midnight, line width=0.25mm] (N-2) -- (N-5);
\draw[midnight, line width=0.25mm] (N-5) -- (N-3);
\draw[midnight, line width=0.25mm] (N-3) -- (N-6);
\draw[midnight, line width=0.25mm] (N-6) -- (N-1);

\draw[midnight, line width=0.25mm] (N-4) -- (N-3);

\node[yshift=-0.2cm, xshift=-0.14cm] at (N-5) {$0$};
\node[yshift=-0.03 cm, xshift=-0.13cm] at (N-2) {$1$};
\node[yshift=0.13cm, xshift=-0.21cm] at (N-4) {$2$};
\node[yshift=0.13cm, xshift=0.21cm] at (N-1) {$3$};
\node[yshift=-0.03 cm, xshift=0.13cm] at (N-6) {$4$};
\node[yshift=-0.2cm, xshift=0.14cm] at (N-3) {$5$};

\end{tikzpicture}
\caption{Quadrangles $(0,1,2,5)$ and $(2,3,4,5)$ correspond to arguments of ${\scriptstyle \Li _{1,1}([x_0, x_1, x_2, x_5],[x_2, x_3, x_4, x_5])}$.}
\end{subfigure}
\caption{Three quadrangulation of a hexagon correspond to three terms in (\ref{Example1}).}
\label{FigureHexagonQuadrangulations}
 \end{figure}
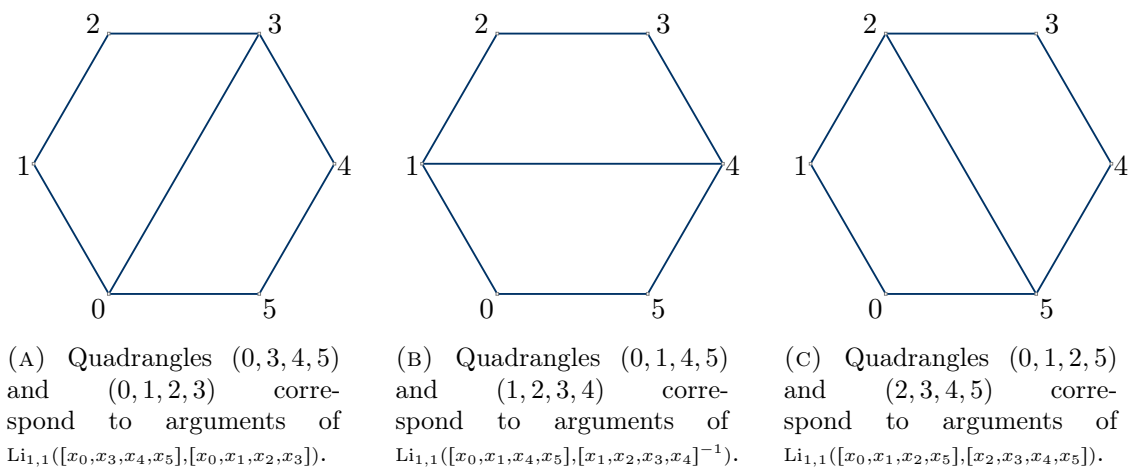 
 For instance, we have the relation
\be
\begin{split}
\label{Example1}
&\textup{QLi}_{2,0}(x_0,x_1,x_2, x_3,x_4,x_5)\stackrel{\shuffle}{=}
\Li_{1,1}([x_0, x_3, x_4, x_5],[x_0, x_1, x_2, x_3])\\
&-\Li_{1,1}([x_0, x_1, x_4, x_5],[x_1,x_2, x_3, x_4]^{-1})+\Li _{1,1}([x_0, x_1, x_2, x_5],[x_2, x_3, x_4, x_5]),\\
\end{split}
\ee
where each term corresponds to one of the three quadrangulations of a convex hexagon, see Figure~\ref{FigureHexagonQuadrangulations}.
Note that this formula is ``nonlinear'' because functions appearing   in~(\ref{Example1}) are not iterated integrals on $\mathfrak{M}_{0,6}$   as they have a nontrivial monodromy around the divisor 
\[
[x_0,x_3,x_4,x_5][x_0,x_1,x_2,x_3]=-\frac{(x_0-x_1)(x_2-x_3)(x_4-x_5)}{(x_1-x_2)(x_3-x_4)(x_5-x_0)}=1.
\]
Theorem~\ref{MainTheoremQuadrangulationCluster} implies Theorem~\ref{MainTheoremDepth}, because  multiple polylogarithms are linear combinations of $\textup{QLi}_{n,k}$ for $k=n$ or $k=n+1,$ which can be expressed via multiple polylogarithms of depth at most $n$ by~(\ref{FormulaMainTHeorem2}).

\subsection{Volumes of non-Euclidean polytopes and a theorem of B{\"o}hm}\label{SectionIntroductionVolumes}

Multiple polylogarithms appear in  computations of volumes of non-Euclidean polytopes, starting with the work of Lobachevsky at the dawn of the hyperbolic geometry era. In \cite{Lob36} Lobachevsky was able to express the volume of an orthoscheme in $\mathbb{H}^3$ as a sum of seven dilogarithms evaluated at certain functions of dihedral angles of an orthoscheme. Every polytope in $\mathbb{H}^3$ can be dissected into orthoschemes, so its volume can be expressed via the dilogarithm.  Lobachevsky applied this argument to a pair of scissors congruent polytopes and obtained a nontrivial identity between integrals, which he was able to check using classical integration techniques. To Lobachevsky, this was a strong argument in favor of his ``Imaginary Geometry.'' 

It is natural to ask if volumes of non-Euclidean polytopes in higher dimensions could be expressed via multiple polylogarithms of higher weight. After we had completed our work, we discovered that a positive answer to this question  could have been extracted from ``Coxeters Integrationmethode'' developed by B{\"o}hm in \cite{Boh63} based on the results of Coxeter (see \cite{Cox35}, \cite{Cox36}). 

\begin{theorem}[B\"ohm, 1962] \label{TheoremBohmVolumes}
	The volume of a hyperbolic orthoscheme of dimension $2n-1$ or $2n$  can be expressed\footnote{For a precise statement see Theorem \ref{MainTheoremQuadrangulationVolume}.} via multiple polylogarithms of weight  $n$ evaluated at algebraic functions in exponents of dihedral angles of the orthoscheme.
\end{theorem}

We want to make a few remarks about the significance and the peculiar history of Theorem \ref{TheoremBohmVolumes}.  Using the Schl{\"a}fli formula, one can easily see that the volume of a non-Euclidean simplex is an iterated integral of $d\log$-forms on a non-rational variety, which was written explicitly (for the spherical case) in \cite[Theorem 1]{Aom77}. Not all such integrals can be expressed via multiple polylogarithms (see  \cite{BD20}), so Theorem \ref{TheoremBohmVolumes} is a much more subtle result. Goncharov showed  in \cite{Gon99} that the volume of a non-Euclidean simplex is a period of a mixed Hodge structure of mixed Tate type of geometric origin. The universality conjecture of Goncharov (\cite[Conjecture 7.4]{Gon01}) implies that any such period can be expressed via multiple polylogarithms.

It was conjectured in  \cite{BG47} and \cite{Mul54} that the volume of an orthoscheme can be expressed via classical polylogarithms. In \cite{Boh63} B{\"o}hm found that in seven-dimensional space, polylogarithms of  weight four and depth two are unavoidable, raising doubts in this conjecture. Using the coproduct defined in  \S \ref{SectionHodgeIterated}, it is easy to show that the conjecture is indeed false, see also \cite[Theorem 8.10]{Wec91}. In the same work, B{\"o}hm shows how Coxeter's approach in \cite{Cox36}  to the proof of the Lobachevsky formula could be generalized to higher dimensions (see also \cite[\S 5.4-5.9]{BH80} for a detailed exposition). With modern techniques, the method of B{\"o}hm can be used to prove Theorem \ref{TheoremBohmVolumes}. B{\"o}hm neither formulates Theorem \ref{TheoremBohmVolumes} explicitly nor gives a formula for  the volume of an orthoscheme. Such formula for orthoschemes of dimensions up to five can be found in \cite{BG47}, \cite{Boh60}, \cite{Mul54}, \cite{Kel92}, and \cite{Kel95}. 

In \S \ref{SecOrt} we revisit these questions from a new perspective. We show that non-Euclidean orthoschemes of dimension $m-1$ are parametrized by a certain abelian cover of  $\mathfrak{M}_{0,m+2}$ and deduce the results of Coxeter from \cite{Cox36}. Next, we derive an explicit formula for the volume of an orthoscheme, which implies Theorem \ref{TheoremBohmVolumes}. This formula is almost identical to (\ref{FormulaMainTHeorem2}) and we deduce both formulas from a more general statement, see \S \ref{SectionIntroductionOnProofs}.

\subsection{Non-Euclidean orthoschemes and $\mathfrak{M}_{0,m+2}$}

A non-Euclidean orthoscheme is a geodesic simplex in  $\mathbb{H}^{m-1}$ or $\mathbb{S}^{m-1}$ with faces $H_1, \dots, H_m$ such that $H_i$ is orthogonal to $H_j$ for $|i-j|>1.$ For $m=3$ this is a right triangle and for $m=4$ this is a tetrahedron with all faces being right triangles. 

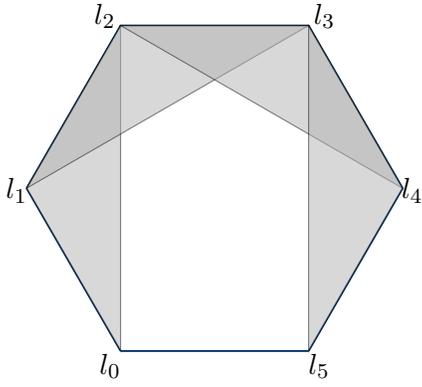
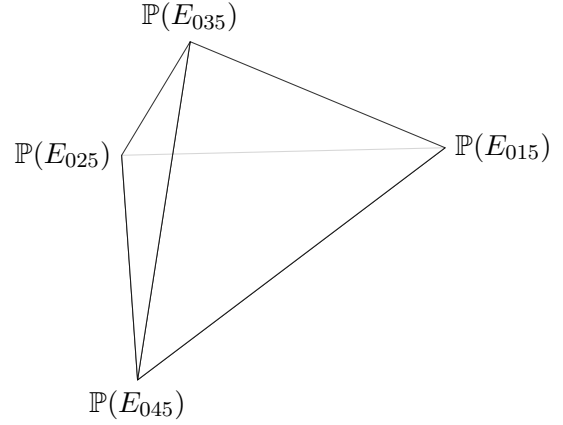
\begin{figure}
\centering
\begin{subfigure}[b]{0.45\textwidth}
\centering
\begin{tikzpicture}[transform shape]
  \foreach \number in {1,2,3}{
        \mycount=\number
        \advance\mycount by -1
  \multiply\mycount by 120
        \advance\mycount by 60
      \coordinate (N-\number) at (\the\mycount:2.5cm) {};
    }
  \foreach \number in {4,5,6}{
        \mycount=\number
        \advance\mycount by -1
  \multiply\mycount by 120
        \advance\mycount by 120
      \coordinate (N-\number) at (\the\mycount:2.5cm) {};
    }
\draw[midnight, line width=0.25mm] (N-1) -- (N-4);
\draw[midnight, line width=0.25mm] (N-4) -- (N-2);
\draw[midnight, line width=0.25mm] (N-2) -- (N-5);
\draw[midnight, line width=0.25mm] (N-5) -- (N-3);
\draw[midnight, line width=0.25mm] (N-3) -- (N-6);
\draw[midnight, line width=0.25mm] (N-6) -- (N-1);

\node[yshift=-0.2cm, xshift=-0.14cm] at (N-5) {$l_0$};
\node[yshift=-0.03 cm, xshift=-0.13cm] at (N-2) {$l_1$};
\node[yshift=0.13cm, xshift=-0.21cm] at (N-4) {$l_2$};
\node[yshift=0.13cm, xshift=0.21cm] at (N-1) {$l_3$};
\node[yshift=-0.03 cm, xshift=0.13cm] at (N-6) {$l_4$};
\node[yshift=-0.2cm, xshift=0.14cm] at (N-3) {$l_5$};

\draw[-, fill=black!30, opacity=.5] (N-5)--(N-2)--(N-4)--cycle;
\draw[-, fill=black!30, opacity=.5] (N-2)--(N-4)--(N-1)--cycle;
\draw[-, fill=black!30, opacity=.5] (N-4)--(N-1)--(N-6)--cycle;
\draw[-, fill=black!30, opacity=.5] (N-1)--(N-6)--(N-3)--cycle;
\end{tikzpicture}
\caption{Corner triangles are dual to faces of the orthoscheme. Triangles with vertices $l_0$ and $l_5$ are vertices of the orthoscheme \textup{ort}(x).}
\label{FigureOrthoschemeA}
\end{subfigure}
\hfill
\begin{subfigure}[b]{0.45\textwidth}
 \centering
\begin{tikzpicture}[line join = round, line cap = round]
\pgfmathsetmacro{\factor}{1/sqrt(2)};
\coordinate [label=right:$\mathbb{P}(E_{015})$] (A) at (2.3,0,-2*\factor);
\coordinate [label=left:$\mathbb{P}(E_{025})$] (B) at (-2,-0.1,-2*\factor);
\coordinate [label=above:$\mathbb{P}(E_{035})$] (C) at (0,2.5,2*\factor);
\coordinate [label=below:$\mathbb{P}(E_{045})$] (D) at (-0.7,-2,2*\factor);
\draw[-, fill=white, opacity=0.8] (A)--(D)--(B)--cycle;
\draw[-, fill=white, opacity=0.8] (A) --(D)--(C)--cycle;
\draw[-, fill=white, opacity=0.8] (B)--(D)--(C)--cycle;
\end{tikzpicture}
\caption{Orthoscheme  $(Q;H_1,H_2,H_3,H_4)$ in $\mathbb{P}^3=\mathbb{P}(E)$ with vertices $\mathbb{P}(E_{0,i,5})$.}
\label{FigureOrthoschemeB}
\end{subfigure}
\caption{Bijection between points of $\mathfrak{M}_{0,6}$ and generic $3$-dimensional orthoschemes.}
\label{FigureOrthoscheme}
\end{figure}
It will be convenient for us to work with an algebro-geometric generalization of non-Euclidean orthoschemes, which we call projective orthoschemes. A {\it projective simplex} $S=(Q; H_1,\dots, H_{m})$ in $\mathbb{P}^{m-1}$  is a configuration of a quadric $Q$ and a collection of hyperplanes $H_1,\dots, H_{m}.$ Given a non-Euclidean simplex, we construct a projective simplex by complexification and projectivization of Klein's model, see \S \ref{SectionProjectiveSimplex}. Hyperplanes $H_i$ are obtained from the faces of the simplex; the quadric $Q$ is obtained from the absolute. 

A projective simplex is called an {\it orthoscheme} if $Q$ is smooth and $H_i$ is orthogonal to $H_j$ with respect to the quadric for $|i-j|>1.$ In \S \ref{SecOrt1} we define a notion of a {\it generic orthoscheme}; being generic is an open condition. Inspired by the results of Coxeter from \cite{Cox36} we construct a bijection $x\mapsto \ort(x)$ between points of $\mathfrak{M}_{0,m+2}$ and generic orthoschemes in $\mathbb{P}^{m-1}.$ For $m=3$ this recovers Gauss's pentagramma mirificum \cite{Gau66}. Figure \ref{FigureOrthoscheme} illustrates the case $m=4.$ Our construction of  orthoscheme $\textup{ort}(x)$ is based on an algebraic version of Maslov index, introduced in an unpublished note by Kashiwara. Later this construction was reinterpreted by Beilinson in terms of sheaf cohomology, see \cite[\S 1.5]{LV80} and \cite{Tho06}.

A point $x\in \mathfrak{M}_{0,m+2}$ defines a configuration of  $m+2$ lines $l_0,\dots, l_{m+1}$ in a $2$-dimensional vector space $V$.  Consider a convex polygon with $m+2$ vertices labeled by the lines (see Figure \ref{FigureOrthoschemeA} for the case $m=4$). Denote by $E$ a  vector space of dimension $m,$ parametrizing collections of vectors in $l_i$ with the sum equal to zero:
\[
E=\textup{Ker}\left(\bigoplus_{i=0}^{m+1} l_i\stackrel{\sum} {\lra}V\right).
\]
 For every triangle $(i_1,i_2,i_3)$ we have a line 
\[
E_{i_1,i_2,i_3}=\textup{Ker}\left(l_{i_1}\oplus l_{i_2}\oplus l_{i_3}\stackrel{\sum} {\lra}V\right)
\]
in $E.$ The  Maslov index is a quadratic form $q$ on $E$ such that  lines $E_{i_1,i_2,i_3}$ and $E_{j_1,j_2,j_3}$ are orthogonal with respect to $q$ if and only if triangles $(i_1,i_2,i_3)$ and $(j_1,j_2,j_3)$ have disjoint interiors. From here it is easy to construct an orthoscheme $\textup{ort}(x)$ in $\mathbb{P}(E)$ (see Figure \ref{FigureOrthoschemeB}). It is a configuration of a quadric $Q$ defined by equation $q=0$ and hyperplanes $H_i,$ which are polar to lines $E_{i-1, i, i+1}$ for $1\leq i \leq m.$ These lines correspond to corner triangles of the polygon, and interiors of two corner triangles $(i-1,i,i+1)$  and $(j-1,j,j+1)$ intersect only if $|i-j|\leq 1.$ Thus the configuration ${\textup{ort}(x)=(Q;H_1,\dots,H_m)}$ is a projective orthoscheme. 

\begin{theorem}\label{MainTheoremOrthoschemes}
Generic orthoschemes in $\mathbb{P}^{m-1}$  are in bijection with points of $\mathfrak{M}_{0,m+2}.$ Faces of an orthoscheme $\textup{ort}(x)$ corresponding to a configuration $x=(x_0,\dots,x_{m+1})$ are isometric to orthoschemes, and correspond to configurations 
\[
(x_0,\dots,\widehat{x}_i,\dots,x_{m+1})\in \mathfrak{M}_{0,m+1}
\] 
for $1\leq i\leq m.$ 
\end{theorem}

An orientation of a projective simplex $S$ is a choice of a family of projective subspaces of maximal dimension  on quadrics $Q\cap \bigcap_{i\in I} H_i$ for all subsets $I\subseteq\{1,\dots,m\}.$  Oriented projective orthoschemes are parametrized by an abelian cover  $\mathfrak{M}_{0,m+2}^{s}$ of $\mathfrak{M}_{0,m+2}.$  This cover can be characterized as the minimal cover on which square roots of cross-ratios are single-valued. Projective orthoschemes coming from hyperbolic geometry have a canonical orientation and are parametrized by a subset $\mathfrak{M}_{0,m+2}^h\subseteq \mathfrak{M}_{0,m+2}^s,$ which is a real ball of dimension $m-1.$

The volume of an oriented orthoscheme $\ort(x^s)$ for $x^s\in \mathfrak{M}_{0,m+2}^s$ is an iterated integral on $\mathfrak{M}_{0,m+2}^{s}$. Our goal is to express the volume explicitly  in terms of multiple polylogarithms. By the Chern-Gauss-Bonnet Theorem, it  suffices to do that for orthoschemes of odd dimension, so we assume that $m=2n.$ 
 
Consider a point $x^s\in \mathfrak{M}_{0,2n+2}^{s},$ which lies over  $x=(x_0,\dots,x_{2n+1})\in\mathfrak{M}_{0,2n+2}.$ Let $\P$ be a convex $(2n+2)$-gon with vertices labeled by points $x_0,\dots,x_{2n+1}\in \mathbb{P}^1$.  For a quadrangulation  $Q\in \mathcal{Q}(\P)$ of $\P$ we consider a tree $\mathrm{t},$ which is  dual to the quadrangulation $Q$. We label a vertex of the tree $\mathrm{t}$ corresponding to a $4$-gon $(i_0, i_1,i_2,i_3)$ by the square root of the cross-ratio of the corresponding points $\sqrt{[x_{i_0},x_{i_1},x_{i_2},x_{i_3}]}.$ Let $\textup{ALi}(\mathrm{t})$ be the alternating sum of $\textup{Li}_0(\mathrm{t})$ over all choices of signs of square roots divided by $2^n$ (see \S \ref{SectionAlternatingPolylogarithms}). Multiple polylogarithms are periods of framed mixed Hodge-Tate structures and in \cite[\S 4.7]{Gon99} Goncharov defined their real periods $\mathrm{per}_\mathbb{R},$ see \S \ref{SectionMHTS}.

\begin{theorem}[Quadrangulation formula volumes of orthoschemes]\label{MainTheoremQuadrangulationVolume}
For $x^s\in \mathfrak{M}^h_{2n+2}$ we have the following formula for the hyperbolic volume of an orthoscheme $\ort(x^s):$ 
\be \label{FormulaQuadrangulationVolume}
\textup{Vol}(\ort(x^s))=\sum_{Q\in \mathcal{Q}(\mathrm{P})} \mathrm{per}_{ \mathbb{R}}\left (\textup{ALi}(\mathrm{t}_Q)\right).
\ee
\end{theorem}
 
 For $n=2$ the formula (\ref{FormulaQuadrangulationVolume}) is equivalent to the formula of Lobachevsky from \cite{Lob36}. Theorem \ref{MainTheoremQuadrangulationVolume} immediately implies Theorem \ref{TheoremBohmVolumes}.

\subsection{Philosophy of mixed Tate motives}\label{SectionIntroduction_Philosophy}
The philosophy of mixed Tate motives and their relation to multiple polylogarithms and volumes of hyperbolic polytopes was developed by Goncharov, see \cite{Gon95} for an overview. According to this philosophy, there should exist a graded Hopf algebra $\H_{\mathcal{M}}$ of framed mixed Tate motives over an arbitrary field $\mathrm{F}.$ Currently, the existence of $\H_{\mathcal{M}}$ is known only for some fields, see \cite{DG05}.  Next, Goncharov conjectures that  $\H_{\mathcal{M}}$ is  generated as a $\mathbb{Q}$-vector space by 
  elements 
\[ 
\Li^{\mathcal{M}}_{n_1,n_2,\dots, n_k}(a_1,a_2,\dots,a_k)\text{\ \ for\ \ }a_1,\dots,a_k\in \mathrm{F},
\]
called {\it motivic multiple polylogarithms}, see \cite[Conjecture 7.4]{Gon01}. All relations between multiple polylogarithms in the usual sense are expected to originate from the corresponding relations in $\H_{\mathcal{M}}.$ 
 
 Let 
 \[
 \L_{\mathcal{M}}=\H_{\mathcal{M}}/(\H_{\mathcal{M},>0}\cdot \H_{\mathcal{M},>0})
 \] 
 be the corresponding Lie coalgebra of indecomposable elements and denote by $\mathcal{I}_\mathcal{M}$ its coideal of elements of weight greater than one.
 Denote the graded dual Lie algebra to $\L_\mathcal{M}$ by $\L_\mathcal{M}^{\vee}.$  Consider a filtration on $\L_\mathcal{M}^{\vee}$ defined by the formula $\F_0 \L_\mathcal{M}^{\vee}=\L_\mathcal{M}^{\vee},$ $\F_{-1} \L_\mathcal{M}^{\vee}=\mathcal{I}_\mathcal{M}^{\vee}$ and 
\[
\F_{-m-1} \L^{\vee}_\mathcal{M}=[\mathcal{I}_\mathcal{M}^{\vee},\F_{-m} \L^{\vee}_\mathcal{M}] \text{\ \ for\ \ } m\geq 1.
\]
The Goncharov depth conjecture \cite[Conjecture 7.6]{Gon01} states that the filtration $\F_{\bullet} \L^{\vee}_\mathcal{M}$ is dual to the filtration on $\L_\mathcal{M}$ by the depth of motivic multiple polylogarithms. The Lie coalgebra $\L_\mathcal{M}$ is graded by weight, from where it is easy to see that (for $m\geq 1$) elements in $(\F_{-m} \L^{\vee}_\mathcal{M})^{\vee}$ have weight at least $2m,$ so motivic multiple polylogarithms of weight $2m$ and $2m+1$ should lie in the subspace spanned by multiple polylogarithms of depth at most $m.$ We prove this result unconditionally, see Theorem \ref{MainTheoremDepth}.

The philosophy of mixed Tate motives predicts that there is a way to construct elements of $\H_{\mathcal{M}}$ out of cohomology groups of algebraic varieties, which carry mixed Hodge structure of mixed Tate type, see \cite{Del71}, \cite{Del74}. For a projective simplex $S=(Q; H_1,\dots, H_{2n})$ in $\mathbb{P}^{2n-1}$  consider the following  cohomology group with rational coefficients:
\be \label{FormulaMHSorthoscheme}
H^{2n-1}\left(\mathbb{P}^{2n-1}\fgebackslash Q, \left (\bigcup_{i=1}^{2n} H_i\right )\fgebackslash Q;\mathbb{Q}\right).
\ee
Goncharov showed in \cite[\S 3.4]{Gon99} that it has mixed Tate type and thus is supposed to define an element $m(S)\in \H_{\mathcal{M}}$.                               

The existence  of $\H_\mathcal{M}(F)$ is only known for some classes of fields, including finite fields and number fields, but not $\C.$ Because of that, we work not with mixed Tate motives, but with their {\it Hodge realizations}. Let $\H$ be the Hopf algebra of framed mixed Hodge-Tate structures, see \S \ref{SectionMHTS}. Conjecturally,  $\H_\mathcal{M}(\C)$ is a subalgebra of $\H$, which consists of mixed Hodge-Tate structures of geometric origin, i.e., coming from cohomology groups of algebraic varieties. Goncharov constructed elements in $\H$ called Hodge multiple polylogarithms 
\[
\Li^{\H}_{n_1,n_2,\dots, n_k}(a_1,a_2,\dots,a_k) \text{\ for \ } a_1,a_2,\dots,a_k\in \C.
\]
For an oriented projective simplex $S$ he also constructed an  element $h(S)\in \H.$ We derive Theorems \ref{MainTheoremDepth}, \ref{MainTheoremQuadrangulationCluster} and \ref{MainTheoremQuadrangulationVolume} from the corresponding statements about elements of $\H.$ 

The main tool for working with mixed Hodge-Tate structures is the coproduct $\Delta^\H.$ On one hand, the coproduct is very explicit, see \S \ref{SectionMHTS}. On the other hand, the coproduct can be used to prove identities between elements of $\H$ by reducing them to identities of lower weight, see \ref{SectionRigidity}. Goncharov found the coproduct of Hodge multiple polylogarithms and the coproduct of $h(S),$ which is related to the Dehn invariant of the scissors  congruence class of a projective simplex $S.$

\subsection{On proofs of Theorems \ref{MainTheoremQuadrangulationCluster} and \ref{MainTheoremQuadrangulationVolume}}
\label{SectionIntroductionOnProofs}

Our main result is Theorem \ref{TheoremFormalQF}, which we call the ``formal quadrangulation formula'' and which we now  motivate. Consider a convex $(2n+2)$-gon $\P$ with vertices labeled by points $x_0,\dots,x_{2n+1}\in \mathbb{P}^1.$ Recall that for every convex quadrangle $(i_0,i_1,i_2,i_3)$ in $\P$ we defined a cross-ratio $[x_{i_0},x_{i_1},x_{i_2},x_{i_3}]$.  For a quadrangulation $Q\in \mathcal{Q}(\P)$ consider  a multiple polylogarithm, whose arguments are products of cross-ratios corresponding to quadrangles in $Q$. For example, for $n=3$ we could take
\be \label{FormulaFormalPolylogExample}
\Li_{1,1,1}([x_0,x_1,x_2,x_3],[x_0,x_3,x_4,x_5],[x_0,x_5,x_6,x_7]).
\ee
The coproduct of (\ref{FormulaFormalPolylogExample}) in $\H$ is quite complicated. It has a ``simple part'', which looks like a deconcatenation coproduct

\begin{align*}
&1\otimes \Li_{1,1,1}([x_0,x_1,x_2,x_3],[x_0,x_3,x_4,x_5],[x_0,x_5,x_6,x_7])\\
+&\Li_{1}([x_0,x_1,x_2,x_3])\otimes \Li_{1,1}([x_0,x_3,x_4,x_5],[x_0,x_5,x_6,x_7])\\
+&\Li_{1,1}([x_0,x_1,x_2,x_3],[x_0,x_3,x_4,x_5])\otimes \Li_{1}([x_0,x_5,x_6,x_7])\\
+&\Li_{1,1,1}([x_0,x_1,x_2,x_3],[x_0,x_3,x_4,x_5],[x_0,x_5,x_6,x_7])\otimes 1,\\
\end{align*}
while other terms form the ``complicated part.'' The element 
\[
\textup{T}_\P=\sum_{Q\in \mathcal{Q}(\P)}\Li (\mathrm{t}_Q)
\]
appearing in Theorems \ref{MainTheoremQuadrangulationCluster} and \ref{MainTheoremQuadrangulationVolume} is characterized by the fact that the ``complicated part'' of the coproduct vanishes. 
 
To formalize this idea, we define a Hopf algebra $\F_\P$, whose elements are symbols like
\[
[\textup{cr}(0,1,2,3),1|\textup{cr}(1,3,4,7),1|\textup{cr}(4,5,6,7),1]
\]
which we think of as formal versions of polylogarithms of like  (\ref{FormulaFormalPolylogExample}). This is a very simple object: essentially, a free commutative Hopf algebra with the quasi-shuffle product and the deconcatenation coproduct ${\Delta^{\mathcal{PP}}\colon \F_\P \lra \F_\P\otimes \F_\P.}$ Using the Connes-Kreimer Hopf algebra of rooted trees, we construct a special element $\textup{T}_\P\in \F_\P$ whose coproduct can be expressed via elements of a similar type for subpolygons of $\P.$  We define a coaction ${\Delta^{\mathcal{HP}}\colon \F_\P \lra \H\otimes \F_\P}$ which is related to the ``complicated part'' of the coproduct of multiple polylogarithms. The formal quadrangulation formula states that elements $\textup{T}_\P \in \F_\P$ are coinvariants of this coaction. 

In order to prove Theorem \ref{MainTheoremQuadrangulationCluster} we construct a collection of maps $\Li_k$ for $k\geq 0$ from $\F_\P$ to  the  Lie coalgebra $\L$ of  indecomposable elements of $\H.$  $\Li_0$ is just the multiple polylogarithm and $\Li_k$ for  $k>0$  is a slight generalization of it. We deduce from the formal quadrangulation formula that the ``complicated part'' of  the coproduct of $\Li_k(\textup{T}_\P)$ vanishes, and the remaining part is identical to the coproduct for the quadrangular polylogarithm $\textup{QLi}_{n,k}.$ This implies Theorem \ref{MainTheoremQuadrangulationCluster}.

Similarly, in \S \ref{SectionAlternatingPolylogarithms}  we construct a  map $\textup{ALi}$  from $\F_\P$ to $\H,$ which we call alternating polylogarithm. The formal quadrangulation formula implies that the ``complicated part'' of  the coproduct of $\textup{ALi}(\textup{T}_\P)$ vanishes and the remaining part is identical to the coproduct of the class of an oriented orthoscheme $h(\textup{ort}(x^s))$ for $x^s\in \mathfrak{M}^s_{0,2n+2},$ which implies Theorem~\ref{MainTheoremQuadrangulationVolume}.

\subsection{Contents and Acknowledgements} 
In \S \ref{SectionHodgeMultiplePolylogarithmsHodgeTateStructures} we review some background material, including mixed Hodge-Tate structures, iterated integrals and polylogarithms. In \S \ref{SectionWorkingWithoutMHS} we describe how one can read the paper without familiarity with mixed Hodge structures.

In \S \ref{SectionHopfAlgebraFormal} we construct the Hopf algebra of formal polylogarithms over any variety $S.$ The most important result of this section is Proposition \ref{TheoremFormalSSR}, which we call the formal quasi-shuffle relation. 

In \S \ref{SectionFormalOnConfiguration} we define formal quadrangular polylogarithms, which are certain formal polylogarithms on the configuration space. The definition of formal quadrangular polylogarithms is based on the universal property of the Connes-Kreimer Hopf algebra, which we recall in \S \ref{Arborification}. We prove the formal quadrangulation formula (Theorem \ref{TheoremFormalQF}); our proof of Theorem \ref{TheoremFormalQF} is very close to the proof of Proposition \ref{TheoremFormalSSR} from \S \ref{SectionFormalSSR}. In particular, the principal coefficient map introduced in \S \ref{SectionProjectionToPrimitiveElements} plays a key role in the proof. 

In \S \ref{SectionQuadrangulaPolylogs} we define quadrangular polylogarithms and prove their main properties. The first one is the formula for a presentation of a quadrangular polylogarithm via multiple polylogarithms (Theorem \ref{TheoremQuadrangulationClusterPolylog}), which we deduce  from Theorem \ref{TheoremFormalQF}. The second one is the formula for a presentation of Hodge correlators via quadrangular polylogarithms (Proposition \ref{PropositionUniversality}). These two results imply Theorem \ref{MainTheoremDepth}.

Finally, in \S \ref{SecOrt} we discuss non-Euclidean orthoschemes.  \S\S \ref{SectionProjectiveSimplex}-\ref{SecOrOrt} contain our reflections  on the ideas of Coxeter in \cite{Cox36}. In \S \ref{SectionVolumeOrthoschemes} we prove the formula for volumes of orthoschemes (Theorem \ref{TheoremVolumeOrt}), which implies the Theorem of B{\"o}hm from \S \ref{SectionIntroductionVolumes}.

\

\paragraph{\bf{ Acknowledgements}}
 Special thanks are due to A. Matveiakin for developing the software, which allowed me to check most results in this paper. In particular, Theorem \ref{TheoremClusterPolylogarithmCoproduct} was discovered jointly with A. Matveiakin via a computer-assisted search.   I am grateful to V. Fock for pointing me to the Maslov index construction.  I am indebted to A. Goncharov for numerous discussions and explanations of the theory of mixed Tate motives. I thank A. Beilinson, S. Bloch, F. Brown and B. Farb for useful discussion and comments.  Also, I thank G. Paseman for his help on a draft of this article. Finally, I am very grateful to the referee who made a lot of useful comments and suggestions.

\section{Multiple polylogarithms and mixed Hodge-Tate structures}\label{SectionHodgeMultiplePolylogarithmsHodgeTateStructures}

\subsection{Multiple polylogarithms and iterated integrals}\label{SectionMultiplePolylogarithmsIteratedIntegrals}

From the most naive point of view, multiple polylogarithm is an analytic function, defined in the polydisc 
\[
|a_1|,|a_2|,\dots, |a_k| <1
\] 
by the power series
\[
\Li_{n_1,n_2,\dots, n_k}(a_1,a_2,\dots,a_k)=\sum_{m_1>m_2>\dots>m_k>0}\frac{a_1^{m_1} a_2^{m_2}\dots a_k^{m_k}}{m_1^{n_1}m_2^{n_2}\dots m_k^{n_k}}.
\]

The properties of multiple polylogarithms, such as shuffle and quasi-shuffle relations, were studied by Goncharov, see \cite[\S 2]{Gon01}. The key idea is the relation between multiple polylogarithms and iterated integrals, which we now recall. 

For a smooth complex manifold $M$ consider a collection of $1$-forms $\omega_1,\ldots,\omega_m$ on $M$ and a piecewise smooth path $\gamma\colon [0,1]\lra M.$ Denote by $\gamma^*\omega_i=f_i(t)d t$
the pull-back of the form $\omega_i$ to the segment $[0,1].$ The iterated integral of $\omega_1,\ldots,\omega_n$ along $\gamma$ is defined as follows: 
\[
\int_\gamma\omega_1\circ\dots\circ\omega_n\\
=\int_{0\leq t_1\leq t_2\leq \dots\leq t_n\leq 1}f_1(t_1)dt_1 \wedge \dots \wedge f_n(t_n)dt_n.
\]

One can show (see \cite[Theorem 2.1]{Gon01}) that multiple polylogarithms can be presented as iterated integrals:
\be \label{FormulaPolylogarithmsIteratedIntegrals}
\begin{split}
	&\Li_{n_1,n_2,\dots, n_k}(a_1,a_2,\dots,a_k)\\
	=&(-1)^k \int_0^1 \underbrace{\frac{dt}{t-(a_1\dots a_k)^{-1}} \circ \frac{dt}{t}\circ\dots \circ\frac{dt}{t}}_{n_1}\circ \dots \circ \underbrace{\frac{dt}{t-a_k^{-1}} \circ \frac{dt}{t}\circ\dots \circ\frac{dt}{t}}_{n_k}.
\end{split}
\ee
Formula (\ref{FormulaPolylogarithmsIteratedIntegrals}) allows defining multiple polylogarithms outside of the polydisc by analytic continuation. Moreover, it shows that multiple polylogarithms are {\it motivic periods}, see \cite{KZ01} and \cite{Bro17}. The theory of motivic periods is largely conjectural and may be viewed as an extension of the Galois theory to certain classes of transcendental numbers. The main advantage of this point of view is that one can study motivic periods via {\it motivic coaction}.

For our purposes, it will be convenient not to view polylogarithms  as motivic periods, but rather as {\it framed mixed Hodge-Tate structures}. Intuitively, this corresponds to working modulo an ideal of the ring of periods, generated by $2\pi i.$ In this setting, the motivic coaction is substituted with the motivic coproduct.

For a detailed exposition of the  background material on mixed Hodge structures and multiple polylogarithms, see \cite[\S 5-6]{Gon01}. In the next few sections, we briefly review the key definitions and facts that we will use later. Also, there is a way to read the paper without working with mixed Hodge structures, see \S \ref{SectionWorkingWithoutMHS}.

\subsection{Mixed Hodge-Tate structures}\label{SectionMHTS} Mixed Hodge structures were introduced in \cite{Del71}, \cite{Del71b}, and \cite{Del74}. Let $H$ be a finite-dimensional vector space over $\Q.$ A {\it pure Hodge structure} of weight $n\in \mathbb{Z}$ on $H$ is a  decreasing filtration $F^\bullet$ on the complexification $H_\C$
\[
\dots \supseteq  F^p H_{\C}\supseteq F^{p+1} H_{\C} \supseteq\dots
\]	
such that
\[
F^p H_\C \oplus \overline{F^q H_\C}=H_\C \text{ for any }p+q=n+1.
\]
A {\it mixed Hodge structure} on $H$ is a pair of an increasing filtration $W_\bullet$ on $H$ (the weight filtration) and a decreasing filtration $F^\bullet$ on $H_\C$ such that the  filtration, induced by $F^\bullet$ on every graded piece
\[
\textup{gr}_n^W (H)=\frac{W_nH}{W_{n-1}H},
\] 
is a pure Hodge structure of weight $n.$ Mixed Hodge structures form a category, which turns out to be abelian; moreover, it is a Tannakian category.

The unique $1$-dimensional pure Hodge structure $H$ of weight $2n$ with $F^{n}H=H$ and $F^{n+1}H=0$ is denoted by $\Q(-n)$ and is called {\it pure Tate}; we have $\Q(-1)=\Q(1)^{\vee}$ and $\Q(n)=\Q(1)^{\otimes n}.$ A mixed Hodge structure $H$ is called {\it mixed Hodge-Tate} if its graded  pieces  $\textup{gr}^W_{2n+1} (H)$ vanish and graded  pieces  $\textup{gr}^W_{2n} (H)$ are sums of pure Tate structures $\Q(-n).$ The category of mixed Hodge-Tate structures is a Tannakian subcategory of the category of mixed Hodge structures.

Deligne showed that singular cohomology groups of algebraic varieties carry a canonical mixed Hodge structure; so do relative cohomology groups of algebraic pairs.  For instance, both $H^{2n}(\mathbb{P}^n,\Q)$ and $H^{2n-1}\left(\mathbb{P}^{2n-1}\fgebackslash Q ,\Q\right)$  for a smooth quadric $Q$ in $\mathbb{P}^{2n-1}$ equal to $\Q(-n)$ and are pure Tate. Here is an example of a mixed Hodge-Tate structure, which is not pure.

\begin{example}
For $a\in \C^\times \fgebackslash \{1\}$ consider the cohomology group $H^1(\mathbb{P}^1 \fgebackslash \{0,\infty\},\{1,a\}).$ From the long exact sequence of pairs one can deduce that the corresponding mixed Hodge structure is mixed Hodge-Tate and defines an extension 
\be \label{FormulaExtensionMHS}
0\lra \mathbb{Q}(-1) \lra H^1\left (\mathbb{P}^1 \fgebackslash \{0,\infty\},\{1,a\}\right) \lra \mathbb{Q}(0)\lra 0.
\ee
One can show that all extensions of $\Q(0)$ by $\Q(-1)$ are obtained in this way. Moreover,
$\textup{Ext}^1(\mathbb{Q}(0), \mathbb{Q}(-1))=\mathbb{C}^\times_\Q$ and $a\in \mathbb{C}^\times_\Q$ corresponds to the extension (\ref{FormulaExtensionMHS}). Here we define the rationalization of an abelian group $A$ by  $A_\Q:=A\otimes \Q.$
\end{example}

A general family of examples of mixed Hodge-Tate structures comes from non-Euclidean geometry. Goncharov showed that for a projective simplex $S=(Q; H_1,\dots, H_{2n})$ in $\mathbb{P}^{2n-1}$  the following  cohomology group is mixed Hodge-Tate:
\[
H(S)=H^{2n-1}\left(\mathbb{P}^{2n-1}\fgebackslash Q, \left (\bigcup_{i=1}^{2n} H_i\right )\fgebackslash Q;\mathbb{Q}\right).
\]
The dimension of $H(S)$ depends on the relative position of the hyperplanes and the quadric. For instance, for a generic projective tetrahedron $S$ in $\mathbb{P}^3$ we have  
\[
\textup{gr}_0^W \left (H(S)\right)=\Q(0), \ \ \ \textup{gr}_2^W \left(H(S)\right)=\Q(-1)^{\oplus 6}, \ \ \ \textup{gr}_4^W \left(H(S)\right)=\Q(-2).
\]

The category of mixed Hodge-Tate structures is Tannakian, so by Tannakian duality, it can be described  via its Hopf algebra of framed objects, which we denote by $\H$ (see \cite{BMS87} and \cite{BGSV90}). Hopf algebra $\H$ is graded: $\H=\bigoplus_{n\geq 0}\H_n.$ The elements of $\H$ can be described explicitly as follows.

An $n$-framed mixed Hodge-Tate structure is a triple $[H; v, f],$ where $H$ is a  mixed Hodge-Tate structure, $v$ is a vector in $\textup{gr}_{2n}^W (H)$, and $f$ is a vector in $\left(\textup{gr}_0^W (H)\right)^\vee.$ Consider the coarsest equivalence relation on the set of all $n$-framed Tate structures for which $H_1 \sim  H_2$ if there is a morphism of mixed Hodge structures $H_1 \rightarrow H_2$ respecting the frames. Then $\H_n$ is the set of equivalence classes; $\H_0=\Q$. 

Addition in $\H_n$ is defined by the following rule:
\[
[H_1;v_1,f_1]+[H_2;v_2,f_2]=[H_1\oplus H_2; v_1\oplus v_2, f_1+f_2].
\]
The tensor product of mixed Hodge structures induces a commutative product. Next, the $(n-k,k)$-component of the coproduct
$
\Delta^{\H\H}\colon{\H_n\lra \bigoplus_{k=0}^n \H_{n-k}\otimes \H_{k}}
$ 
is defined by formula \[
\Delta^{\H\H}_{n-k,k}[H;v,f]=\sum_{i}[H\otimes \Q(k);v,e^i]\otimes [H;e_i,f]
\]
for a basis $(e_i)$ of $\textup{gr}_{2k}^W(H)$ and the dual basis $(e^i)$ of $\left(\textup{gr}_{2k}^W(H)\right)^\vee$.

\begin{example}
The Hodge logarithm $\log^\H(a)\in \H_1$ is the triple $[H,v,f]$, where 
\[
H=H^1\left (\mathbb{P}^1 \fgebackslash \{0,\infty\},\{1,a\}\right),
\] 
vector $v$ is the generator of  
$\textup{gr}_2^W(H)=H^1(\mathbb{P}^1 \fgebackslash \{0,\infty\})$ represented by the form $\dfrac{dz}{z}$, and  vector $f$ is the generator of $\left(\textup{gr}_0^W(H)\right)^{\vee}=H_1(\mathbb{P}^1,\{1,a\})$ represented by a
 path from $1$ to $a.$ Hodge logarithms generate $\H_1=\mathbb{C}^\times_\mathbb{Q}.$ It is easy to see that Hodge logarithms are primitive elements for the coproduct $\Delta^{\H\H}:$
 \[
 \Delta^{\H\H} \log^\H(a) =\log^\H(a)\otimes 1+ 1\otimes \log^\H(a).
 \] 
 \end{example}

	The usual logarithm	 $\log(a)$ is defined only up to a multiple of $2\pi i.$ Notice that Hodge logarithm $\log^\H(a)$ is defined canonically. In particular, there does not exist an evaluation morphism from $\H$ to $\C.$ Remarkably, there exists a canonical morphism 
\be \label{FormulaRealPeriod}
\mathrm{per}_{\mathbb{R}}\colon \H_n \lra \mathbb{R}(n-1),
\ee
where $\mathbb{R}(n-1)=\mathbb{R}(2\pi i)^{n-1}\subseteq \C.$ It was defined by Goncharov in \cite[\S 4]{Gon99} and is called ``real period.'' The construction of the real period is explicit, but rather  tricky, we do not repeat it here (see \cite[\S 3]{GZ18} for a more conceptual explanation).  For $n=1$ we have 
\[
\mathrm{per}_{\mathbb{R}}(\log^\H(a))=\log(|a|)\in \mathbb{R}.
\]
\begin{example}
Consider an nondegenerate oriented projective tetrahedron $S=(Q; H_1,\dots, H_{4})$ in $\mathbb{P}^{3}$ defined in \S \ref{SectionProjectiveSimplex}; let $q=0$ be an equation of the quadric $Q$. We associate to $S$ a framed mixed Hodge structure 
\[
h(S)=[H(S),v,f]\in \H_2,
\] 
where $H(S)=H^3\left(\mathbb{P}^{3}\fgebackslash Q, \left (\bigcup_{i=1}^{4} H_i\right )\fgebackslash Q\right ),$ $v$ is a generator of  $H^3\left(\mathbb{P}^3\fgebackslash Q\right),$
and $f$ is a generator of $H_3\left(\mathbb{P}^3, \bigcup H_i \right).$ The coproduct $\Delta^{\H\H} h(S)\in \H \otimes \H$ is a version of the Dehn invariant.

Notice that both $H_3\left(\mathbb{P}^3, \bigcup H_i \right)=\Q(0)$ and  $H^3(\mathbb{P}^3\fgebackslash Q)=\Q(-2)$ have two natural generators. The choice of $f$ is fixed by the ordering of the hyperplanes $H_1,\dots,H_4$ and the choice  of  
\[
v=\pm\sqrt{\textup{disc}(q)}\frac{\sum_{i=0}^3(-1)^ix_idx_1\wedge \dots \wedge \widehat{dx_i}\wedge \ldots \wedge dx_4}{q^2}
\]
depends on the orientation of $S,$ see \cite[\S 3.3]{Gon99} for the details.
If the projective tetrahedron $S$ comes from a hyperbolic simplex $\mathcal{S}$ in $\mathbb{H}^{3}$ then the real period of $h(S)$ equals to the hyperbolic volume of $\mathcal{S}.$
\end{example}

We will use the same symbol $\H$ for the completion of $\H$ with respect to the $\mathcal{I}$-adic topology, where $\mathcal{I}$ is the fundamental ideal $\bigoplus_{n>0} \H_n.$  The Lie coalgebra of indecomposable elements in $\H$ is denoted by $\L:$
\[
\L=\H_{>0}/(\H_{>0}\cdot\H_{>0}).
\] 
For $A, B\in \H_{>0}$ we say that $A$ and $B$ {\it equal modulo products} if their projections coincide in $\L;$ we write $A\stackrel{\shuffle}{=}B.$ We denote the Lie cobracket in $\L$ by $\Delta^\L.$ 

Sometimes we work not with  framed mixed Hodge-Tate structures but with framed unipotent variations of mixed Hodge-Tate structures over an algebraic variety $S$; we denote the corresponding Hopf algebra by $\H[S]$. Not much changes in this setting, see \cite[\S5.2]{Gon01} and \cite[\S7]{Bro17} for more details.

\subsection{Hodge iterated integrals, correlators and multiple polylogarithms}\label{SectionHodgeIterated}
Goncharov constructed elements of $\H$ called Hodge iterated integrals in \cite{Gon01}, \cite{Gon02} and \cite{Gon05}. For every collection of points $x_0,\dots, x_{n+1} \in \mathbb{C}$ Goncharov constructed {\it Hodge iterated integrals}
\be \label{FormulaHodgeIterated}
\textup{I}^{\H}(x_0;x_1,x_2,\dots,x_n;x_{n+1})\in \H_n,
\ee
which are Hodge versions of iterated integrals
\[
\int_{x_0}^{x_{n+1}}\frac{dt}{t-x_1}\circ \dots \circ \frac{dt}{t-x_n}=\int_{x_0\leq t_1\leq t_2\leq \dots\leq t_n\leq x_{n+1}}\frac{dt_1}{t_1-x_1}\wedge  \frac{dt_2}{t_2-x_2}\wedge\dots \wedge\frac{dt_n}{t_n-x_n}.
\]
Note that unlike usual iterated integrals, Hodge iterated integrals do not depend on the paths from $x_0$ to $x_{n+1}.$  We give a brief outline of their construction for $x_i\notin \{x_0,x_{n+1}\}.$ For the general case one needs to work with tangential base points, see \cite[\S 5.3]{Gon01}.

Let 
\[
\pi^{un}=\pi^{un}(\mathbb{C}\fgebackslash \{x_1,\ldots,x_n\}; x_0, x_{n+1})
\]
be the prounipotent completion to the topological torsor of paths from $x_0$ to $x_{n+1}$. In \cite{Hai87} Hain showed that its algebra of functions $\mathcal{O}(\pi^{un})$ is equipped with the structure of a projective limit of mixed Hodge-Tate structures, see also \cite[\S4]{Gon01}. In particular, it has a weight filtration, such that 
\[
\textup{gr}_0^W\left(\mathcal{O}(\pi^{un})\right)=\Q(0) \text{ and } \textup{gr}_{2n}^W\left(\mathcal{O}(\pi^{un})\right)=H^1(\C\fgebackslash \{x_1,\ldots,x_n\})^{\otimes n}.
\] 
The Hodge iterated integral (\ref{FormulaHodgeIterated}) is a framed mixed Hodge-Tate structure 
\[
\left[\mathcal{O}(\pi^{un});v,f\right]
\] for the obvious $f\in \left (\textup{gr}_0^W\right)^{\vee}$ and  
\[
v=\frac{dz}{z-x_1}\otimes \dots \otimes \frac{dz}{z-x_n}\in\textup{gr}_{2n}^W\left(\mathcal{O}(\pi^{un})\right).
\]

\begin{example}\label{ExampleLog}
We have  $I^\H(x_0;x_1)=1\in \H_0$ and 
\[
I^\H(x_0;x_1;x_2)=\log^\H(x_2-x_1)-\log^\H(x_1-x_0)\in \H_1.
\]	
Next, we define the Hodge dilogarithm:
\[
\textup{Li}_2^\H(a)=-I^\H(0;1,0;a)\in \H_2.
\]
\end{example}

\begin{figure}
\begin{center}
\begin{tikzpicture}[transform shape]
  \foreach \number in {1,...,5}{
        \mycount=\number
        \advance\mycount by -1
  \multiply\mycount by 45
        \advance\mycount by 0
      \node[draw, very thin, color=gray,inner sep=0.001cm] (N-\number) at (\the\mycount:3.4cm) {};
    }
  \foreach \number in {9,10,11,12}{
        \mycount=\number
        \advance\mycount by -1
  \multiply\mycount by 45
        \advance\mycount by 22.5
      \node[draw, very thin, color=gray,inner sep=0.001cm](N-\number) at (\the\mycount:3.4cm) {};
    }
\draw[midnight, line width=0.25mm] (N-1) -- (N-9);
\draw[midnight, line width=0.25mm] (N-9) -- (N-2);
\draw[midnight, line width=0.25mm] (N-2) -- (N-10);
\draw[midnight, line width=0.25mm] (N-10) -- (N-3);
\draw[midnight, line width=0.25mm] (N-3) -- (N-11);
\draw[midnight, line width=0.25mm] (N-11) -- (N-4);
\draw[midnight, line width=0.25mm] (N-4) -- (N-12);
\draw[midnight, line width=0.25mm] (N-12) -- (N-5);
\draw[midnight, line width=0.25mm] (N-5) -- (N-1);

\node[left] at (N-5) {$x_0$};
\node[left] at (N-12) {$x_1$};
\node[yshift=0.11cm, xshift=-0.31cm] at (N-4) {$x_2$};
\node[yshift=0.21cm, xshift=-0.21cm] at (N-11) {$x_3$};
\node[above] at (N-3) {$x_4$};
\node[yshift=0.21cm, xshift=0.21cm] at (N-10) {$x_5$};
\node[yshift=0.11cm, xshift=0.31cm] at (N-2) {$x_6$};
\node[right] at (N-9) {$x_7$};
\node[right] at (N-1) {$x_8$};

\draw[midnight, line width=0.25mm] (N-5) -- (N-11);
\draw[midnight, line width=0.25mm] (N-3) -- (N-2);
\draw[midnight, line width=0.25mm] (N-2) -- (N-1);
\end{tikzpicture}
\end{center}
\caption{The term  corresponding to the sequence $(0,3,4,6,8)$  is equal to 
 $\left( \textup{I}^{\H}(x_0;x_1,x_2;x_3)\textup{I}^{\H}(x_4;x_5;x_6)\textup{I}^{\H}(x_6;x_7;x_8)\right)\otimes \textup{I}^{\H}(x_0;x_3,x_4,x_6;x_8).$  
 }
\label{FigureCoproduct}
\end{figure}
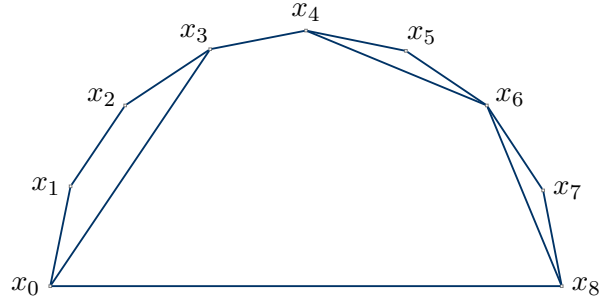

In \cite[Theorem 3.3(c)]{Gon05} Goncharov computed the coproduct of Hodge iterated integrals:
\be
\label{FormulaCoproduct}
\begin{split}
	&\Delta^{\H\H} \textup{I}^{\H}(x_0;x_1,\dots,x_n;x_{n+1})\\
	=&\sum_{(i_0,i_1,\dots,i_{r+1})}\left (\prod_{p=0}^{r} \textup{I}^{\H}(x_{i_p};x_{i_p+1},\dots,x_{i_{p+1}-1};x_{i_{p+1}})\right )\otimes  \textup{I}^{\H}(x_{i_0};x_{i_1},\dots, x_{i_r} ; x_{i_{r+1}}),
\end{split}
\ee
where the summation goes over all sequences $0=i_0<i_1<\dots<i_r<i_{r+1}=n+1$ (see Figure \ref{FigureCoproduct}). 

We will use the following property of Hodge iterated integrals: if $x_0\neq x_1$ and $n>1$ then for any $\lambda \in \C^{\times}$ we have
\be \label{FormulaHomogenuity}
\textup{I}^{\H}(x_0;x_1,x_2,\dots,x_n;x_{n+1})=\textup{I}^{\H}( \lambda x_0;\lambda x_1,\lambda x_2,\dots,\lambda x_n;\lambda x_{n+1}).
\ee
For $x_0=x_1$ this is true only modulo products. In general, the following property holds. For $x_0\neq x_1$ we put
\[
\begin{split}
	&\textup{I}^{\H}_\bullet(x_0;x_1,x_2,\dots,x_n;x_{n+1})
	=\sum_{n\geq 0}\textup{I}^{\H}_\bullet(x_0;\underbrace{x_0,\ldots,x_0}_n,x_1,x_2,\dots,x_n;x_{n+1})\in \H.
\end{split}
\]
It is well-known that for $n\geq  0$ and $\lambda\neq 0$ we have
\be \label{FormulaLogPower}
 \textup{I}^\H(0;\underbrace{0,\dots,0}_{n};\lambda)=\frac{\left (\log^\H(\lambda)\right )^{n}}{n!},
\ee
and so
$
\textup{I}^\H_\bullet(0;\lambda)=e^{\log^\H(\lambda)}.
$
Then the following formula generalizes (\ref{FormulaHomogenuity}):
\be \label{FormulaHomogenuityDivergent}
\textup{I}^{\H}_\bullet(\lambda x_0;\lambda x_1,\lambda x_2,\dots,\lambda x_n;\lambda x_{n+1})=
e^{\log^\H(\lambda)}\textup{I}^{\H}_\bullet(x_0;x_1,x_2,\dots,x_n;x_{n+1});
\ee
we leave the proof to the reader.

Another important fact is that if $x_0=x_{n+1}$ then
\[
\textup{I}^{\H}_\bullet(x_0;x_1,x_2,\dots,x_n;x_{n+1})=0.
\]
Finally, in the proof of Proposition \ref{LemmaAlternatingCoproduct}  we will use the following statement, which can be easily checked:
\be \label{EqualityRoots}
\sum_{\epsilon_{1},\dots, \epsilon_{n}\in \{-1,1\} }  \left ( \prod_{i=1}^k \epsilon_i\right )
\textup{I}^{\H}(x_0;x_1,\dots,x_n;x_{n+1})=\textup{I}^{\H}(x_0^2;x_1^2,\dots,x_n^2;x_{n+1}^2).
\ee

In \cite{Gon19} Goncharov constructed elements in $\mathcal{L}$
called Hodge correlators, see also \cite[\S 2]{GR18}. Consider $a\in \C$ and $x_0,\dots, x_{n}\in \mathbb{C}\fgebackslash\{a\}.$ Hodge correlators are certain elements
$\textup{Cor}_a(x_0,\dots,x_n) \in \mathcal{L}_n;$ there construction is based on a natural description of the prounipotent completion of the fundamental group $\pi^{un}(\mathbb{P}^1 \fgebackslash \{x_1,\ldots,x_n\},a)$, see \cite[\S 2]{GR18}. For $a=\infty$ we put
\[
\textup{Cor}(x_0,\dots,x_n)=\textup{Cor}_{\infty}(x_0,\dots,x_n).
\]
There is a simple formula reducing the correlator with an arbitrary $a$ to the correlator with $a=\infty:$
\be \label{CorrelatorChangeBasepoint}
\textup{Cor}_{a}(x_0,\dots,x_n)=
\sum_{s=0}^{n} \sum_{0\leq i_1<\dots<i_s\leq n}(-1)^s \textup{Cor}(x_0,\dots,a,\dots,a,\dots,x_{n}).
\ee
Here the $s$-$th$ term is $(-1)^s$ times a sum over $0\leq i_1<\dots<i_s\leq n$ of correlators obtained from the correlator $\textup{Cor}(x_0,\dots,x_{n})
$ by inserting the point $a$ instead of points $x_{i_1} , \dots , x_{i_s}.$

Denote the projection of $I^\H\in \H$ to $\L$ by $I^\L.$  Hodge correlators and Hodge iterated integrals are related by a simple formula
\be \label{FormulaIIviaCor}
(-1)^{n+1}\textup{I}^{\mathcal{L}}(x_0;x_1,x_2,\dots,x_n;x_{n+1})=\textup{Cor}(x_1,\dots, x_{n+1})-\textup{Cor}(x_0,\dots, x_n).
\ee
It is easy to see that $\textup{Cor}(x_0,\dots, x_n)=0$ for $x_0=x_1=\dots=x_n.$
Formula (\ref{FormulaIIviaCor}) implies that 
\[
\textup{Cor}(x_0,\dots, x_n)=(-1)^{n+1}\sum_{i=1}^{n}\textup{I}^{\mathcal{L}}(x_0;\underbrace{x_0,\ldots,x_0}_i,x_1,\dots,x_{n-i};x_{n-i+1}).
\]

\begin{example}
\[
\begin{split}
&\textup{Cor}(x_0,x_1)=\log^{\mathcal{L}}(x_0-x_1);\\
&\textup{Cor}(x_0,\underbrace{x_1,x_1,\dots,x_1}_{n})=0 \text{\ for \ } n\geq 2;\\	
&\textup{Cor}_a(x_0,x_1,x_2)=\Li^{\mathcal{L}}_2([a,x_0,x_1,x_2]).\\	
\end{split}
\]
\end{example}
Hodge correlators are cyclically symmetric, i.e.,
\be \label{FormulaCorrelatorsSymmetry}
\textup{Cor}_a(x_0,\dots, x_{n-1}, x_n)=\textup{Cor}_a(x_1,\dots, x_n, x_0).
\ee
Their coproduct can be computed by the following formula:
\be\label{FormulaCoproductCorrelators}
\Delta^{\mathcal{L}}\textup{Cor}_a(x_0,\dots, x_n)=\sum_{i<j}  \textup{Cor}_a(x_j,\dots,x_{i-1})\wedge \textup{Cor}_a(x_i,\dots,x_j), 
\ee
where the notation ``$i < j$'' means that the order of the points $x_i,...,x_j$ in the first factor, and hence in the second, is compatible with the cyclic order of the points $x_k$. We say that the term $\textup{Cor}_a(x_j,\dots,x_{i-1}) \wedge \textup{Cor}_a(x_i,\dots,x_j)$ comes from the ``cut'' beginning at $x_j$ and ending between $x_{i-1}$ and $x_{i},$ see Figure \ref{FigureCoproductCut}.
 
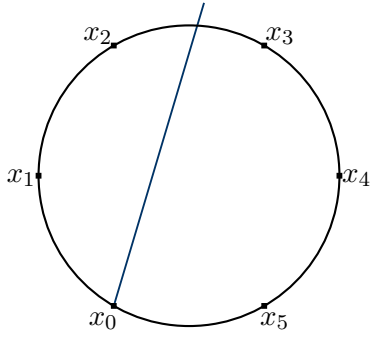
\begin{figure}
 \centering
\begin{tikzpicture}[transform shape]
  \foreach \number in {1,2,3}{
        \mycount=\number
        \advance\mycount by -1
  \multiply\mycount by 120
        \advance\mycount by 60
      \node[draw, very thick, color=black, inner sep=0.01cm] (N-\number) at (\the\mycount:2.0cm) {};
    }
  \foreach \number in {4,5,6}{
        \mycount=\number
        \advance\mycount by -1
  \multiply\mycount by 120
        \advance\mycount by 120
      \node[draw, very thick, color=black, inner sep=0.01cm](N-\number) at (\the\mycount:2cm) {};
    }

\draw[midnight, line width=0.25mm] (N-5) -- (0.2,2.3);

\node[yshift=-0.2cm, xshift=-0.14cm] at (N-5) {$x_0$};
\node[yshift=-0.03 cm, xshift=-0.23cm] at (N-2) {$x_1$};
\node[yshift=0.13cm, xshift=-0.21cm] at (N-4) {$x_2$};
\node[yshift=0.13cm, xshift=0.21cm] at (N-1) {$x_3$};
\node[yshift=-0.03 cm, xshift=0.23cm] at (N-6) {$x_4$};
\node[yshift=-0.2cm, xshift=0.14cm] at (N-3) {$x_5$};
\draw[black, thick] (0,0) circle (2 cm);

\end{tikzpicture}
\caption{The term $\textup{Cor}(x_0,x_3,x_4,x_5) \wedge \textup{Cor}(x_0,x_1,x_2)$  in the coproduct $\Delta^\L\left(\textup{Cor}(x_0,x_1,x_2, x_3,x_4,x_5)\right)$ comes from the cut beginning at $x_0$ and ending between $x_2$ and $x_3.$}
\label{FigureCoproductCut}
\end{figure}

For $n_1,\dots,n_k\in \mathbb{N}$ we define the multiple polylogarithm
\be \label{FormulaMPviaII}
\begin{split}
	&\Li^{\H}_{n_1,\dots,n_k}(a_1,a_2,\dots,a_k)\\
&=(-1)^{k}\textup{I}^{\H}(0;1,\underbrace{0,\dots,0,a_1}_{n_1},\dots,\underbrace{0,\dots,0,a_1a_2\dots a_{k-1}}_{n_{k-1}},\underbrace{0,\dots,0;a_1a_2\dots a_{k}}_{n_k});
\end{split}
\ee
this definition is motivated by (\ref{FormulaPolylogarithmsIteratedIntegrals}).
We will need a slight generalization: for $n_1,\dots,n_k\in \mathbb{N}$ and $n_0\geq 0$ we  define the (generalized) multiple polylogarithm
\[
\begin{split}
	&\Li^{\H}_{n_0;n_1,\dots,n_k}(a_1,a_2,\dots,a_k)\\
&=(-1)^{k}\textup{I}^{\H}(0;\underbrace{0,\dots,0}_{n_0},1,\underbrace{0,\dots,0,a_1}_{n_1},\dots,\underbrace{0,\dots,0,a_1a_2\dots a_{k-1}}_{n_{k-1}},\underbrace{0,\dots,0;a_1a_2\dots a_{k}}_{n_k}).
\end{split}
\]
For $n_0>0$ the corresponding iterated integral is divergent, but can be regularized. We will only work with their image in $\L.$

The number $n=n_0+n_1+n_2+\dots+n_k$ is called the weight of the multiple polylogarithm and  $k$ is called its depth. We have $\Li^{\H}_1\left(1-\frac{1}{x}\right)=\log^{\H}(x),$ see Example \ref{ExampleLog}. It is often convenient to consider the sum 
\[
\Li^{\H}_{\bullet;n_1,\dots,n_k}(a_1,a_2,\dots,a_k)=\sum_{n_0=0}^{\infty} \Li^{\H}_{n_0;n_1,\dots,n_k}(a_1,a_2,\dots,a_k)\in \H.
\]

Using (\ref{FormulaCoproduct}) one can compute the coproduct of multiple polylogarithms. 
\begin{example} \label{ExampleCoproduct}
 In weight two and depth one we have
\[
\Delta^{\H\H}\Li^{\H}_{2}(a)=-\Delta^{\H\H}\textup{I}^{\H}(0;1,0;a)=1 \otimes\Li^{\H}_{2}(a)+\Li^{\H}_{2}(a)\otimes 1+\log^{\H}(a)\otimes \Li^{\H}_1(a).
\]
In weight two and depth two we have
\begin{align*}
\Delta^{\H\H}\Li^{\H}_{1,1}(a_1,a_2)&=\Delta^{\H\H}\textup{I}^{\H}(0;1,a_1;a_1a_2)\\
  &=1\otimes\Li^{\H}_{1,1}(a_1,a_2)+  \Li^{\H}_1(a_1)\otimes \Li^{\H}_1(a_2)+\Li^{\H}_{1,1}(a_1,a_2)\otimes 1\\
  &\quad+(\Li^{\H}_1(a_2)-\Li^{\H}_1(a_1^{-1}))\otimes \Li^{\H}_1(a_1a_2).\\
\end{align*}
\end{example}

\subsection{Proving identities between periods} \label{SectionRigidity}
Some of the main results of our work, i.e., Theorems \ref{MainTheoremDepth} and \ref{MainTheoremQuadrangulationCluster}, are certain identities between multiple polylogarithms ``modulo products of polylogarithms of lower weight.'' Polylogarithms are multivalued functions, which are defined only up to a product of $2\pi i$ and a function of lower weight, so writing explicit ``numerically checkable'' identities between multiple polylogarithms is possible only locally. 

Instead of working with actual functions, we work with framed unipotent variations of mixed Hodge-Tate structures: elements of the Hopf algebra $\H[S]$. Given an element $[H;v,f]\in \H[S]$, one can recover the multivalued function by choosing a (local) splitting of the weight filtration $W_\bullet H_\Q$ and  taking a flat section of the corresponding variation. For details, see the definition of the {\it complex period} in \cite[\S4.1]{Gon99}. The resulting function is defined up to a product of $2\pi i$ and a complex period of a framed variation of mixed Hodge-Tate structure of lower weight. 

So, every identity between elements of $\H[S]$ implies an identity between the corresponding multivalued functions. Similarly, every identity between  elements of $\L[S]$ leads to an identity between the corresponding multivalued functions modulo product of functions of lower weight.

The general method of proving such identities is discussed in \cite[\S 8]{Gon02}; here we sketch the main idea. If we need to check an identity $A=B$ between actual functions, we can proceed in two steps. First, we prove that the differentials of both sides coincide ($d A =d B$), thus showing that the identity is valid up to a constant. Then we specialize the functions to a point and check that the constant $A-B$ is zero.	

A similar method works for an identity $A=B$ between elements of $\H_n.$ In most examples that we encounter, terms in $A$ and $B$ are Hodge iterated integrals (\ref{FormulaHodgeIterated}) with arguments being rational functions on a smooth connected algebraic variety $S.$ It means that equality $A=B$ is equality of framed unipotent variations of  mixed Hodge-Tate structures on $S.$ To check that $A=B$, we first check equality of coproducts $\Delta^{\H\H}A=\Delta^{\H\H}B$. Note that this step is parallel to checking the equality of differentials of functions, as the $(1,n-1)$-component of the coproduct is closely related to the differential.

By the Rigidity Lemma \cite[Lemma 8.2]{Gon02}, vanishing of the coproduct $\Delta^{\H\H}(A-B)$ implies that $A-B$ is a framed constant variation on $S.$ Thus, to prove  the equality $A=B$ it suffices to prove it for any specialization to a point in $S.$ Specialization of framed variations of mixed Hodge structures is discussed in \cite[\S 4]{Gon02}. Notice that one can  prove identities in the Lie coalgebra $\L$ in a similar way, see \cite[Lemma 2.7]{GR18}.

\subsection{Working with polylogarithms without knowledge of mixed Hodge structures} \label{SectionWorkingWithoutMHS}

In \S\S \ref{SectionMHTS}--\ref{SectionRigidity} we gave the background material on mixed Hodge-Tate structures and their relation to polylogarithms. The notion of a mixed Hodge structure is neither intuitive nor an easy one to learn; in this section, we sketch an alternative way to read the paper. We construct a Hopf algebra of polylogarithms in a way similar to the formal definition of higher Bloch groups, see \cite[\S 1.9]{Gon95B}.

Let $\mathrm{F}$ be a field of zero characteristic. Our goal is to define a commutative graded Hopf algebra $\mathbb{G}_\bullet(\mathrm{F})$ of multiple polylogarithms. Conjecturally, $\mathbb{G}_\bullet(\mathrm{F})$ is isomorphic to the Hopf algebra $\H_{\mathcal{M}}$ from \S \ref{SectionIntroduction_Philosophy}.

 First, we define a  commutative Hopf algebra $\mathbb{I}_\bullet(\mathrm{F})$ generated by symbols 
\be \label{FormulaIter}
\mathbb{I}(x_0;x_1,\dots,x_n;x_{n+1})\in \mathbb{I}_n(\mathrm{F})
\ee
for $x_0,x_1,\dots,x_n, x_{n+1}\in \mathrm{F}$ subject to the following relations:
\begin{enumerate}
\item $\mathbb{I}(x_0;x_1)=1;$
\item $\mathbb{I}(x_0;x_1,\dots,x_n;x_{n+1})=0$ for $x_0=x_{m+1};$
\item $\mathbb{I}(x_0;x_1,\dots,x_n;x_{n+1})=\sum_{r=0}^n\mathbb{I}(x_0;x_1,\dots,x_r;x)\mathbb{I}(x;x_{r+1},\dots,x_n;x_{n+1}).$
\end{enumerate}

The cobracket 
 \[
\Delta\colon \mathbb{I}_\bullet(\mathrm{F})\lra\mathbb{I}_\bullet(\mathrm{F})\otimes \mathbb{I}_\bullet(\mathrm{F})
\]
is defined by the formula
\[
\begin{split}
	&\Delta \mathbb{I}(x_0;x_1,\dots,x_n;x_{n+1})\\
	=&\sum_{(i_0,i_1,\dots,i_{r+1})}\left (\prod_{p=0}^{r} \mathbb{I}(x_{i_p};x_{i_p+1},\dots,x_{i_{p+1}-1};x_{i_{p+1}})\right )\otimes  \mathbb{I}(x_{i_0};x_{i_1},\dots, x_{i_r} ; x_{i_{r+1}}),
\end{split}
\]
where the summation goes over all sequences $0=i_0<i_1<\dots<i_r<i_{r+1}=n+1.$ 
It is easy to see that  $\mathbb{I}_\bullet(\mathrm{F})$ is a commutative Hopf algebra, see \cite[Proposition 2.2]{Gon05}.

Next, we define the space of relations $R_n(\mathrm{F})\subseteq \mathbb{I}_n(\mathrm{F})$ inductively; the coalgebra of multiple polylogarithms is defined as the quotient
\[
\mathbb{G}_n(\mathrm{F})=\frac{\mathbb{I}_n(\mathrm{F})}{R_n(\mathrm{F})}.
\]

In weight one we define $R_1(\mathrm{F})$ to be the kernel of the map sending $\mathbb{I}(x_0; x_1;x_2)\in \mathbb{G}_1(\mathrm{F})$ to $\dfrac{x_2-x_0}{x_0-x_1}\in \mathrm{F}^\times_{\Q}:=\mathrm{F}$ for distinct $x_0,x_1,x_2$ and to zero otherwise.  By definition, $\mathbb{G}_1(\mathrm{F})$. Next, assume that spaces $R_i(\mathrm{F})$ are defined for $i<n.$ The space $R_n(\mathrm{F})$ is spanned by elements
\be \label{FormulaEquationsHomotopy}
r(0)-r(1)\in \mathbb{I}_n(\mathrm{F})
\ee
for elements $r(t) \in \mathbb{I}_n(\mathrm{F}(t)),$ which are primitive in $\mathbb{G}_\bullet(\mathrm{F}(t))\otimes \mathbb{G}_\bullet(\mathrm{F}(t)).$  Proving that this construction is well-defined requires some work; we will develop this approach elsewhere. 

Elements (\ref{FormulaIter}) are analogs of Hodge iterated integrals. More precisely, one can show that there exists a homomorphism of Hopf algebras from $\mathbb{G}(\C)$ to $\H$ sending $\mathbb{I}$ to $\textup{I}^\H.$  We can easily define elements of $\mathbb{G}_\bullet(\C)$ corresponding to Hodge correlators and Hodge multiple polylogarithms using  formulas like (\ref{FormulaIIviaCor}) and (\ref{FormulaMPviaII}). Theorem \ref{MainTheoremQuadrangulationCluster} can be proven for these elements as the rigidity argument can be also performed in $\mathbb{G}_\bullet(\C)$ thanks to (\ref{FormulaEquationsHomotopy}).

\section{Hopf algebra of formal polylogarithms}\label{SectionHopfAlgebraFormal}

\subsection{Quasi-shuffle algebra and quasi-shuffle relation} \label{SectionQuasiShuffle}

Recall the definition of a quasi-shuffle algebra from \cite{Hof00}, \cite{HI17}. Consider a set  of letters $\mathcal{A}$ which carries a structure of a commutative semigroup. We denote the product of $a, b\in \mathcal{A}$ by $a\cdot b\in \mathcal{A}.$ Next, consider a set of words  $\mathcal{A}^*$ over the alphabet $\mathcal{A}$. For a word $\omega \in\mathcal{A}^*$ we denote by $l(\omega)$ the number of letters in $\omega,$ called its length. The only word of length $0$ is the empty word $1\in \mathcal{A}^*.$   For two words $\omega_1, \omega_2\in\mathcal{A}^*$ we denote their concatenation by $\omega_1\omega_2\in\mathcal{A}^*.$ Next, for a word $\omega=a_1a_2\dots a_n\in \mathcal{A}^*$ and $a\in \mathcal{A}$ we denote the word $(a\cdot a_1)a_2\dots a_n$ by $a\cdot w$; we have  $a\cdot 1=0.$

Let $A$ be a $\Q$-vector space with a basis given by $\mathcal{A}^*.$ The tensor algebra on $A$ carries a commutative Hopf algebra structure with  quasi-shuffle product and the deconcatenation coproduct which we denote by $\textup{QSh}^\mathcal{A}$ and call a free quasi-shuffle algebra on $\mathcal{A}.$ We use the bar notation for its elements, i.e., we denote $v_1\otimes v_2\otimes \dots\otimes v_n$ by $[v_1|v_2|\dots|v_n].$ Next, we use the same notation for a word $\omega=a_1a_2\dots a_n\in \mathcal{A}^*$ and an element  $[a_1|\dots|a_n]$ of the quasi-shuffle algebra; such elements form a basis of $\textup{QSh}^\mathcal{A}.$ The quasi-shuffle product of the basis elements $\omega_1, \omega_2 \in \textup{QSh}^\mathcal{A}$ is defined recursively: 
\[
\begin{split}
    & \omega \star 1=1 \star \omega=\omega,\\
	&(a_1\omega_1)\star (a_2\omega_2)=a_1 (\omega_1\star (a_2\omega_2))+a_2((a_1\omega_1)\star (\omega_2))+(a_1\cdot a_2)(\omega_1\star \omega_2).\\
\end{split}
\]
The deconcatenation coproduct is defined by the formula
\be \label{FormulaCoproductQS}
\Delta(\omega)=\sum_{i=0}^n a_1 \dots a_i \otimes a_{i+1}\dots a_n.
\ee

As an abstract Hopf algebra, $\textup{QSh}^\mathcal{A}$ does not depend on the semigroup structure on $\mathcal{A}.$ In particular, if we take the product on $\mathcal{A}$ to be zero, we obtain the free shuffle algebra. Also, we will use the fact that the space of primitive elements in  $\textup{QSh}^\mathcal{A}$ is spanned by the words of length 1.
 
Here is a particular example which we will use. Let $S$ be an irreducible algebraic variety over $\C.$ Consider the set of letters $(\varphi,n)$ where $\varphi\in\C(S)^{\times}$ is a nonzero rational function on $S$ and $n$ is a positive integer. The product 
\[
(\varphi_1,n_1)\cdot (\varphi_2,n_2)=(\varphi_1\varphi_2,n_1+n_2)
\]
 defines a semigroup structure on the alphabet $(\varphi,n)$. Let $\mathcal{F}_S$ be the corresponding quasi-shuffle algebra with deconcatenation coproduct $\Delta^{\mathcal{F}\mathcal{F}}$. 

Let $\H_S$ be the inductive limit of Hopf algebras $\H[U]$ of framed unipotent variations of mixed Hodge-Tate structures over all nonempty open subset $U\subseteq S.$  Consider a map $\Li^{\H}\colon \mathcal{F}_S\lra \H_S$ defined by the rule
\[
\Li^{\H}([\varphi_1,n_1|\varphi_2,n_2|\dots|\varphi_k,n_k])=
\Li^{\H}_{n_1,n_2,\dots,n_k}(\varphi_1,\varphi_2,\dots,\varphi_k).
\]
This map is very far from being a homomorphism of coalgebras but remarkably is a homomorphism of algebras. This statement is known as the quasi-shuffle relation for multiple polylogarithms \cite[Theorem 1.2]{Gon02}, see also \cite[\S 2.5]{Gon01} for an analytic version (where it is called the first shuffle relation). 

\begin{example} For $S=\mathbb{A}^2$ the  equality 
\[
 [a_1,n_1|a_2,n_2]+ [a_2,n_2|a_1,n_1]+[a_1a_2,n_1+n_2]= [a_1,n_1]\star[a_2,n_2]
\]
implies that
\[
\Li_{n_1,n_2}^{\H}(a_1,a_2)+\Li_{n_2,n_1}^{\H}(a_2,a_1)+
\Li_{n_1+n_2}^{\H}(a_1a_2)=\Li_{n_1}^{\H}(a_1)\Li_{n_2}^{\H}(a_2).
\]
\end{example}

The quasi-shuffle relation is almost immediate for the corresponding power series but is much harder to prove on the Hodge level. A cohomological proof is explained in \cite[\S 9]{Gon02} and a proof using the rigidity argument (\S \ref{SectionRigidity}) of a similar statement for Hodge correlators can be found in \cite{Mal20}.  In this section we give another proof of  the quasi-shuffle relation, deriving it from a more general Proposition \ref{TheoremFormalSSR}, which we call ``formal quasi-shuffle relation.'' The proof of Proposition \ref{TheoremFormalSSR} uses the same technique as a more complicated formal quadrangulation formula (Theorem \ref{TheoremFormalQF}).

\subsection{Smash coproduct Hopf algebras}\label{SectionSmashCoproduct}
Let $(\F,\Delta^{\F\F},m_\F)$ and $(\H,\Delta^{\H\H},m_\H)$ be  commutative Hopf algebras. In applications we will take $\F=\F_S$ and $\H=\H_S$ defined in \S \ref{SectionQuasiShuffle}.  In this section, we recall the construction of a smash coproduct Hopf algebra $\H \times \F$ from \cite{Mol77}, which plays a major role in what follows. Commutative Hopf algebras are dual to pro-affine algebraic groups. In the ``dual'' language a smash coproduct Hopf algebra corresponds to a split extension of  the corresponding groups.
 
A coaction of $\H$ on $\F$ is a map 
\[
\Delta^{\mathcal{HF}}\colon \F \lra  \H \otimes \F,
\]
which gives $\F$ a structure of a comodule over $\H.$ Comodules over $\H$ form a tensor category; in particular, we have a comodule structure on $\F\otimes \F$ 
\[
\Delta^{\mathcal{HFF}}\colon \F\otimes \F \lra  \H\otimes \F \otimes  \F.
\]
defined by a map
\[
\Delta^{\mathcal{HFF}}(x\otimes y)= (m_\H \otimes 1 \otimes 1)(1\otimes T \otimes 1)(\Delta^{\H\F}(x)\otimes \Delta^{\H\F}(y)),
\]
where $T$ denotes the twist map $T(t_1\otimes t_2)=t_2 \otimes t_1$. Assume that the coaction $\Delta^{\H\F}$ commutes with the product $m_{\F}$ and coproduct $\Delta^{\F\F},$ as well as with the unit  and the counit.  Under these assumptions we can construct the smash coproduct Hopf algebra $\H\times \F,$ see \cite[Theorem (2.14)]{Mol77}. As an algebra, $\H\times \F$  is the tensor product $\H \otimes \F.$ The coproduct 
\be \label{FormulaSmashCoproduct}
\Delta^{\H\times\F}\colon \H\times \F \lra (\H\times \F)\otimes (\H\times \F)
\ee 
in  $\H\times \F$ is defined by a rather complicated formula
\be \label{FormulaCoproductSmash}
\Delta^{\H\times \F}=(1\otimes T \otimes 1)(1\otimes m_\H \otimes 1 \otimes 1)(1\otimes 1 \otimes T \otimes 1)(\Delta^{\H\H}\otimes 1\otimes \Delta^{\H\F})(1 \otimes \Delta^{\F\F}).
\ee

The space  of coinvariants
\[
\F^{\H}=\{x\in \F\mid \Delta^{\H\F}x=1\otimes x\}
\]
is a subalgebra of $\F;$ it is closed under the product $m_{\F}.$ Let $i\colon \F \lra \H \times \F$ be an embedding $i(x)=1\otimes x.$ The map $i$ is an algebra homomorphism, but (in general) it does not commute with coproducts. Nevertheless, it is easy to see from formula (\ref{FormulaSmashCoproduct}) that  if $x \in \F^\H$ and $\Delta(x) \in \F^\H\otimes \F^\H,$ we have 
\be\label{FormulaCoinvariantsEmbedding}
\Delta^{\H\times\F}(i(x))=(i\otimes i)\Delta^{\F\F}(x).
\ee
We use the same symbol for $x\in \F^{\H}$ and $i(x)\in \H \times \F.$

Next, we give a group-theoretic interpretation of the smash coproduct Hopf algebra $\H \times \F$.  Let $\widehat{\F}$ and $\widehat{\H}$  be pro-affine algebraic groups dual to $\F$ and $\H.$ Then the coaction $\Delta^{\H\F}$ that commutes with product, coproduct, unit, and counit gives a homomorphism \[\varphi \colon \widehat{\H}\lra \textup{Aut}(\widehat{\F}).\] The dual of $\H \times \F$ is the semidirect product $\widehat{\H}\ltimes_{\varphi} \widehat{\F}.$ Finally, $\F^{\H}$ is the algebra of functions on $\widehat{\F}$ that are invariant under the action of $\widehat{\H}.$

We use the following notation for the reduced coaction:
\[
\widetilde{\Delta}^{\H\F}(x)=\Delta^{\H\F}(x)-1\otimes x
\]
and for the reduced coproduct:
\[
\begin{split}
&\widetilde{\Delta}^{\F\F}(x)=\Delta^{\F\F}(x)-x\otimes 1-1\otimes x,\\
&\widetilde{\Delta}^{\H\H}(x)=\Delta^{\H\H}(x)-x\otimes 1-1\otimes x.\\	
\end{split}
\]

\subsection{Coaction of $\H_S$ on $\mathcal{F}_S$} \label{SectionCoaction}
In \S \ref{SectionQuasiShuffle} we defined  Hopf algebras $\F_S$ and  $\H_S.$
In this section, we define a coaction of   $\H_S$ on $\F_S$ inspired by the coproduct of multiple polylogarithms.

For an element
${x=[\varphi_1,n_1|\dots|\varphi_k,n_k]\in \F_S}$  and any $\varphi_0\in \C(S)^{\times}$
consider a sequence 
\[
\begin{split}
&v_{\varphi_0}(x)=(x_0,\dots,x_{n+1})\\
=&(0,\varphi_0,\underbrace{0,\dots,0,\varphi_0 \varphi_1}_{n_1},\underbrace{0,\dots,0,\varphi_0 \varphi_1 \varphi_2}_{n_2},\dots,\underbrace{0,\dots,0,\varphi_0 \varphi_1 \varphi_2\dots \varphi_{k}}_{n_k}).
\end{split}
\]
For an increasing sequence $I$ of indices $0=i_0<i_1<\dots<i_r<i_{r+1}=n+1$ with $i_1=1$ the subsequence 
\[
(x_{i_0},x_{i_1},\dots,x_{i_{r+1}})
\]
is equal to $v_{\varphi_0}(x_I)$ for a unique element $x_I\in \F_S.$
 We define the coaction 
\[
 \Delta^{\H\F}\colon \F_S\lra \H_S \otimes \F_S
\]
 by the formula
\be\label{FormulaCoaction}
\begin{split}
&\Delta^{\mathcal{HF}}(x)=(-1)^{l(x)-l(x_I)}\sum_{I=(i_0,\dots,i_{r+1})} \left (\prod_{p=1}^{r} \textup{I}^{\H}(x_{i_p};x_{i_p+1},\dots,x_{i_{p+1}-1};x_{i_{p+1}})\right )\otimes  x_I
\end{split}
\ee
where the summation goes over all sequences $0=i_0<i_1<\dots<i_r<i_{r+1}=n+1$ with $i_1=1.$ For the empty word we have $\Delta^{\H\F}(1)=1\otimes 1.$ Using (\ref{FormulaHomogenuityDivergent}) one can show that the coaction does not depend on $\varphi_0,$ so we can assume that $\varphi_0=1.$

\begin{figure}
\begin{center}
\begin{tikzpicture}[transform shape]
  \foreach \number in {1,...,5}{
        \mycount=\number
        \advance\mycount by -1
  \multiply\mycount by 45
        \advance\mycount by 0
      \node[draw, very thin, color=gray,inner sep=0.001cm] (N-\number) at (\the\mycount:3.4cm) {};
    }
  \foreach \number in {9,10,11,12}{
        \mycount=\number
        \advance\mycount by -1
  \multiply\mycount by 45
        \advance\mycount by 22.5
      \node[draw, very thin, color=gray,inner sep=0.001cm](N-\number) at (\the\mycount:3.4cm) {};
    }
\draw[midnight, line width=0.25mm] (N-1) -- (N-9);
\draw[midnight, line width=0.25mm] (N-9) -- (N-2);
\draw[midnight, line width=0.25mm] (N-2) -- (N-10);
\draw[midnight, line width=0.25mm] (N-10) -- (N-3);
\draw[midnight, line width=0.25mm] (N-3) -- (N-11);
\draw[midnight, line width=0.25mm] (N-11) -- (N-4);
\draw[midnight, line width=0.25mm] (N-4) -- (N-12);
\draw[midnight, line width=0.25mm] (N-12) -- (N-5);
\draw[midnight, line width=0.25mm] (N-5) -- (N-1);

\node[left] at (N-5) {$0$};
\node[left] at (N-12) {$1$};
\node[yshift=0.11cm, xshift=-0.31cm] at (N-4) {$0$};
\node[yshift=0.21cm, xshift=-0.21cm] at (N-11) {$a_1$};
\node[above] at (N-3) {$0$};
\node[yshift=0.21cm, xshift=0.21cm] at (N-10) {$0$};
\node[yshift=0.15cm, xshift=0.41cm] at (N-2) {$a_1a_2$};
\node[right] at (N-9) {$0$};
\node[right] at (N-1) {$a_1a_2a_3$};

\draw[midnight, line width=0.25mm] (N-12) -- (N-3);
\draw[midnight, line width=0.25mm] (N-10) -- (N-1);
\end{tikzpicture}
\end{center}
\caption{Term corresponding to $(0,1,4,5,8)$ in $\Delta^{\H\F}[a_1,1|a_2,2|a_3,1]$  is equal to 
 $\left( \textup{I}^{\H}(1;0,a_1;0)\textup{I}^{\H}(0;a_1a_2,0;a_1a_2a_3)\right)\otimes [a_1a_2a_3,2].$  
 }
\label{FigureCoaction}
\end{figure}
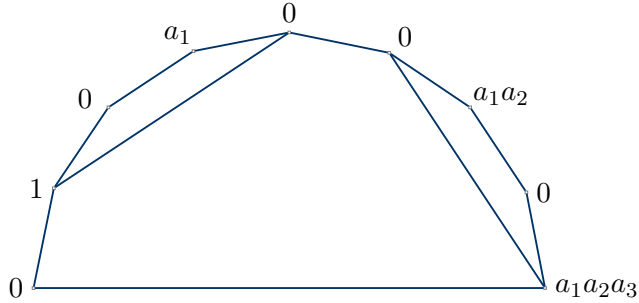


To illustrate the formula (\ref{FormulaCoaction}) consider a polygon inscribed in a semi-circle with diameter $x_0 \ x_{n+1}$ and vertices $x_0,\dots,x_{n+1}$ arranged clockwise. Each term in $\Delta^{\H\F}$ corresponds to a choice of marked points $x_{i_0},\dots, x_{i_{r+1}}$ on it with $i_0=0, i_1=1, i_{r+1}=n+1 $. These points form a convex polygon $I$, which defines the term $x_I.$ The ``complement'' of this polygon in the semi-circle defines a collection of polygons. The first one has vertices  $x_{i_1},x_{i_1+1},\dots,x_{i_2},$ the second one has vertices   $x_{i_2},x_{i_2+1},\dots,x_{i_3}$ and so on.
The term 
\[
\prod_{p=1}^{r} \textup{I}^{\H}(x_{i_p};x_{i_p+1},\dots,x_{i_{p+1}-1};x_{i_{p+1}})
\]
is a product of iterated integrals corresponding to these polygons (see Figure \ref{FigureCoaction}). The only difference between formula (\ref{FormulaCoaction}) for the coaction and formula  (\ref{FormulaCoproduct}) for the coproduct of Hodge iterated integrals is the requirement that $i_1=1$ for the subsequences involved in the summation. 

So, formula (\ref{FormulaCoaction}) is related to the formula for the coproduct of multiple polylogarithm 
\[
\Li^{\H}(x)=\Li^{\H}_{n_1,n_2,\dots,n_k}(\varphi_1,\varphi_2,\dots,\varphi_k).
\]
More precisely, if we substitute in (\ref{FormulaCoaction}) each $x_I$ with $\Li^{\H}(x_I)$ we will get the terms of the coproduct $\Delta^{\H\H}\Li^{\H}(x),$ which correspond to sequences  $0=i_0<i_1<\dots<i_r<i_{r+1}=n+1$ with $i_1=1.$ In \S \ref{SectionSSR} we show that in order to get the whole coproduct $\Delta^{\H\H}\Li^{\H}(x)$ we need to apply the formula (\ref{FormulaCoproductSmash}) for the coproduct in $\H_S\times \F_S.$ The sign $(-1)^{l(x)-l(x_I)}$ is related to the sign $(-1)^k$ in (\ref{FormulaMPviaII}).

\begin{example} We take $S=\mathbb{A}^2.$ In weight $2$ we have
\[
\begin{split}
&\Delta^{\mathcal{HF}}[a_1,1|a_2,1]=1\otimes[a_1,1|a_2,1] -\textup{I}^{\mathcal{\H}}(1;a_1;a_1a_2)\otimes [a_1a_2,1];\\
&\Delta^{\mathcal{HF}}[a,2]=1\otimes [a,2]+\log^{\H}(a)\otimes [a,1].\\
\end{split}
\]
Here is a more complicated example in weight $3$ and depth $2$:
\begin{align*}
\Delta^{\mathcal{HF}}[a_1,1|a_2,2]&=1\otimes [a_1,1|a_2,2]-\textup{I}^{\mathcal{\H}}(1;a_1,0;a_1a_2) \otimes [a_1a_2,1]\\
&\quad +\log^{\H}(a_2)\otimes [a_1,1|a_2,1]-\log^{\H}\left(\frac{a_1}{1-a_1}\right)\otimes [a_1a_2,2].
\end{align*}
\end{example}

\begin{lemma}\label{LemmaCoassociativity}
	The morphism $\Delta^{\mathcal{HF}}\colon \F_S \lra \H_S \otimes \F_S$ defines a coaction of $\H_S$ on $\F_S.$
\end{lemma}	
\begin{proof}
The counit axiom is immediate. It remains to check coassociativity, namely the equality
\be \label{FormulaCoassociativity}
(1\otimes \Delta^{\H\F})\Delta^{\H\F} =(\Delta^{\mathcal{HH}}\otimes 1)\Delta^{\H\F}.
\ee
Each term in $\H_S\otimes \H_S\otimes \F_S$ on both sides of equation (\ref{FormulaCoassociativity}) corresponds to a choice of a pair of nested polygons 
\[
\{x_0,x_1,x_{n+1}\}\subseteq I\subseteq I' \subseteq \{x_0,x_1,\dots,x_{n+1}\}.
\]
and is equal to the product $(-1)^{l(x)-l(x_I)}\prod I_{\alpha}^\H\otimes \prod I_{\beta}^\H \otimes x_{I},$ where iterated integrals $I_\alpha^\H$ correspond to the polygons between $I'$ and  $\{x_0,\dots,x_{n+1}\}$ and iterated integrals  $I_{\beta}^\H$  correspond to the polygons between $I$ and  $I'.$
\end{proof}

Next, we see that coaction (\ref{FormulaCoaction}) is compatible with the coalgebra structure on $\F_S.$ 

\begin{lemma}\label{LemmaCoproductCompatibility}
The coproduct $\Delta^{\mathcal{FF}}$ is a map of $\H_S$-comodules. In other words, the following diagram is commutative:	
\[
\begin{gathered}
    \xymatrix{
    &\F_S  \ar[d]^{\Delta^{\mathcal{FF}}}\ \ \  \ar[r]^{\Delta^{\mathcal{HF}}}&\ \ \ \  \H_S\otimes \F_S  \ar[d]^{1\otimes\Delta^{\mathcal{FF}}} \\
 & \F_S\otimes \F_S  \ar[r]^{\Delta^{\mathcal{HFF}}}& \H_S\otimes \F_S\otimes \F_S
 }
\end{gathered}.
\]
\end{lemma}
\begin{proof}
We need to show that
\be \label{FormulaCoproductCompatibility}
\Delta^{\H\F\F}\Delta^{\F\F}(x)=(1\otimes \Delta^{\F\F}) \Delta^{\H\F}(x)
\ee
for $x=[\varphi_1,n_1|\varphi_2,n_2|\dots|\varphi_k,n_k].$ After applying formulas  (\ref{FormulaCoproductQS}) and  (\ref{FormulaCoaction}), both sides of (\ref{FormulaCoproductCompatibility}) are  the sums of the terms of the type
\be\label{FormulaGeneralTerm}
(-1)^{l(x)-l(x_{I})-l(x_{I}')}\left (\prod \textup{I}^\H\right ) \otimes x_{I} \otimes  x_{I}'.
\ee
Each term corresponds to a pair consisting of a subpolygon of 
$x_0,\dots,x_{n+1}$ with vertices  $x_{i_0},\dots, x_{i_{r+1}}$ such that $i_0=0, i_1=1, i_{r+1}=n+1$ and a chord from $x_0$ to $x_{i_c}\neq 0$, see Figure \ref{FigureCoaction2}.
Assume that  $x_{i_c}=\varphi_1\ldots \varphi_j.$  Then $I$ is a sequence $(i_0, i_1,\ldots, i_{c},i_{n+1})$ and $x_I$ is the unique element in $\F_S$ with 
\[
v_1(x_I)=(x_{i_0},x_{i_1},\ldots, x_{i_{c}}).
\]
Next, $x_{I}'$ is the unique element of $\F_S$ with
\[
v_{x_{i_c}}(x_I')=(0,x_{i_{c}},x_{i_{c+1}}, x_{i_{c+2}},\ldots, x_{i_{r+1}}).
\]
Finally, the iterated integrals in (\ref{FormulaGeneralTerm}) correspond to the connected components of the complement  of the subpolygon with vertices $x_{i_0},\dots, x_{i_{r+1}}$.

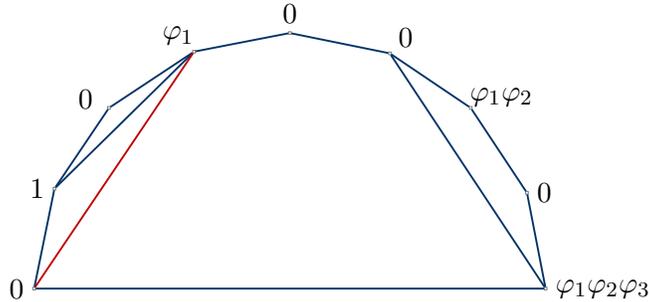
\begin{figure}
\begin{center}
\begin{tikzpicture}[transform shape]
  \foreach \number in {1,...,5}{
        \mycount=\number
        \advance\mycount by -1
  \multiply\mycount by 45
        \advance\mycount by 0
      \node[draw, very thin, color=gray,inner sep=0.001cm] (N-\number) at (\the\mycount:3.4cm) {};
    }
  \foreach \number in {9,10,11,12}{
        \mycount=\number
        \advance\mycount by -1
  \multiply\mycount by 45
        \advance\mycount by 22.5
      \node[draw, very thin, color=gray,inner sep=0.001cm](N-\number) at (\the\mycount:3.4cm) {};
    }
\draw[midnight, line width=0.25mm] (N-1) -- (N-9);
\draw[midnight, line width=0.25mm] (N-9) -- (N-2);
\draw[midnight, line width=0.25mm] (N-2) -- (N-10);
\draw[midnight, line width=0.25mm] (N-10) -- (N-3);
\draw[midnight, line width=0.25mm] (N-3) -- (N-11);
\draw[midnight, line width=0.25mm] (N-11) -- (N-4);
\draw[midnight, line width=0.25mm] (N-4) -- (N-12);
\draw[midnight, line width=0.25mm] (N-12) -- (N-5);
\draw[midnight, line width=0.25mm] (N-5) -- (N-1);

\node[left] at (N-5) {$0$};
\node[left] at (N-12) {$1$};
\node[yshift=0.11cm, xshift=-0.31cm] at (N-4) {$0$};
\node[yshift=0.21cm, xshift=-0.21cm] at (N-11) {$\varphi_1$};
\node[above] at (N-3) {$0$};
\node[yshift=0.21cm, xshift=0.21cm] at (N-10) {$0$};
\node[yshift=0.15cm, xshift=0.41cm] at (N-2) {$\varphi_1\varphi_2$};
\node[right] at (N-9) {$0$};
\node[right] at (N-1) {$\varphi_1\varphi_2\varphi_3$};

\draw[midnight, line width=0.25mm] (N-12) -- (N-11);
\draw[red, line width=0.25mm] (N-11) -- (N-5);
\draw[midnight, line width=0.25mm] (N-10) -- (N-1);
\end{tikzpicture}
\end{center}
\caption{The pair of the subpolygon $(0,1,3,4,5,8)$ and the chord $(0,3)$ corresponding to $\left( \textup{I}^{\H}(1;0;\varphi_1)\textup{I}^{\H}(0;\varphi_1 \varphi_2,0;\varphi_1\varphi_2\varphi_3)\right)\otimes [\varphi_1,1]\otimes [\varphi_2\varphi_3,3].$  
 }
\label{FigureCoaction2}
\end{figure}
Since the terms in both sides of (\ref{FormulaCoproductCompatibility}) are the same, the statement follows.

\end{proof}

To construct the smash coproduct Hopf algebra $\H_S\times \F_S$, we need to check that the quasi-shuffle multiplication ${\F_S\otimes \F_S \stackrel{\star}{\lra} \F_S}$ is a morphism of $\H_S$-comodules. We will prove it in \S \ref{SectionFormalSSR}; we need to do some preparation first.

\subsection{Principal coefficient}\label{SectionProjectionToPrimitiveElements}
The elements $[\varphi|n]\in \F_S$ play a particular role as these elements are primitives in $\F_S:$
\[
\Delta^{\F\F}[\varphi|n]=1\otimes[\varphi|n]+[\varphi|n]\otimes 1.
\]
First, we compute the coaction on these elements.

\begin{lemma} \label{LemmaCoactionClassical}
We have the following  formula for the coaction on primitive elements of $\mathcal{F}_S:$
\[
\Delta^{\mathcal{HF}}[\varphi|n]=\sum_{i=0}^{n-1} \frac{\left (\log^\H(\varphi)\right )^{i}}{i!}\otimes [\varphi|n-i].
\]
\end{lemma}
\begin{proof}
By (\ref{FormulaCoaction}) we have 
\[
\Delta^{\mathcal{HF}}[\varphi|n]=\sum_{i=0}^{n-1}\textup{I}^\H(0;\underbrace{0,\dots,0}_{i};\varphi)\otimes [\varphi|n-i],
\]
because all the other terms vanish. The statement of the lemma follows from (\ref{FormulaLogPower}).

\end{proof}

Consider a linear map $\pi_d \colon \F_S \lra \H_S$ for $d \geq 0$
which sends $x=[\varphi_1,n_1|\varphi_2,n_2|\dots|\varphi_k,n_k]$ to the coefficient  in $\H_S$ at the term
\[
\left[\prod_{i=1}^k \varphi_i,n_1+n_2+\ldots+n_k-d \right]
\]
of the coaction. For the empty word we put 
\[
\pi_d(1)=
\begin{cases}
1 & \text{ if }d=0, \\
0 & \text{ if }d>0. \\	
\end{cases}
\]
  It is easy to see that
\[
\pi_0([\varphi_1,n_1|\varphi_2,n_2|\dots|\varphi_k,n_k])=
\begin{cases}
	0 & \text{\ for  \ } k \geq 1,\\
	1 & \text{\ for  \ } k=1.\\
\end{cases}
\]
From here it follows that 
\be\label{FormulaPi0Product}
\pi_0(x\star y)=\pi_0(x)\pi_0(y).
\ee
The map $\pi_1$ lands in the weight $1$ component of $\H_S,$ which is isomorphic to $\mathbb{C}(S)^{\times}_\Q.$  We call the map $\pi_1\colon \F_S\lra \mathbb{C}(S)^{\times}_\Q$ the principal coefficient map. This map plays a fundamental role in the proof of the formal quasi-shuffle relation (Proposition \ref{TheoremFormalSSR}) and  the proof of the formal quadrangulation formula.

\begin{lemma} \label{LemmaMapPi} The following equalities hold for $n\geq 3$
\be\label{FormulaMapPi}
\begin{split}
	&\pi_1([\varphi_1,2])=\log^{\H}\left(\varphi_1\right),\\
	&\pi_1([\varphi_1,n])=\log^{\H}\left(\varphi_1\right),\\
	&\pi_1([\varphi_1,1|\varphi_2,1])=\log^{\H}\left(\frac{\ \: 1-\varphi_1^{-1}}{1-\varphi_2}\right),\\
	&\pi_1([\varphi_1,n-1|\varphi_2,1])=-\log^{\H}\left(1-\varphi_2\right),\\
	&\pi_1([\varphi_1,1|\varphi_2,n-1])=\log^{\H}\left(1-\varphi_1^{-1}\right).\\
	\end{split}
\ee
For all other words in $\F_S$  we have $\pi_1([\varphi_1,n_1|\varphi_2,n_2|\dots|\varphi_k,n_k])=0.$
\end{lemma}
\begin{proof} 
For $x=[\varphi_1,n_1|\varphi_2,n_2|\dots|\varphi_k,n_k]$ consider a sequence 
\[
(x_0,\dots,x_{n+1})=(0,1,\underbrace{0,\dots,0,\varphi_1}_{n_1},\underbrace{0,\dots,0,\varphi_1\varphi_2}_{n_2},\dots,\underbrace{0,\dots,0,\varphi_1\varphi_2\dots \varphi_{k}}_{n_k}).
\]  
Then $\pi_1(x)$ is equal to zero unless the sequence $(x_0,\dots,x_{n+1})$ has a subsequence 
\[
(0,1,\underbrace{0,\dots,0}_{n-2},\varphi_1\varphi_2\dots \varphi_{k}),
\]
where $n=n_1+n_2+\ldots+n_k.$ This cannot happen for $k>2.$  For $k=2$ it is only possible if $n_1$ or $n_2$ is equal to $1;$ these cases can be verified by hand. Finally, for $k=1$ the statement follows from Lemma \ref{LemmaCoactionClassical}. 
\end{proof}

The following statement is a very special instance of the fact that the coaction commutes with the quasi-shuffle product.
\begin{corollary}
The following equality holds for $x,y\in \F_S$:
\be \label{FormulaPiProduct}
\pi_1(x\star y)=\pi_0(x)\pi_1(y)+\pi_1(x)\pi_0(y)\in \mathbb{C}(S)^{\times}_\Q.
\ee	
\end{corollary}
\begin{proof}
If the length of one of the words $x$ or $y$ is at least three, formula (\ref{FormulaPiProduct}) holds because all terms involved vanish by Lemma \ref{LemmaMapPi}. The remaining cases can be easily verified.
\end{proof}

\subsection{Formal quasi-shuffle relation}\label{SectionFormalSSR}
Our main result is the following theorem. 
 \begin{proposition}[Formal Quasi-Shuffle Relation]\label{TheoremFormalSSR}
The quasi-shuffle product $\star\colon \F_S\otimes \F_S\lra \F_S$ is a map of $\H_S$-comodules. In other words, the following diagram is commutative:	
\[ 
\begin{gathered}
    \xymatrix{
    &\F_S\otimes \F_S  \ar[d]^{\star} \  \ar[r]^{\Delta^{\mathcal{HFF}}}&  \H_S\otimes \F_S \otimes \F_S \ar[d]^{1\otimes\star} \\
 & \F_S \ar[r]^{\Delta^{\mathcal{HF}}}& \H_S\otimes \F_S }
\end{gathered}.
\]
\end{proposition}

\begin{proof}
It is easy to see that it suffices to prove that for  
$S=\textup{Spec}\left(\C[a_1,\dots,a_k,a_1',\dots,a_{k}']\right)$ and
\be \label{FormulaXY}
\begin{split}
	&x=[a_1,n_1|\dots|a_{k},n_k]\in \mathcal{F}_S, \\
	&y=[a'_1,n'_1|\dots|a_{k}',n'_{k'}]\in \mathcal{F}_S\\
\end{split}
\ee
we have $\Delta^{\mathcal{HF}}(x\star y)=\Delta^{\mathcal{HF}}(x)\Delta^{\mathcal{HF}}(y).$
We prove this statement by induction on $n+n',$ where $n=n_1+\dots+n_k$ and $n'=n_1'+\dots+n_{k'}'.$ The base  case of $n=n'=0$ is obvious. 
 
Denote the difference by 
\be\label{FormulaDifference}
D=\bigl(\Delta^{\mathcal{HF}}(x\star y)-\Delta^{\mathcal{HF}}(x)\Delta^{\mathcal{HF}}(y)\bigr)\in \H_S \otimes \F_S.
\ee 

First, we show that the difference $D$ satisfies 
\be \label{FormulaFirstEquality}
(1\otimes\widetilde{\Delta}^{\F\F})D=0.
\ee
In more  concrete terms, we need to show that 
\be \label{FormulaPrimitive}
(1\otimes \Delta^{\F\F})D=D \otimes 1+(T\otimes 1)(1\otimes D)\in \H_S\otimes \F_S\otimes\F_S
\ee
where $T(a\otimes b)=b \otimes a.$ 
Indeed, assume that 
\[
\begin{split}
&\Delta^{\F\F}x=\sum x_i^{(1)}\otimes x_i^{(2)}\in \F_S\otimes\F_S,\\
&\Delta^{\F\F}y=\sum y_j^{(1)}\otimes y_j^{(2)}\in \F_S\otimes\F_S.
\end{split}
\]
Then 
\[
\Delta^{\F\F}(x\star y)=\sum (x_i^{(1)}\star y_j^{(1)})\otimes (x_i^{(2)}\star y_j^{(2)})\in \F_S\otimes\F_S.
\]
By induction, we know that
\[
\Delta^\mathcal{HF}\bigl(x_i^{(1)}\star y_j^{(1)}\bigr)=\Delta^\mathcal{HF}\bigl(x_i^{(1)}\bigr)\Delta^\mathcal{HF}(y_j^{(1)})
\]
unless $x^{(1)}=x$ and $y^{(1)}=y;$ similar for $\Delta^\mathcal{HF}\bigl(x_i^{(2)}\star y_j^{(2)}\bigr).$

Thus the following equality holds in $\H_S \otimes \F_S\otimes \H_S \otimes \F_S:$

\begin{align*}
(\Delta^\mathcal{HF}\otimes \Delta^{\mathcal{HF}}) \Delta^{\mathcal{FF}}(x\star y)&=\sum \Delta^\mathcal{HF}(x_i^{(1)}\star y_j^{(1)})\otimes \Delta^\mathcal{HF}(x_i^{(2)}\star y_j^{(2)})\\
&=\Delta^\mathcal{HF}(x\star y)\otimes 1\otimes 1+1\otimes 1\otimes  \Delta^\mathcal{HF}(x\star y)\\
&\quad-\left(\Delta^\mathcal{HF}(x)\Delta^\mathcal{HF}(y)\right)\otimes 1\otimes 1-1\otimes 1\otimes \left(\Delta^\mathcal{HF}(x)\Delta^\mathcal{HF}(y)\right)\\
&\quad+\sum \left(\Delta^\mathcal{HF}\bigl(x_i^{(1)}\bigr)\Delta^\mathcal{HF}\bigl(y_j^{(1)}\bigr)\right)\otimes \left(\Delta^\mathcal{HF}\bigl(x_i^{(2)}\bigr)\Delta^\mathcal{HF}\bigl(y_j^{(2)}\bigr)\right),
\end{align*}
which implies that
\begin{align*}
(\Delta^\mathcal{HF}\otimes \Delta^{\mathcal{HF}}) \Delta^{\mathcal{FF}}(x\star y)&=D\otimes1 \otimes 1+1\otimes 1\otimes D\\
&\quad+\left((\Delta^\mathcal{HF}\otimes \Delta^{\mathcal{HF}}) \Delta^{\mathcal{FF}}(x)\right)\left((\Delta^\mathcal{HF}\otimes \Delta^{\mathcal{HF}}) \Delta^{\mathcal{FF}}(y)\right).
\end{align*}

Using the definition of $\Delta^{\H\F\F}$, we can rewrite it as follows:
\[ 
\Delta^{\H\F\F}(\Delta^{\F\F}(x\star y))=D \otimes 1+(T\otimes 1)(1\otimes D)+(\Delta^{\H\F\F}(\Delta^{\F\F}(x)))(\Delta^{\H\F\F}(\Delta^{\F\F}(y))).
\]
 On the other hand, by Lemma \ref{LemmaCoproductCompatibility} and (\ref{FormulaDifference}) this implies that
 \[
\left(1\otimes \Delta^{\F\F}\right)(D)=D \otimes 1+(T\otimes 1)(1\otimes D),
\]
which proves (\ref{FormulaPrimitive}).

Next, we show that the difference $D$ satisfies 
\be \label{FormulaSecondEquality}
(\widetilde{\Delta}^{\mathcal{HH}}\otimes 1)D=(1\otimes \widetilde{\Delta}^{\H\F})D.
\ee
Equivalently, we need to show that
\be\label{FormulaSecondResult}
(\Delta^{\mathcal{HH}}\otimes 1)D-(1\otimes \Delta^{\H\F})D=1\otimes D.
\ee
By Lemma \ref{LemmaCoassociativity} we have
\be \label{FormulaSecondA}
(\Delta^{\mathcal{HH}}\otimes 1)\Delta^{\H\F}(x\star y)-(1\otimes \Delta^{\H\F})\Delta^{\H\F}(x\star y)=0.
\ee
Assume that 
\[
\begin{split}
&\Delta^{\H\F}(x)=\sum f_i^{(1)}\otimes x_i^{(2)}\in \H\otimes \F,\\
&\Delta^{\H\F}(y)=\sum g_j^{(1)}\otimes y_j^{(2)}\in \H\otimes \F.
\end{split}
\]
Then
\begin{align*}
(\Delta^{\H\H}\otimes 1)(\Delta^{\H\F}(x)\Delta^{\H\F}(y))=&\sum \Delta^{\H\H}\bigl(f_i^{(1)}g_j^{(1)}\bigr) \otimes \bigl(x_i^{(2)}\star y_j^{(2)}\bigr)\\
=&\sum \Delta^{\H\H}\bigl(f_i^{(1)}\bigr) \Delta^{\H\H}\bigl(  g_j^{(1)}\bigr) \otimes \bigl ( x_i^{(2)}\star y_j^{(2)}\bigr) \\
=&\sum \left ( \Delta^{\H\H}\bigl(f_i^{(1)}\bigr) \otimes   x_i^{(2)} \right) \left ( \Delta^{\H\H}\bigl(  g_j^{(1)}\bigr) \otimes y_j^{(2)}\right) \\
=&\sum \left ( f_i^{(1)} \otimes   \Delta^{\H\F}\bigl(x_i^{(2)}\bigr) \right) \left (  g_j^{(1)} \otimes \Delta^{\H\F}\bigl( y_j^{(2)}\bigr)\right)\\
=&\sum f_i^{(1)}g_j^{(1)} \otimes \left(\Delta^{\H\F}\bigl(x_i^{(2)}\bigr)\Delta^{\H\F}\bigl(y_j^{(2)}\bigr)\right)\\
\end{align*}

By induction we know that $\Delta^{\H\F}(x_i^{(2)}y_j^{(2)})=\Delta^{\H\F}(x_i^{(2)})\Delta^{\H\F}(y_j^{(2)})$ unless $x_i^{(1)}=x$ and $y_j^{(2)}=y.$ It follows that
\begin{align*}
&(\Delta^{\H\H}\otimes 1)(\Delta^{\H\F}(x)\Delta^{\H\F}(y))
=\sum f_i^{(1)}g_j^{(1)} \otimes \left(\Delta^{\H\F}\bigl(x_i^{(2)}\bigr)\Delta^{\H\F}\bigl(y_j^{(2)}\bigr)\right)\\
=&\sum f_i^{(1)}g_j^{(1)} \otimes \left(\Delta^{\H\F}\bigl(x_i^{(2)}\star y_j^{(2)}\bigr)\right)-1\otimes \Delta^{\H\F}(x\star y)+1\otimes \Delta^{\H\F}(x)\Delta^{\H\F}(y)\\
=&(1\otimes \Delta^{\H\F})\bigl(\Delta^{\H\F}(x)\Delta^{\H\F}(y)\bigr)-1\otimes D.
\end{align*}
So, we have
\be \label{FormulaSecondB}
(\Delta^{\mathcal{HH}}\otimes 1)\left( \Delta^{\H\F}(x) \Delta^{\H\F}(y)\right)-(1\otimes \Delta^{\H\F})\left(\Delta^{\H\F}(x) \Delta^{\H\F}(y)\right)=-1\otimes D.
\ee
From (\ref{FormulaSecondA}) and (\ref{FormulaSecondB}) we get (\ref{FormulaSecondResult}).

Now we are ready to finish the proof of the lemma. Denote the space of primitive elements of $\F_S$ by $\textup{Prim}(\F_S);$ this space is generated by words of length $1$. Let  $a=\prod_{i=1}^k a_i$ and $a'=\prod_{i=1}^{k'} a'_i.$
From (\ref{FormulaFirstEquality}) it follows that the  element $D\in \H_S\otimes \F_S$ lies in $\H_S\otimes \textup{Prim}(\F_S)$, and it is easy to see that only words in $\textup{Prim}(\F_S)$ which may appear in $D$ have the form
\[
\left[a a'  , n+n'-d \right].
\]
for $0\leq d\leq n+n'-1.$ This follows from the observation that all the sequences $v_1(\omega)$ corresponding to the words $\omega$ appearing in $x\star y$ have $a a'$ as the last term. So, we have 
\[
D=\sum_{i=0}^{n+n'-1}  h_i \otimes \left[a a' , n+n'-i\right]
\]
for certain $ h_0,\dots, h_{n+n'-1}\in \H_S.$ We need to show that $h_0=\dots=h_{n+n'-1}=0.$ Assume that it is not the case; let $r$ be the smallest index such that $h_r\neq 0.$

We have $h_0=\pi_0(x\star y)-\pi_0(x)\pi_0(y)=0$ by (\ref{FormulaPi0Product}), so $r\neq 0.$
Also, by (\ref{FormulaPiProduct}) we have 
$h_1=\pi_1(x\star y)-\pi_0(x)\pi_1(y)-\pi_1(x)\pi_0(y)=0,$
so $r>1.$

From (\ref{FormulaSecondEquality}) we have 
\[
\sum_{i=0}^{n+n'-1} \left ( \widetilde{\Delta}^{\mathcal{HH}}(h_i)\otimes \left[aa', n+n'-i\right]\right)=\sum_{i=0}^{n+n'-1} \left( h_i\otimes \widetilde{\Delta}^{\mathcal{HF}}\left[aa', n+n'-i\right]\right).
\]

Comparing coefficients in front of $[aa'|n+n'-r]$
we get that
\[
\widetilde{\Delta}^{\mathcal{HH}}(h_r)=\sum_{i=1}^r h_{r-i}\otimes \frac{(\log^{\H}(aa'))^i}{i!} ,
\]
and since $h_0=\dots=h_{r-1}=0$ we have $\widetilde{\Delta}^{\mathcal{HH}}(h_r)=0.$ Consider specialization to the point $a_1=\dots=a_n=a_1'=\dots=a_{n'}'=0.$ From (\ref{FormulaCoaction}) it is clear that $h_r$ is a linear combination of iterated integrals 
\[
\textup{I}^{\H}(0;1,\underbrace{0,\dots,0,m_1}_{d_1},\dots,\underbrace{0,\dots,0,m_1m_2\dots m_{s-1}}_{d_{s-1}},\underbrace{0,\dots,0;m_1m_2\dots m_{s}}_{d_s}),
\]
where $m_i$ are monomials in $a_1,\dots, a_k, a_1',\dots, a_{k'}'$ of degree at least $1,$ so $h_r$ vanishes after the specialization. By the rigidity argument (\S \ref{SectionRigidity}) $h_r$ vanishes, which contradicts our assumption. This finishes the proof of the theorem. 
\end{proof}

\subsection{Quasi-shuffle relation for multiple polylogarithms}\label{SectionSSR}
Using the results of \S\S\ref{SectionSmashCoproduct}-\ref{SectionFormalSSR} we  construct the smash coproduct Hopf algebra $\H_S\times \F_S$. Consider a map 
${\Li^{\H}\colon \F_S \lra \H_S}$ defined by the formula
\[
\Li^{\H}[\varphi_1,n_1|\varphi_2,n_2|\dots|\varphi_k,n_k]=
\Li^{\H}_{n_1,\dots,n_k}(\varphi_1,\varphi_2,\dots,\varphi_k).
\]
This map can be  extended to an $\H_S$-linear map
${\Li^{\H}\colon \H_S\times \F_S \lra \H_S}.$ 

\begin{proposition}\label{PropositionMapToPolylogarithms} The map
	${\Li^{\H}\colon \H_S \times \F_S \lra \H_S}$  is a homomorphism of  Hopf algebras. In particular, the quasi-shuffle relation for multiple polylogarithms
	\[
	\Li^{\H}(x \star y)=\Li^{\H}(x)\Li^{\H}( y)
	\]
	 holds for $x, y \in \F_S.$
\end{proposition}
\begin{proof}
It is easy to see that the map $\Li^{\H}$ commutes with the unit and the counit. We need to check that it commutes with the product and the coproduct. For elements $h\otimes 1\in \H_S \otimes \F_S$ this is obvious. It suffices to check these statements on elements of type $1\otimes x \in \H_S\otimes \F_S.$

First, we check compatibility with the coproduct, namely the equality
\[
( \Li^\H\otimes \Li^\H)\Delta^{\H\times \F}(1\otimes x)=\Delta^{\H\H}\Li^\H(1\otimes x)
\] 	
for $x=[\varphi_1,n_1|\varphi_2,n_2|\dots|\varphi_k,n_k].$ From (\ref{FormulaCoproductSmash}) it is easy to see that we have
\[
\Delta^{\H\times \F}(1\otimes x)=1\otimes \bigl((1\otimes \Delta^{\H\F})\Delta^{\F\F}x\bigr ) \in  \H_S\otimes \F_S\otimes \H_S\otimes \F_S.
\]
Recall the presentation of $\Li^\H(x)$ as an iterated integral 
\[
\begin{split}
	&\Li^{\H}_{n_1,\dots,n_k}(\varphi_1,\varphi_2,\dots,\varphi_k)\\
&=(-1)^{k}\textup{I}^{\H}(0;1,\underbrace{0,\dots,0,\varphi_1}_{n_1},\underbrace{0,\dots,0,\varphi_1\varphi_2}_{n_2},\dots,\underbrace{0,\dots,0;\varphi_1\varphi_2\dots \varphi_{k}}_{n_k}).
\end{split}
\] 
Let 
\[
(x_0,\dots,x_{n+1})=(0,1,\underbrace{0,\dots,0,\varphi_1}_{n_1},\underbrace{0,\dots,0,\varphi_1\varphi_2}_{n_2},\dots,\underbrace{0,\dots,0,\varphi_1\varphi_2\dots \varphi_{k}}_{n_k})
\]
be the sequence of  arguments of the corresponding iterated integral. From here, we see that
\[
\begin{split}
	&\Delta^{\H\H}\Li^{\H}(x)= (-1)^k\Delta^{\H\H} \textup{I}^{\H}(x_0;x_1,\dots,x_n;x_{n+1})\\
	&=(-1)^k\sum_{(i_0,\dots,i_{r+1})\in \mathcal{I}}\left (\prod_{p=0}^{r} \textup{I}^{\H}(x_{i_p};x_{i_p+1},\dots,x_{i_{p+1}-1};x_{i_{p+1}})\right )\otimes  \textup{I}^{\H}(x_{i_0};x_{i_1},\dots,x_{i_r}; x_{i_{r+1}})
\end{split}
\]
where the summation goes over the  set $\mathcal{I}$ of  sequences 
\[
0=i_0<i_1<\dots<i_r<i_{r+1}=n+1.
\]
We break the set $\mathcal{I}$ into subsets $\mathcal{I}_j$ with $i_1=j$. Since $x_0=0$, the iterated integral $\textup{I}^{\H}(x_{i_0};x_{i_0+1},\dots,x_{i_{1}-1};x_{i_{1}})$ vanishes if $x_{i_1}=0,$ so all the terms in the formula above corresponding to sequences in $\mathcal{I}_j$ vanish unless
\[
j \in \{1,1+n_1,1+n_1+n_2,\dots ,1+n_1+n_2+\dots+n_k\}.
\]
For $j=1+n_1+\dots+n_s$ where $0\leq s \leq k$ we have
\[
(-1)^s\textup{I}^{\H}(x_{i_0};x_{i_0+1},\dots,x_{i_1-1};x_{i_{1}})=\Li^{\H}[\varphi_1,n_1|\varphi_2,n_2|\dots|\varphi_s,n_s]
\]
and 
\[
\begin{split}
&(-1)^{k-s}\sum_{(i_0,\dots,i_{r+1})\in \mathcal{I}_{j}}\left (\prod_{p=1}^{r} \textup{I}^{\H}(x_{i_p};x_{i_p+1},\dots,x_{i_{p+1}-1};x_{i_{p+1}})\right )\otimes  \textup{I}^{\H}(x_{i_0};x_{i_1},\dots,x_{i_r}; x_{i_{r+1}})\\
&=(1\otimes \Li^\H)\Delta^{\H\F}[\varphi_{s+1},n_{s+1}|\varphi_{s+2},n_{s+2}|\dots|\varphi_k,n_k].
\end{split}
\]
It follows that 
\be \label{FormulaCoproductMultiplePolylogarithm}
\begin{split}
&\Delta^{\H\H}\Li^\H(x)\\
&=\sum_{s=0}^k \left (\Li^{\H}[\varphi_1,n_1|\dots|\varphi_s,n_s]\otimes 1) \left ((1\otimes \Li^\H)\Delta^{\H\F}[\varphi_{s+1},n_{s+1}|\dots|\varphi_k,n_k]\right )\right).
\end{split}
\ee
Compatibility of the map ${\Li^{\H}\colon \H_S \times \F_S\lra \H_S}$ with the coproduct follows from here.

To check compatibility with the product, it is sufficient to prove that
for $x,y$ defined by (\ref{FormulaXY})  we have
\[
\Li^\H(x\star y)=\Li^\H(x)\Li^\H(y)\in \H_S,
\]
namely the quasi-shuffle relation for multiple polylogarithms. We apply the rigidity argument from \S \ref{SectionRigidity}. By Proposition \ref{TheoremFormalSSR} we have 
\begin{align*}
\Delta^{\H\H} \left(\Li^\H(x\star y)-\Li^{\H}(x)\Li^{\H}(y)\right)&=\Delta^{\H\H} (\Li(x\star y))-\Delta^{\H\H} (\Li(x))\Delta^{\H\H} (\Li(y))\\
&=(\Li^\H\otimes \Li^\H)\left(\Delta^{\H\times \F} (x\star y)-\Delta^{\H\times \F} (x)\Delta^{\H\times\F} (y)\right)\\
&=0.
\end{align*}
Specializing $\left (\Li^{\H}(x\star y)-\Li^{\H}(x)\Li^{\H}(y)\right)$ to the point $a_1=\dots=a_k=a_1'=\dots=a_{k'}'=0$  we conclude that $\Li^{\H}(x\star y)-\Li^{\H}(x)\Li^{\H}(y)=0.$
 This finishes the proof of the proposition.
\end{proof}

Proposition \ref{PropositionMapToPolylogarithms} can be extended to generalized multiple polylogarithms in the following way. To simplify the exposition, we work in the Lie coalgebra $\L$ instead of $\H.$ Consider a map 
${\Li_\bullet^{\L}\colon \F_S \lra \L_S}$ defined by the formula
\begin{align*}
\Li_\bullet^{\L}[\varphi_1,n_1|\varphi_2,n_2|\dots|\varphi_k,n_k]&=
\Li^{\L}_{\bullet;n_1,\dots,n_k}(\varphi_1,\varphi_2,\dots,\varphi_k)\\
&=\sum_{n_0=0}^{\infty} \Li^{\L}_{n_0;n_1,\dots,n_k}(\varphi_1,\varphi_2,\dots,\varphi_k).
\end{align*}

\begin{proposition}\label{PropositionMapToGeneralizedPolylogarithms}
	The generalized quasi-shuffle relation $\Li_\bullet^{\L}(x  \star  y)=0$
	 holds for homogeneous elements $x, y\in \F_S$ of positive degrees.
\end{proposition}

\begin{proof}

Similarly to (\ref{FormulaCoproductMultiplePolylogarithm}), we have
\begin{align}\label{FormulaCoproductPolylogWithZeros}
\begin{split}
&\Delta^\L\Li^{\L}_{\bullet;n_1,\dots,n_k}(\varphi_1,\dots, \varphi_k)\\
&=\Li^{\L}_{\bullet;n_1,\dots,n_k}(\varphi_1,\dots, \varphi_k)\wedge \log^\L(\varphi_1\dots \varphi_k)\\
&\quad+\sum_{s=0}^k\Li^{\L}_{\bullet;n_1,\dots,n_s}(\varphi_1,\dots, \varphi_s)\wedge \Li^{\L}_{\bullet;n_{s+1},\dots,n_k}(\varphi_{s+1},\dots, \varphi_k)\\
&\quad+(1\wedge \Li_\bullet^\L)\Delta^{\H\F}[\varphi_{1},n_{1}|\dots|\varphi_k,n_k].\\
\end{split}
\end{align}
Here is the brief explanation.
Let 
\[
(x_0,\dots,x_{n+1})=(0,\underbrace{0,\dots,0}_{n_0},1,\underbrace{0,\dots,0,\varphi_1}_{n_1},\underbrace{0,\dots,0,\varphi_1\varphi_2}_{n_2},\dots,\underbrace{0,\dots,0,\varphi_1\varphi_2\dots \varphi_{k}}_{n_k})
\]
be the sequence of  arguments of the iterated integral corresponding to the generalized multiple polylogarithm $\Li^{\L}_{n_0;n_1,\dots,n_k}$. 
The terms of the coproduct $\Delta^\L\Li^{\L}_{n_0;n_1,\dots,n_k}$ correspond to pairs $0\leq i<j\leq n+1.$ Indeed, in (\ref{FormulaCoproduct}) only the terms with $r=2$ survive in $\L\wedge\L.$ 

The term $\Li^{\L}_{\bullet;n_1,\dots,n_k}(\varphi_1,\dots, \varphi_k)\wedge \log^\L(\varphi_1\dots \varphi_k)$ in (\ref{FormulaCoproductPolylogWithZeros}) comes from the pair $i=1,j=n+1$.  The term 
\[\sum_{s=0}^k\Li^{\L}_{\bullet;n_1,\dots,n_s}(\varphi_1,\dots, \varphi_s)\wedge \Li^{\L}_{\bullet;n_{s+1},\dots,n_k}(\varphi_{s+1},\dots, \varphi_k)
\]
comes from pairs with  $1\leq i\leq n_0< j\leq n$ The term 
\[
(1\wedge \Li_\bullet^\L)\Delta^{\H\F}[\varphi_{1},n_{1}|\dots|\varphi_k,n_k]
\]
comes from pairs with  $n_0 <i< j\leq n.$ The remaining terms vanish in $\Lambda^2 \L.$

Now, we finish the proof of the generalized quasi-shuffle relation. For
\[
\begin{split}
	&x=[a_1,n_1|\dots|a_{k},n_k]\in \mathcal{F}_S, \\
	&y=[a'_1,n'_1|\dots|a_{k}',n'_{k'}]\in \mathcal{F}_S\\
\end{split}
\]
let $a= \prod_{i=1}^{k} a_i$ and  $b=\prod_{j=1}^{k'} b_j.$
We argue by induction on $k+k'.$ The base case is trivial.
Then 
\begin{align*}
&\Delta^\L\Li_\bullet^{\L}(x  \star  y)\\
&=\Li_\bullet^{\L}(x  \star  y)\wedge \log^\L\left(ab\right)+(\Li_\bullet^{\L}\wedge \Li_\bullet^{\L})(\Delta^{\F\F}(x\star y))+(1\wedge \Li_\bullet^{\L})\Delta^{\H\F}(x\star y))\\
&=\Li_\bullet^{\L}(x  \star  y)\wedge \log^\L\left(a b\right)+(\Li_\bullet^{\L}\wedge \Li_\bullet^{\L})(\Delta^{\F\F}(x)\star \Delta^{\F\F}(y))+(1\wedge \Li_\bullet^{\L})(\Delta^{\H\F}(x) \Delta^{\H\F}(y)).
\end{align*}
The terms $(\Li_\bullet^{\L}\wedge \Li_\bullet^{\L})(\Delta^{\F\F}(x)\star \Delta^{\F\F}(y))$ and $(1\wedge \Li_\bullet^{\L})(\Delta^{\H\F}(x) \Delta^{\H\F}(y))$ vanish by induction. We get 
\[
\Delta^\L\Li_\bullet^{\L}(x  \star  y)=\Li_\bullet^{\L}(x  \star  y)\wedge \log^\L\left(ab\right),
\]
or, equivalently,
\[
\Delta^\L\Li_{n_0+1}^{\L}(x  \star  y)=\Li_{n_0}^{\L}(x  \star  y)\wedge \log^\L\left(ab\right) \text{ for }n_0\geq 0.
\]
The statement  follows by induction on $n_0$ and the rigidity argument.
\end{proof}

\section{Formal polylogarithms on the configuration space}\label{SectionFormalOnConfiguration}

\subsection{Alternating polygons}\label{SectionAlternatingPolygons} 
An alternating polygon $\P=(p_0,p_1,\dots,p_{2n+1})$ is an increasing sequence of positive integers such that  $p_{i+1}-p_i$ is odd for $0\leq i \leq 2n.$  It is convenient to draw terms $p_0, p_1, \dots, p_{2n+1}$ in the vertices of a convex $(2n+2)$-gon, which we will denote by the same letter $\P.$ Note that $p_0-p_{2n+1}$ is also odd, so the condition of being alternating is ``cyclically invariant''. Alternating polygons are objects of  a category $\underline{\textup{Alt}}$ with morphisms being parity- and order-preserving injective maps of the corresponding sequences. There are two classes of isomorphisms of objects in $\underline{\textup{Alt}}$ for every $n,$ represented by sequences $(0,1,\dots,2n+1)$ and $(1,2,\dots,2n+2).$ The polygon $\P$ is called even if $p_0$ is even and odd if $p_0$ is odd. 

An alternating subpolygon of $\P$ is a subsequence  $(p_{i_0},p_{i_1},\dots,p_{i_{2k+1}})$ such that indices $p_{i_r}$ have alternating parity for $0\leq r\leq 2k+1$. Two subpolygons of $\P$ are called disjoint if their interiors do not intersect. Every diagonal of $\P$ with ends of different parity decomposes $\P$ into a pair of alternating polygons $\P_1$  and $\P_2.$  We say that $\P$ is a disjoint union of  $\P_1$  and $\P_2$ and denote it by $\P=\P_1 \sqcup  \P_2.$ More generally, every collection of such diagonals which do not intersect pairwise decomposes $\P$ into disjoint alternating polygons.  Denote by $\mathcal{D}(\P)$ the set of all such decompositions of $\P.$

 The set $\mathcal{D}(\P)$  is partially ordered: $D_1\leq D_2$ if every polygon in the decomposition $D_2$ is  contained in some polygon in the decomposition $D_1.$ The poset $\mathcal{D}(\P)$ has $\{\P\}$ as its smallest element. Maximal elements of $\mathcal{D}(\P)$ are called {\it quadrangulations}. In other words, a  quadrangulation $Q$ of  $\P$ is a decomposition of $\P$ into disjoint quadrangles, which will necessarily be alternating subpolygons of $\P$ (see Figure \ref{FigureHexadecagonQuadrangulation}). We denote by $\mathcal{Q}(\P)\subseteq \mathcal{D}(\P)$  the set of quadrangulations of $\P.$ The number of quadrangulations  of a $(2n+2)$-gon is equal to 
\[
|\mathcal{Q}(\P)|=\frac{(3n)!}{n!(2n+1)!}.
\]

 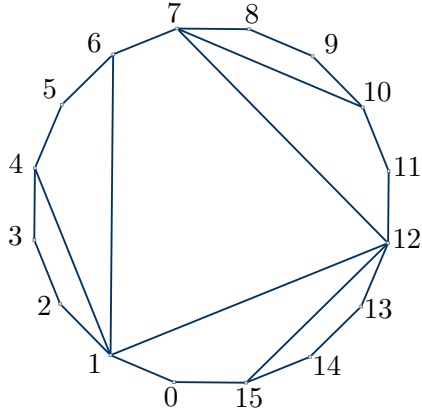
\begin{figure}
 \begin{center}
\begin{tikzpicture}[transform shape]
  \foreach \number in {1,...,8}{
        \mycount=\number
        \advance\mycount by -1
  \multiply\mycount by 45
        \advance\mycount by 11.25
      \node[draw, very thin, color=gray,inner sep=0.001cm] (N-\number) at (\the\mycount:2.4cm) {};
    }
  \foreach \number in {9,...,16}{
        \mycount=\number
        \advance\mycount by -1
  \multiply\mycount by 45
        \advance\mycount by 33.75
      \node[draw, very thin, color=gray,inner sep=0.001cm](N-\number) at (\the\mycount:2.4cm) {};
    }
\draw[midnight, line width=0.25mm] (N-1) -- (N-9);
\draw[midnight, line width=0.25mm] (N-2) -- (N-10);
\draw[midnight, line width=0.25mm] (N-3) -- (N-11);
\draw[midnight, line width=0.25mm] (N-4) -- (N-12);
\draw[midnight, line width=0.25mm] (N-5) -- (N-13);
\draw[midnight, line width=0.25mm] (N-6) -- (N-14);
\draw[midnight, line width=0.25mm] (N-7) -- (N-15);
\draw[midnight, line width=0.25mm] (N-8) -- (N-16);
\draw[midnight, line width=0.25mm] (N-1) -- (N-16);
\draw[midnight, line width=0.25mm] (N-2) -- (N-9);
\draw[midnight, line width=0.25mm] (N-3) -- (N-10);
\draw[midnight, line width=0.25mm] (N-4) -- (N-11);
\draw[midnight, line width=0.25mm] (N-5) -- (N-12);
\draw[midnight, line width=0.25mm] (N-6) -- (N-13);
\draw[midnight, line width=0.25mm] (N-7) -- (N-14);
\draw[midnight, line width=0.25mm] (N-8) -- (N-15);

\draw[midnight, line width=0.25mm] (N-7) -- (N-16);
\draw[midnight, line width=0.25mm] (N-6) -- (N-16);
\draw[midnight, line width=0.25mm] (N-6) -- (N-11);
\draw[midnight, line width=0.25mm] (N-3) -- (N-16);
\draw[midnight, line width=0.25mm] (N-12) -- (N-6);
\draw[midnight, line width=0.25mm] (N-3) -- (N-9);

\node[yshift=-0.2cm, xshift=-0.04cm] at (N-14) {$0$};
\node[yshift=-0.11cm, xshift=-0.21cm] at (N-6) {$1$};
\node[yshift=-0.03cm, xshift=-0.21cm] at (N-13) {$2$};
\node[yshift=0.06cm, xshift=-0.25cm] at (N-5) {$3$};
\node[yshift=0.06cm, xshift=-0.25cm] at (N-12) {$4$};
\node[yshift=0.21cm, xshift=-0.16cm] at (N-4) {$5$};
\node[yshift=0.16cm, xshift=-0.25cm] at (N-11) {$6$};
\node[yshift=0.21cm, xshift=-0.04cm] at (N-3) {$7$};
\node[yshift=0.21cm, xshift=0.04cm] at (N-10) {$8$};
\node[yshift=0.16cm, xshift=0.25cm] at (N-2) {$9$};
\node[yshift=0.21cm, xshift=0.19cm] at (N-9) {$10$};
\node[yshift=0.06cm, xshift=0.25cm] at (N-1) {$11$};
\node[yshift=0.06cm, xshift=0.25cm] at (N-16) {$12$};
\node[yshift=-0.03cm, xshift=0.23cm] at (N-8) {$13$};
\node[yshift=-0.11cm, xshift=0.23cm] at (N-15) {$14$};
\node[yshift=-0.2cm, xshift=0.04cm]  at (N-7) {$15$};
\end{tikzpicture}
\end{center}
\caption{A decomposition of a hexadecagon into five odd and two even quadrangles.}
\label{FigureHexadecagonQuadrangulation}
 \end{figure}

For an alternating sequence $\P$ consider the moduli space $\mathfrak{M}_{\P}$ of configurations of $2n+2$ points ${x_{p_0},x_{p_1},\dots, x_{p_{2n+1}}}$ 
in $\mathbb{P}^1.$ We get a contravariant functor $\P\mapsto \mathfrak{M}_\P$ from $\underline{\textup{Alt}}$ to the category of algebraic varieties: morphisms in $\underline{\textup{Alt}}$ are sent to  the corresponding forgetful maps between configuration spaces. We define a regular function on $\mathfrak{M}_\P,$ which we call the cross-ratio:
\be \label{FormulaDefinitionCrossRatio}
\textup{cr}(\P)= \begin{cases} 
      \prod_{i=1}^{n} [x_{p_0},x_{p_{2i-1}},x_{p_{2i}},x_{p_{2i+1}}] & \text{\ if \ } \P \text{\ is even,}  \\
     \prod_{i=1}^{n} [x_{p_0},x_{p_{2i-1}},x_{p_{2i}},x_{p_{2i+1}}]^{-1}& \text{\ if \ } \P \text{\ is odd}.  \\
   \end{cases}
\ee
Abusing the notation, we use the same symbol  $\textup{cr}(\P)$ to denote the pullback of the cross-ratio function to $\mathfrak{M}_{\mathrm{P'}}$ for an arbitrary alternating polygon $\mathrm{P'}$ containing $\P.$  For a decomposition $\{\P_1,\dots,\P_k\}\in \mathcal{D}(\P)$ we have
\be\label{FormulaCrossRatioDecomposition}
\textup{cr}(\P)=\prod_{i=1}^k \textup{cr}(\P_i).
\ee 
To make Definition \ref{FormulaDefinitionCrossRatio} more explicit, assume that $\mathrm{P}=(0,1,\ldots, 2n+1)$ and $x_i\in \C \subseteq \mathbb{P}^1.$ Then 
\[
\textup{cr}(\P)=\prod_{i=1}^{n} [x_{0},x_{2i-1},x_{2i},x_{2i+1}]=(-1)^{n-1}\frac{(x_0-x_1)(x_2-x_3)\dots (x_{2n}-x_{2n+1})}{(x_1-x_2)(x_3-x_4)\dots (x_{2n+1}-x_{0})\: \: }.
\]
From this equality one can easily deduce (\ref{FormulaCrossRatioDecomposition}).

For an alternating polygon $\P$ we denote  by $\F_{\P}$ and $\H_\P$ the  Hopf algebras $\F_{S}$ and $\H_{S}$ defined in  \S \ref{SectionQuasiShuffle}, where $S=\overline{\mathfrak{M}}_\P$ is the Deligne-Mumford compactification of $\mathfrak{M}_\P.$ Similarly, we denote $\L_S$ by $\L_\P.$

\begin{definition} Let $\P$ be an alternating polygon. The algebra of coinvariants 
\[
\F^\H_\P=\textup{Ker}(\widetilde{\Delta}^{\H\F} \colon \F_\P\lra \H_\P\otimes \F_\P)
\] 
is called the algebra of formal quadrangular polylogarithms.
\end{definition}

A priori it is not clear how to construct elements in $\F_\P^\H$ of weight greater than $1.$ In \S \ref{SectionFormalCluster} we  construct an element $ \textup{T}_\P$ in $\F^\H_{\P}$ of weight $n$ for an arbitrary alternating $(2n+2)$-gon~$\P.$  
\begin{example} For $\P=(0,1,2,3,4,5)$  the element
\begin{align*}
\textup{T}_\P&=[\textup{cr}(0, 3, 4, 5),1|\textup{cr}(0, 1, 2, 3),1]\\
&\quad-[\textup{cr}(0, 1, 4, 5),1|\textup{cr}(1, 2, 3, 4),1]+[\textup{cr}(0, 1, 2, 5),1|\textup{cr}(2, 3, 4, 5),1]
\end{align*}
lies in  $\F^\H_\P.$ Indeed, we have
\begin{align*}
\widetilde{\Delta}^{\H\F}\textup{T}_\P &=\widetilde{\Delta}^{\H\F}[\textup{cr}(0, 3, 4, 5),1|\textup{cr}(0, 1, 2, 3),1]-\\
&\quad \widetilde{\Delta}^{\H\F}[\textup{cr}(0, 1, 4, 5),1|\textup{cr}(1, 2, 3, 4),1]+\widetilde{\Delta}^{\H\F}[\textup{cr}(0, 1, 2, 5),1|\textup{cr}(2, 3, 4, 5),1]\\
&=(-\textup{I}^{\H}(1;\textup{cr}(0, 3, 4, 5);\textup{cr}(0,1,2, 3, 4, 5)))\otimes [\textup{cr}(0,1,2, 3, 4, 5),1]\\
&\quad-(-\textup{I}^{\H}(1;\textup{cr}(0, 1, 4, 5);\textup{cr}(0,1,2, 3, 4, 5)))\otimes [\textup{cr}(0,1,2, 3, 4, 5),1]\\
&\quad+(-\textup{I}^{\H}(1;\textup{cr}(0, 1, 2, 5);\textup{cr}(0,1,2, 3, 4, 5)))\otimes [\textup{cr}(0,1,2, 3, 4, 5),1]\\
&=\log^\H \left (\frac{(1-[x_0,x_3,x_4,x_5]^{-1})(1-[x_1,x_2,x_3,x_4]^{-1})(1-[x_0,x_1,x_2,x_5]^{-1})}{(1-[x_0,x_1,x_2,x_3])(1-[x_0,x_1,x_4,x_5]^{-1})(1-[x_2,x_3,x_4,x_5])}\right ) \\
&\quad \otimes [\textup{cr}(0,1,2, 3, 4, 5),1]\\
&=0.
\end{align*}
\end{example}

\subsection{Arborification map}\label{Arborification}

Recall the definition of the Hopf algebra of rooted trees, introduced by Connes and Kreimer in \cite{CK99} in order to clarify the renormalization procedure in quantum field theory. 
  A rooted tree is a finite connected graph without cycles with a special vertex, called the root. A decorated rooted tree is a rooted tree with vertices labeled by elements of some set $\mathcal{A}.$ Consider  a free commutative unitary $\Q$-algebra $\mathcal{T}     ^{\mathcal{A}}$  generated by isomorphism classes of decorated rooted trees. A $\Q$-basis of this algebra is given by decorated rooted forests. The product in $\mathcal{T}^{\mathcal{A}}$ is given by the concatenation of rooted forests; the unit $1$ is represented by the empty forest. The weight of a forest is the number of vertices in it; $\mathcal{T}^{\mathcal{A}}$ is graded by weight. We define the bialgebra structure on $\mathcal{T}^{\mathcal{A}}$ in the following way. A  cut $c$ of  a decorated rooted  tree $\mathrm{t}$ is a nonempty subset of the set of edges of  $\mathrm{t}.$ A cut is called admissible if any shortest  path in the tree from a vertex to the root meets at most one edge in the cut. Denote by $\textup{Adm}(\mathrm{t})$ the set  of admissible cuts. After cutting all edges of $\mathrm{t}$ in $C$ we obtain a rooted forest. Its connected component containing the root of $\mathrm{t}$ is denoted $R^c(\mathrm{t}).$ The product of the remaining connected components is denoted $P^c(\mathrm{t}).$ Then the coproduct 
\[
\Delta^{\mathcal{T}}\colon \mathcal{T}^{\mathcal{A}} \lra \mathcal{T}^{\mathcal{A}}\otimes \mathcal{T}^{\mathcal{A}}
\] 
is given by the formula
\[
\Delta^{\mathcal{T}} (\mathrm{t})=1\otimes \mathrm{t}+\mathrm{t} \otimes 1 +\sum_{c \: \in \textup{Adm}(\mathrm{t})} R^c(\mathrm{t})\otimes P^c(\mathrm{t}).
\]

For $a\in \mathcal{A}$ we define  the ``grafting'' operator 
\[
B_{a}^{+}\colon \mathcal{T}^{\mathcal{A}}\lra \mathcal{T}^{\mathcal{A}}
\] 
associating to every labeled rooted forest $\mathrm{t}_1 \dots \mathrm{t}_n$ a tree obtained by grafting the roots of  $\mathrm{t}_1,\dots,\mathrm{t}_n$ on the common new root labeled by $a \in \mathcal{A}.$  This operator satisfies the following equation: for $\mathrm{t}\in \mathcal{T}^{\mathcal{A}}$ we have
\[
\Delta^{\mathcal{T}} \circ B_{a}^{+}(\mathrm{t})=(B_{a}^{+}\otimes \textup{id} )\circ \Delta^{\mathcal{T}} (\mathrm{t}) +1 \otimes B_{a}^{+}(\mathrm{t}).
\]
In other words, $B_a^+$ is a family of Hochschild $1$-cocycles on the Hopf algebra $\mathcal{T}^{\mathcal{A}}.$

The Connes-Kreimer Hopf algebra has the following universal property. Let $(H,\Delta, m)$ be a Hopf algebra with a collection of Hochschild $1$-cocycles $L_{a}$ for  $a\in \mathcal{A},$  namely linear maps $L_{a}\colon H\lra H$ such that
\[
\Delta L_{a}  = ( L_a\otimes  \textup{id}) \Delta +  1 \otimes L_a\text{\ \ for \ \ } a\in \mathcal{A}.
\]
Then there exists a unique Hopf  algebra homomorphism $\phi \colon \mathcal{T}^\mathcal{A} \lra H$ such that 
\[
\phi \circ B^+_a=L_a \circ \phi.
\]

Let $\textup{QSh}_\mathcal{A}$ be the quasi-shuffle on an alphabet $\mathcal{A},$ defined in \S \ref{SectionQuasiShuffle}. Recall that the alphabet $\mathcal{A}$ has the structure of a commutative semigroup. One can easily check that operators $L_a(\omega)=-a\omega$ and $L_a(\omega)=a\omega+a\cdot \omega$ are Hochschild $1$-cocycles on $\textup{QSh}_\mathcal{A}.$  

Now, assume that a ``parity'' map $p\colon \mathcal{A}\lra \{0,1\}$ is given. We call vertices labeled by $a\in \mathcal{A}$ {\it even} if $p(a)=0$ and {\it odd} if $p(a)=1.$ From the universal property of the Connes-Kreimer Hopf algebra we have the following.

\begin{proposition}\label{PropositionArborification} There exists a unique Hopf algebra homomorphism
$
\textup{Arb}\colon \mathcal{T}^\mathcal{A}\lra \textup{QSh}_\mathcal{A}
$	
such that 
\be \label{FormulaArborification}
\textup{Arb}(B_{a}^+(\mathrm{t}))=
\begin{cases}
-a\textup{Arb}(\mathrm{t})& \text{\ if \ } p(a)=0,\\
\ \: \,  a\textup{Arb}(\mathrm{t})+a\cdot\textup{Arb}(\mathrm{t})& \text{\ if \ } p(a)=1.\\
\end{cases}
\ee
\begin{example}
Let $\mathcal{A}=\{a_1, a_2, a_3\}$ be an alphabet with $p(a_1)=1, \ p(a_2)=p(a_3)=0.$ Consider a tree $\mathrm{t}$ with a root labelled by $a_1$ and two leaves labelled by $a_2$ and $a_3.$  Then 
\[
\mathrm{t}=B_{a_1}^+\left(B_{a_2}^+(1)B_{a_3}^+(1)\right).
\]
By (\ref{FormulaArborification}) we have 
$\textup{Arb}\left(B_{a_2}^+(1)\right)=-[a_2]$ and $\textup{Arb}\left(B_{a_3}^+(1)\right)=-[a_3],$
so 
\[
\textup{Arb}\left(B_{a_2}^+(1)B_{a_3}^+(1)\right)=(-[a_2])\star (-[a_3])=[a_2|a_3]+[a_3|a_2]+[a_2\cdot a_3].
\]
\end{example}
It follows that
\[
\textup{Arb}(\mathrm{t})=[a_1|a_2|a_3]+[a_1|a_3|a_2]+[a_1|a_2\cdot a_3]+[a_1\cdot a_2|a_3]+[a_1\cdot a_3|a_2]+[a_1\cdot a_2\cdot a_3].
\]

\end{proposition}

\subsection{Formal quadrangular polylogarithms}\label{SectionFormalCluster}

For an alternating polygon 
\[
\P=(p_0,p_1,\dots,p_{2n+1})
\]
consider the Hopf algebra $\mathcal{T}^{\P}$ of rooted trees labeled by elements 
\[
[\textup{cr}(p_{i_0},p_{i_1},p_{i_2},p_{i_3}),1]\in \F_\P.
\] 
To a quadrangulation $Q=\{\mathrm{Q}_1,\dots,\mathrm{Q}_n\}\in \mathcal{Q}(\P)$ we associate a tree $\mathrm{t}_{Q}\in \mathcal{T}^{\P}$ in the following way. The tree is the dual tree of the quadrangulation (see Figure \ref{FigureHexadecagonTree}); the quadrangle adjacent to the side $p_0, p_{2n+1}$ is the root. Each vertex of the tree is labeled by a pair $[\textup{cr}(\mathrm{Q}_i),1]\in \F_\P,$ where  $\mathrm{Q}_i$ is the corresponding quadrangle of the quadrangulation. A parity of the vertex $[\textup{cr}(\mathrm{Q}_i),1]$ is defined as the parity of the quadrangle $\mathrm{Q}_i.$ We define an element $\mathrm{t}_{\P}$ of $\mathcal{T}^{\P}$ as the sum of trees corresponding to quadrangulations of $\P$:
 \be \label{FormulaSumTrees}
  \mathrm{t}_{\P}=\sum_{Q\in\mathcal{Q}(\P)} \mathrm{t}_{Q}.
 \ee
 
 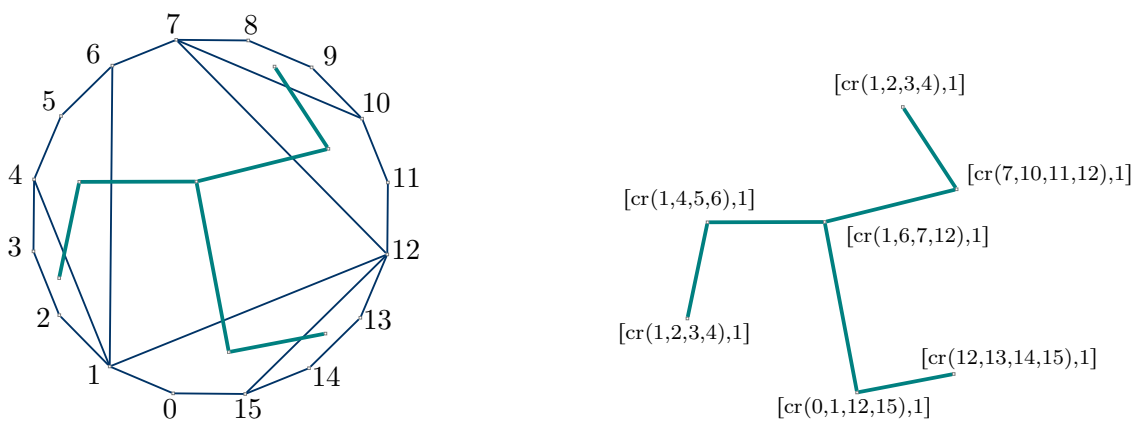
\begin{figure}
  \centering
\begin{subfigure}[b]{0.45\textwidth}
\begin{tikzpicture}[transform shape]
  \foreach \number in {1,...,8}{
        \mycount=\number
        \advance\mycount by -1
  \multiply\mycount by 45
        \advance\mycount by 11.25
      \node[draw, very thin, color=gray,inner sep=0.001cm] (N-\number) at (\the\mycount:2.4cm) {};
    }
  \foreach \number in {9,...,16}{
        \mycount=\number
        \advance\mycount by -1
  \multiply\mycount by 45
        \advance\mycount by 33.75
      \node[draw, very thin, color=gray,inner sep=0.001cm](N-\number) at (\the\mycount:2.4cm) {};
    }
\draw[midnight, line width=0.25mm] (N-1) -- (N-9);
\draw[midnight, line width=0.25mm] (N-2) -- (N-10);
\draw[midnight, line width=0.25mm] (N-3) -- (N-11);
\draw[midnight, line width=0.25mm] (N-4) -- (N-12);
\draw[midnight, line width=0.25mm] (N-5) -- (N-13);
\draw[midnight, line width=0.25mm] (N-6) -- (N-14);
\draw[midnight, line width=0.25mm] (N-7) -- (N-15);
\draw[midnight, line width=0.25mm] (N-8) -- (N-16);
\draw[midnight, line width=0.25mm] (N-1) -- (N-16);
\draw[midnight, line width=0.25mm] (N-2) -- (N-9);
\draw[midnight, line width=0.25mm] (N-3) -- (N-10);
\draw[midnight, line width=0.25mm] (N-4) -- (N-11);
\draw[midnight, line width=0.25mm] (N-5) -- (N-12);
\draw[midnight, line width=0.25mm] (N-6) -- (N-13);
\draw[midnight, line width=0.25mm] (N-7) -- (N-14);
\draw[midnight, line width=0.25mm] (N-8) -- (N-15);

\draw[midnight, line width=0.25mm] (N-7) -- (N-16);
\draw[midnight, line width=0.25mm] (N-6) -- (N-16);
\draw[midnight, line width=0.25mm] (N-6) -- (N-11);
\draw[midnight, line width=0.25mm] (N-3) -- (N-16);
\draw[midnight, line width=0.25mm] (N-12) -- (N-6);
\draw[midnight, line width=0.25mm] (N-3) -- (N-9);

\node[yshift=-0.2cm, xshift=-0.04cm] at (N-14) {$0$};
\node[yshift=-0.11cm, xshift=-0.21cm] at (N-6) {$1$};
\node[yshift=-0.03cm, xshift=-0.21cm] at (N-13) {$2$};
\node[yshift=0.06cm, xshift=-0.25cm] at (N-5) {$3$};
\node[yshift=0.06cm, xshift=-0.25cm] at (N-12) {$4$};
\node[yshift=0.21cm, xshift=-0.16cm] at (N-4) {$5$};
\node[yshift=0.16cm, xshift=-0.25cm] at (N-11) {$6$};
\node[yshift=0.21cm, xshift=-0.04cm] at (N-3) {$7$};
\node[yshift=0.21cm, xshift=0.04cm] at (N-10) {$8$};
\node[yshift=0.16cm, xshift=0.25cm] at (N-2) {$9$};
\node[yshift=0.21cm, xshift=0.19cm] at (N-9) {$10$};
\node[yshift=0.06cm, xshift=0.25cm] at (N-1) {$11$};
\node[yshift=0.06cm, xshift=0.25cm] at (N-16) {$12$};
\node[yshift=-0.03cm, xshift=0.23cm] at (N-8) {$13$};
\node[yshift=-0.11cm, xshift=0.23cm] at (N-15) {$14$};
\node[yshift=-0.2cm, xshift=0.04cm]  at (N-7) {$15$};

\node[draw, very thin, color=gray,inner sep=0.001cm] (N-x1) at (barycentric cs:N-6=1,N-13=1,N-5=1,N-12=1) {};
\node[draw, very thin, color=gray,inner sep=0.001cm] (N-x2) at (barycentric cs:N-6=1,N-11=1,N-12=1,N-4=1) {};
\node[draw, very thin, color=gray,inner sep=0.001cm] (N-x3) at (barycentric cs:N-11=1,N-3=1,N-16=1,N-6=1) {};
\node[draw, very thin, color=gray,inner sep=0.001cm] (N-x4) at (barycentric cs:N-3=1,N-9=1,N-1=1,N-16=1) {};
\node[draw, very thin, color=gray,inner sep=0.001cm] (N-x5) at (barycentric cs:N-3=1,N-10=1,N-2=1,N-9=1) {};
\node[draw, very thin, color=gray,inner sep=0.001cm] (N-x6) at (barycentric cs:N-6=1,N-16=1,N-7=1,N-14=1) {};
\node[draw, very thin, color=gray,inner sep=0.001cm] (N-x7) at (barycentric cs:N-16=1,N-8=1,N-15=1,N-7=1) {};
\draw[color=teal,line width=0.5mm] (N-x1) -- (N-x2);
\draw[color=teal,line width=0.5mm] (N-x2) -- (N-x3);
\draw[color=teal,line width=0.5mm] (N-x3) -- (N-x4);
\draw[color=teal,line width=0.5mm] (N-x4) -- (N-x5);
\draw[color=teal,line width=0.5mm] (N-x3) -- (N-x6);
\draw[color=teal,line width=0.5mm] (N-x6) -- (N-x7);
\end{tikzpicture}
\end{subfigure}
\hfill
\begin{subfigure}[b]{0.45\textwidth}
\begin{tikzpicture}[transform shape]
  \foreach \number in {1,...,8}{
        \mycount=\number
        \advance\mycount by -1
  \multiply\mycount by 45
        \advance\mycount by 11.25
      \coordinate (N-\number) at (\the\mycount:2.4cm) {};
    }
  \foreach \number in {9,...,16}{
        \mycount=\number
        \advance\mycount by -1
  \multiply\mycount by 45
        \advance\mycount by 33.75
      \coordinate (N-\number) at (\the\mycount:2.4cm) {};
    }

\node[draw, very thin, color=gray,inner sep=0.001cm] (N-x1) at (barycentric cs:N-6=1,N-13=1,N-5=1,N-12=1) {};
\node[draw, very thin, color=gray,inner sep=0.001cm] (N-x2) at (barycentric cs:N-6=1,N-11=1,N-12=1,N-4=1) {};
\node[draw, very thin, color=gray,inner sep=0.001cm] (N-x3) at (barycentric cs:N-11=1,N-3=1,N-16=1,N-6=1) {};
\node[draw, very thin, color=gray,inner sep=0.001cm] (N-x4) at (barycentric cs:N-3=1,N-9=1,N-1=1,N-16=1) {};
\node[draw, very thin, color=gray,inner sep=0.001cm] (N-x5) at (barycentric cs:N-3=1,N-10=1,N-2=1,N-9=1) {};
\node[draw, very thin, color=gray,inner sep=0.001cm] (N-x6) at (barycentric cs:N-6=1,N-16=1,N-7=1,N-14=1) {};
\node[draw, very thin, color=gray,inner sep=0.001cm] (N-x7) at (barycentric cs:N-16=1,N-8=1,N-15=1,N-7=1) {};
\draw[color=teal,line width=0.5mm] (N-x1) -- (N-x2);
\draw[color=teal,line width=0.5mm] (N-x2) -- (N-x3);
\draw[color=teal,line width=0.5mm] (N-x3) -- (N-x4);
\draw[color=teal,line width=0.5mm] (N-x4) -- (N-x5);
\draw[color=teal,line width=0.5mm] (N-x3) -- (N-x6);
\draw[color=teal,line width=0.5mm] (N-x6) -- (N-x7);

\node[yshift=-0.2cm, xshift=-0.04cm] at (N-x1) {${\scriptstyle [\textup{cr}(1,2,3,4),1]}$};
\node[yshift=0.3cm, xshift=-0.24cm] at (N-x2) {${\scriptstyle [\textup{cr}(1,4,5,6),1]}$};
\node[yshift=-0.2cm, xshift=1.24cm] at (N-x3) {${\scriptstyle [\textup{cr}(1,6,7,12),1]}$};
\node[yshift=0.2cm, xshift=1.24cm] at (N-x4) {${\scriptstyle [\textup{cr}(7,10,11,12),1]}$};
\node[yshift=0.3cm, xshift=-0.04cm] at (N-x5) {${\scriptstyle [\textup{cr}(1,2,3,4),1]}$};
\node[yshift=-0.2cm, xshift=-0.04cm] at (N-x6) {${\scriptstyle [\textup{cr}(0,1,12,15),1]}$};
\node[yshift=0.2cm, xshift=0.74cm] at (N-x7) {${\scriptstyle [\textup{cr}(12,13,14,15),1]}$};
\end{tikzpicture}

 \end{subfigure}

\caption{Quadrangulation of a hexadecagon and its dual tree. Vertex $[\textup{cr}(0,1,12,15),1]$ is the root.}
\label{FigureHexadecagonTree}
 \end{figure}

 \begin{definition}
Let $\P$ be  an alternating polygon.  The formal quadrangular polylogarithm $\textup{T}_\P$ is defined by the formula
\[
 \textup{T}_\P=\textup{Arb}(\mathrm{t}_{\P})\in \F_\P.
\]	
 \end{definition}

In \S \ref{SectionQuadrangulationFormula} we show that $\textup{T}_\P$ is an element of the  algebra of formal  quadrangular polylogarithms $\F_\P^\H.$

\begin{example}
We have the following equalities for $\P=(p_0,p_1,p_2,p_3):$	
\be \label{FormulaFormalPolylogWeight1}
 \textup{T}_\P=
\begin{cases}
-[\textup{cr}(p_0,p_1,p_2,p_3),1] 	&\text{\ if \ } \P \text{\ is even,} \\
 \ \: \,  [\textup{cr}(p_0,p_1,p_2,p_3),1] &\text{\ if \ }  \P \text{\ is odd.} \\
\end{cases}
\ee
\end{example}

 Consider the set $\textup{alt}(\P)$ of alternating subpolygons $(p_{i_0},p_{i_1},\dots, p_{i_{2r+1}})$ of $\P$ with $i_0=0$ and $i_{2r+1}=2n+1$. For a polygon  $\mathrm{S}\in \textup{alt}(\P)$ denote by $\mathrm{S}^0,\dots, \mathrm{S}^{2r}$ the subpolygons, which are obtained by taking closures of connected components of the set $\P\fgebackslash \mathrm{S}.$ It is easy to see that 
 \[
 \P=\mathrm{S}\sqcup \left ( \bigsqcup_{i=0}^{2r} \mathrm{S}^i \right ).
 \]
   
 \begin{proposition}\label{PropositionCoproduct} For an alternating polygon $\P$ the following equality holds
 	\be \label{FormulaCoproductFormalPolylogarithm}
 	 \Delta^{\mathcal{FF}}\textup{T}_{\P}= \sum_{\mathrm{S}\in\textup{alt}(\P)} \textup{T}_{\mathrm{S}}\otimes \left (\prod_{i=0}^{2r}\textup{T}_{\mathrm{S}^i} \right).
 	\ee
 \end{proposition}

 \begin{proof}
The arborification map commutes with the product and the coproduct by Proposition \ref{PropositionArborification}, so it suffices to show that the following equality holds:
\[
 	  \Delta^{\mathcal{T}}\mathrm{t}_{\P}= \sum_{\mathrm{S}\in\textup{alt}(\P)}\mathrm{t}_{\mathrm{S}}\otimes \left (\prod_{i=0}^{2r}\mathrm{t}_{\mathrm{S}^i}\right).
\]	  
It is easy to see that the coproduct of the sum of trees $\mathrm{t}_{Q}$ corresponding to quadrangulations $Q\in \mathcal{Q}(\P)$ such that 
\[
\{\mathrm{S},\mathrm{S}^0,\dots,S^{2r}\}\leq Q
\]
is equal to  $\mathrm{t}_{\mathrm{S}}\otimes \left ( \mathrm{t}_{\mathrm{S}^0}\cdot \ldots \cdot \mathrm{t}_{\mathrm{S}^{2r}} \right).$ The conclusion follows from this.
 \end{proof}

Next we prove a simple inductive formula for $\textup{T}_{\P}.$ It suffices to compute $\textup{T}_{\P}$ for $\P=(0,1,\dots,2n+1)$ and $\P=(1,2,\dots,2n+2).$ 

\begin{lemma}\label{LemmaRecursiveTP}
For an even polygon $\P=(0,1,\dots,2n+1)$ we have
\be \label{FormulaRecursiveEven}
\begin{split}
&\textup{T}_{\P}=\sum_{\substack{0< i< j<2n+1\\ i \text{ is odd}\\ j \text{ is even}}} \textup{T}_{(0,i,j,2n+1)}\Bigl(\textup{T}_{(0,\dots,i)}\star\textup{T}_{(i,\dots,j)}\star\textup{T}_{(j,\dots,2n+1)} \Bigl).\\
\end{split}
\ee	
For an  odd polygon  $\P=(1,2,\dots,2n+2)$  we have
\begin{align}\label{FormulaRecursiveOdd}
\begin{split}
\textup{T}_{\P}&=\sum_{\substack{1< i< j<2n+2\\ i \text{ is even}\\ j \text{ is odd}}}
 \Bigl(\textup{T}_{(1,i,j,2n+2)}
\left(\textup{T}_{(1,\dots,i)}\star  \textup{T}_{(i,\dots,j)} \star \textup{T}_{( j,\dots,2n+2)} \right)\\
&\quad+\textup{T}_{(1,i,j,2n+2)}\cdot
\left(\textup{T}_{(1,\dots,i)}\star  \textup{T}_{(i,\dots,j)} \star \textup{T}_{( j,\dots,2n+2)} \right) \Bigr).\\
\end{split}
\end{align}	
\end{lemma}

\begin{proof}
We consider the even case; the odd case is similar. Let $\left ( \mathrm{t}_\mathrm{P}\right )_{ij}$ be the sum of trees $\mathrm{t}_Q$ corresponding to quadrangulations $Q$ with the root $(0,i,j,2n+1).$  Then
\[
\mathrm{t}_\mathrm{P}=\sum_{\substack{0< i< j<2n+1\\ i \text{ is odd}\\ j \text{ is even}}}\left ( \mathrm{t}_\mathrm{P}\right )_{ij}.
\]
Let $a=[(0,i,j,2n+1),1].$ Then we have
\[
\left ( \mathrm{t}_\mathrm{P}\right )_{ij}=B^+_{a}(\mathrm{t}_{(0,\dots,i)}\mathrm{t}_{(i,\dots,j)}\mathrm{t}_{(j,\dots,2n+1)})
\]
This implies (\ref{FormulaRecursiveEven}).
\end{proof}

\begin{example}For  an even polygon $\P=(0,1,2,3,4,5)$  we  have
\begin{align*}
\textup{T}_\P&=[\textup{cr}(0,1,2,5),1|\textup{cr}(2,3,4,5),1]\\
&\quad-[\textup{cr}(0,1,4,5),1|\textup{cr}(1,2,3,4),1]+[\textup{cr}(0,3,4,5),1|\textup{cr}(0,1,2,3),1].\\
\end{align*}
For  an odd polygon $\P=(1,2,3,4,5,6)$  we  have
\begin{align*}
\textup{T}_\P&=[\textup{cr}(1,2,3,6),1|\textup{cr}(3,4,5,6),1]-[\textup{cr}(1,2,5,6),1|\textup{cr}(2,3,4,5),1]\\
&\quad +[\textup{cr}(1,4,5,6),1|\textup{cr}(1,2,3,4),1]+[\textup{cr}(1,2,3,4,5,6),2].\\
\end{align*}
\end{example}

\subsection{Principal coefficient of a formal quadrangular polylogarithm}
Our goal is to show that formal quadrangular polylogarithms lie in the algebra of coinvariants $\F_\P^\H.$ In this section, we prove that the principal coefficient $\pi_1(\textup{T}_\P)\in\H_\P$ defined in \S \ref{SectionProjectionToPrimitiveElements} is equal to zero.

 For an element $x\in \F_\P$ denote by  $\textup{pr}_n(x)$ the projection of $x$ to the subspace of $\F_\P$ spanned by elements $[*|n].$ Next, let $\textup{pr}_{n-1,1}$ be the projection from  $\F_\P$ to its subspace spanned by elements 
 \be\label{FormulaWordsType1}
 [*,n-1|*,1]
 \ee
  and let $\textup{pr}_{1,n-1}$ be the projection from  $\F$  to its subspace spanned by elements  
 \be\label{FormulaWordsType2}
 [*,1|*,n-1].
 \ee 
To compute $\pi_1(\textup{T}_\P)$, we first compute  the values of these projections on formal quadrangular polylogarithms.

\begin{lemma}\label{LemmaMapPiZero}  For a  $(2n+2)$-alternating  polygon $\P$ with $n\geq 2$  we have 
\[
\textup{pr}_n(\textup{T}_\P)=
\begin{cases}
0 & \text{if\ } 	\P\text{\ is even,}\\
[\textup{cr}(\P),n] & \text{if\ } 	\P \text{\ is odd.}
\end{cases}
\]
\end{lemma}
\begin{proof}
First, we consider the case of an even $\P$. From  (\ref{FormulaRecursiveEven}) it follows that each word in $\textup{T}_\P$ is obtained by concatenation of a word of length one and a word of length at least one. Thus the  concatenation has length at least two, so $\textup{pr}_n(\textup{T}_\P)=0.$

For an odd polygon $\P=(1,\dots,2n+2)$ we proceed by induction. Only terms in the second sum in formula (\ref{FormulaRecursiveOdd}) may contribute to $\textup{pr}_n(\textup{T}_\P)$ by the same reason as in the even case. Terms with $j=i+1$ give a contribution $[\textup{cr}(\P),n]$ and terms with $j=i+3$ give a contribution $-[\textup{cr}(\P),n].$ It is easy to see that words coming from other terms have length at least two, so do not contribute to the projection $
 \textup{pr}_n(\textup{T}_\P)$. Thus
\[
 \textup{pr}_n(\textup{T}_\P)=n[\textup{cr}(\P),n]+(n-1)(-[\textup{cr}(\P),n])=[\textup{cr}(\P),n].
\]
\end{proof}

Next, notice that for all words $[\textup{cr}_1,n_1|\dots|\textup{cr}_k,n_k]$ in $\textup{T}_\P$  we have $\prod_{i=1}^k \textup{cr}_i=\textup{cr}(\P)$ and $\sum_{i=1}^k n_i=n,$ so in words of type (\ref{FormulaWordsType1}) and (\ref{FormulaWordsType2}) each side determines the other.

\begin{lemma} \label{LemmaMapPiEven} The following equalities hold for an even polygon 
$\P=(0,1,\dots, 2n+1)$ for $n\geq 3:$
\begin{align*}
\textup{pr}_{1,n-1}(\textup{T}_\P)&=[\textup{cr}(0,1,2n-2,2n+1),1|*]-[\textup{cr}(0,1,2n,2n+1),1|*]\\
&\quad-[\textup{cr}(0,3,2n-2,2n+1),1|*]+[\textup{cr}(0,3,2n,2n+1),1|*],\\
\textup{pr}_{n-1,1}(\textup{T}_\P)&=0.\\
\end{align*}
\end{lemma}
\begin{proof}
	The statement follows from (\ref{FormulaRecursiveEven}) and Lemma \ref{LemmaMapPiZero}.
\end{proof}

\begin{lemma} \label{LemmaMapPiOdd} The following equalities hold for an odd polygon $\P=(1,2,\dots,2n+2)$ for $n\geq 3$:
\[
\begin{split}
&\textup{pr}_{n-1,1}(\textup{T}_\P)=\sum_{i=0}^{2n-2}(-1)^i[*|\textup{cr}(i+1,i+2,i+3,i+4),1],\\
&\textup{pr}_{1,n-1}(\textup{T}_\P)=\sum_{i=1}^{n}[\textup{cr}(1,2i,2i+1,2n+2),1|*]-\sum_{i=1}^{n-1}[\textup{cr}(1,2i,2i+3,2n+2),1|*].
\end{split}
\]
\end{lemma}
\begin{proof}
	The statement follows from  (\ref{FormulaRecursiveOdd}) and Lemma \ref{LemmaMapPiZero}.
\end{proof}

\begin{corollary}\label{CorollaryVanishingPrimitive} For an alternating polygon $\P$ we have
$\pi_1(\textup{T}_\P)=0.$
\end{corollary}
\begin{proof}
Cases $n=1,2$ can be easily checked by hand. For an even  $\P=(0,1,\dots,2n+1)$ and $n\geq 3$ by Lemma \ref{LemmaMapPi} and Lemma \ref{LemmaMapPiEven}  we need to check that 
\begin{align*}
&\log^\H\left(1-\textup{cr}(0,1,2n-2,2n+1)^{-1}\right)-\log^\H\left(1-\textup{cr}(0,1,2n,2n+1)^{-1}\right)\\
-&\log^\H\left(1-\textup{cr}(0,3,2n-2,2n+1)^{-1}\right)+
\log^\H\left(1-\textup{cr}(0,3,2n,2n+1])^{-1}\right)=0,\\
\end{align*}
which follows from an equality 
\[
[x_{0},x_{2n-2},x_{2n+1},x_{1}][x_{0},x_{2n},x_{2n+1},x_{1}]^{-1}[x_{0},x_{2n-2},x_{2n+1},x_{3}]^{-1}[x_{0},x_{2n},x_{2n+1},x_{3}]=1.
\]
The proof for an odd polygon is similar though more tedious.
\end{proof}

\subsection{Quadrangulation formula}\label{SectionQuadrangulationFormula}

\begin{theorem}[Quadrangulation Formula]\label{TheoremFormalQF}
For an alternating polygon $\P$ the formal quadrangular polylogarithm $\textup{T}_\P$ lies in $\F^\H_\P.$ 
\end{theorem}
\begin{proof}
We need to show that $\widetilde{\Delta}^{\H\F}\textup{T}_\P=0.$ We proceed by induction on $n.$ If $n=1$ then $\P$ is a $4$-gon and the statement is obvious. By induction and Proposition \ref{PropositionCoproduct}  element  $\widetilde{\Delta}^{\F\F}\textup{T}_\P$ lies in $\F_\P^\H\otimes \F_\P^\H,$ so we have
\[
\Delta^{\H\mathcal{FF}}\Delta^{\F\F}(\textup{T}_\P)=
(T\otimes 1)(1\otimes \Delta^{\H\F}(\textup{T}_\P))+\Delta^{\H\F}(\textup{T}_\P)\otimes 1+1\otimes \widetilde{\Delta}^{\F\F}(\textup{T}_\P).
\]
On the other hand, we have
\[
\begin{split}
&(1\otimes \Delta^{\F\F})\Delta^{\H\F}(\textup{T}_\P)\\
=&(1\otimes \widetilde{\Delta}^{\F\F})\Delta^{\H\F}(\textup{T}_\P)+(T\otimes 1)(1\otimes \Delta^{\H\F}(\textup{T}_\P))+\Delta^{\H\F}(\textup{T}_\P)\otimes 1\\
=&(1\otimes \widetilde{\Delta}^{\F\F})\widetilde{\Delta}^{\H\F}(\textup{T}_\P)+(T\otimes 1)(1\otimes \Delta^{\H\F}(\textup{T}_\P))+\Delta^{\H\F}(\textup{T}_\P)\otimes 1+1\otimes \widetilde{\Delta}^{\F\F}(\textup{T}_\P).
\end{split}
\]

Lemma \ref{LemmaCoproductCompatibility} implies that
\[
\Delta^{\H\mathcal{FF}}\Delta^{\F\F}\left( \textup{T}_\P\right)=(1\otimes \Delta^{\F\F})\Delta^{\H\F}\left(\textup{T}_\P\right),
\]
thus
\[
(1\otimes \widetilde{\Delta}^{\F\F})\widetilde{\Delta}^{\H\F}\textup{T}_\P=0.
\] 
In other words, $\widetilde{\Delta}^{\H\F}\textup{T}_\P$ lies in $\H_\P\otimes \textup{Prim}(\F_\P).$

For all words $[\textup{cr}_1,n_1|\dots|\textup{cr}_k,n_k]$
appearing in $\textup{T}_\P$ we have  $\prod_{i=1}^k \textup{cr}_i=\textup{cr}(\P)$ and $\sum_{i=1}^k n_i=n.$ Thus  a primitive word appearing in $\widetilde{\Delta}^{\H\F}\textup{T}_\P$  equals  $[\textup{cr}(\P),i]$ for $1\leq i\leq n-1.$ It follows that there exist elements $h_1,\dots, h_{n-1}\in \H$ such that
\[
\widetilde{\Delta}^{\H\F}\textup{T}_\P=\sum_{i=1}^{n-1} h_{i}\otimes [\textup{cr}(\P),n-i].
\]
Let $r$ be the smallest index such that $h_r\neq 0.$ From Corollary \ref{CorollaryVanishingPrimitive} we know that  $\pi_1(\textup{T}_\P)=h_{1}=0,$ so $r>1.$

By Lemma \ref{LemmaCoassociativity} we have 
\[
(\widetilde{\Delta}^{\mathcal{HH}}\otimes 1)\widetilde{\Delta}^{\H\F}\textup{T}_\P=(1\otimes \widetilde{\Delta}^{\H\F})\widetilde{\Delta}^{\H\F}\textup{T}_\P,
\] 
so
\[
\sum_{i=1}^{n-1} \left ( (\widetilde{\Delta}^{\mathcal{HH}}h_{i})\otimes \left[\textup{cr}(\P), n-i \right]\right)=\sum_{i=1}^{n-1} \left( h_{i}\otimes \widetilde{\Delta}^{\mathcal{HF}}\left[\textup{cr}(\P), n-i\right]\right).
\]
Comparing coefficients in front of $[\textup{cr}(\P),n-r]$
we get that
\[
\widetilde{\Delta}^{\mathcal{HH}}(h_r)=\sum_{i=1}^{r-1} h_{r-i}\otimes \frac{(\log^{\H}(\textup{cr}(\P))^i}{i!} ,
\]
so since $h_1=\dots=h_{r-1}=0$ we have $\widetilde{\Delta}^{\mathcal{HH}}(h_r)=0.$ Consider the specialization to a point $p$ in the Deligne-Mumford compactification $\overline{\mathfrak{M}}_\P$ where points $x_{2i}$ and $x_{2i+1}$  collide for $0\leq i \leq n$.  For an alternating subpolygon  $\mathrm{Q}$ of $\P$ the function $\textup{cr}(\mathrm{Q})$ vanishes at $p.$  From (\ref{FormulaCoaction}) it is clear that $h_r$ is a linear combination of Hodge iterated integrals 
\[
\textup{I}^{\H}(0;1,\underbrace{0,\dots,0,m_1}_{d_1},\dots,\underbrace{0,\dots,0,m_1m_2\dots m_{s-1}}_{d_{s-1}},\underbrace{0,\dots,0;m_1m_2\dots m_{s}}_{d_s}),
\]
where $m_i$ are monomials in variables $\textup{cr}(\mathrm{Q})$ of degree at least $1$ for alternating subpolygons $\mathrm{Q}$ of $\P$. All such integrals specialize to zero at $p,$ so by the rigidity argument (\S \ref{SectionRigidity}) $h_r=0,$ which contradicts our assumption. This finishes the proof of the theorem. 
\end{proof}

In \S \ref{SectionFormalCluster} we defined a collection of alternating subpolygons $\textup{alt}(\P)$ of an alternating polygon $\P.$ By Theorem \ref{TheoremFormalQF} we can view the element $\textup{T}_\P$ as an element in $\F_\P^{\H}\subseteq\H \times \F_\P.$ 
\begin{corollary}
The following equality holds: 
\be 
 	  \Delta^{\H\times \F}\textup{T}_\P=\sum_{\mathrm{S}\in\textup{alt}(\P)} \textup{T}_{\mathrm{S}}\otimes \left (\prod_{i=0}^{2r}\textup{T}_{\mathrm{S}^i} \right) .
\ee
\end{corollary}
\begin{proof}
	The statement follows from Theorem \ref{TheoremFormalQF} and (\ref{FormulaCoinvariantsEmbedding}).
\end{proof}

\section{Quadrangular polylogarithms}\label{SectionQuadrangulaPolylogs}

\subsection{The formula for quadrangular polylogarithms via Hodge correlators}\label{SectionQuadrangularCorr}
For $n\geq 0$ consider a collection of points $x_0,\dots,x_{2n+1}\in \C.$ For any $k\geq 0$ we define the quadrangular polylogarithm $\textup{QLi}_{n,k}$ as a certain element of the Lie coalgebra of framed mixed Hodge-Tate structures of weight $n+k$
\[
\textup{QLi}_{n,k}(x_0,\dots, x_{2n+1})\in \mathcal{L}_{n+k}.
\]
Quadrangular polylogarithm  is defined by an explicit formula, see Definition \ref{DefClusterPolylog}. This definition looks very ad hoc; more naturally, quadrangular polylogarithm can be defined inductively as the unique element with the coproduct given by formula (\ref{ClusterPolylogarithmCoproduct}).

Consider a set  $\mathcal{C}_{n,k}$ of all nondecreasing sequences $\bar{s}=(i_0,\dots,i_{n+k})$ of indices 
\[
0\leq i_0\leq i_1\leq \dots \leq i_{n+k}\leq 2n+1
\]
such that every even number $1\leq s\leq 2n+1$ appears in the sequence $s$ at most once and the sequence $s$ contains at least one element in each pair $\{2i,2i+1\}$ for $0\leq i\leq n.$ For a sequence $\bar{s}\in \mathcal{C}_{n,k}$ we define 
\[
\textup{sign}(\bar{s})=\begin{cases}
-1 & \text{\ if\ } \bar{s} \text{\ contains an odd number of even elements,}\\
\ \ 1 & \text{\ if\ }  \bar{s} \text{\ contains an even number of even elements.}\\
\end{cases}
\] 

\begin{example}
We have 	
\[
\mathcal{C}_{1,0}=\{(0,2),\ (0,3),\ (1,2),\ (1,3)\}
\]
and 
\[
\mathcal{C}_{1,1}=\{(0,1,2),\ (0,2,3),\ (0,1,3),\ (0,3,3),\ (1,2,3),\ (1,1,2),\ (1,1,3),\ (1,3,3)\}.
\]
\end{example}

\begin{definition}
\label{DefClusterPolylog}
For  $x_0, \dots, x_{2n+1}\in\C$ we define the quadrangular polylogarithm of weight $n+k$ by the formula 
\be\label{FormulaClusterViaCorrelators}
\begin{split}
&\textup{QLi}_{n,k}(x_0,\dots, x_{2n+1})=(-1)^{n+1}\sum_{\bar{s}\in\mathcal{C}_{n,k}} \textup{sign}(\bar{s})\textup{Cor}(x_{i_0},\dots,x_{i_{n+k}})\in \mathcal{L}_{n+k}.\\
\end{split}
\ee
\end{definition}

\begin{example}
	 In weight $1$ we have 
\[
\textup{QLi}_{0,1}(x_0,x_1)=\textup{Cor}(x_0,x_1)=\log^{\mathcal{L}}(x_0-x_1)\\
\]
and
\begin{align*}
\textup{QLi}_{1,0}(x_0,x_1,x_2, x_{3})&=\textup{Cor}(x_0,x_2)-\textup{Cor}(x_0,x_3)-\textup{Cor}(x_1,x_2)+\textup{Cor}(x_1,x_3)\\
	&=\log^{\mathcal{L}}(1-[x_0,x_1,x_2, x_{3}])=-\Li_1([x_0,x_1,x_2,x_3]).\\
\end{align*}
\end{example}

Here is a more interesting example; the computation below uses the five-term relation~(\ref{FormulaFiveTerm}).

\begin{example}
For $n=1,k=1$ we recover the  dilogarithm:
\[
\begin{split}
	&\textup{QLi}_{1,1}(x_0,x_1,x_2, x_{3})=\textup{Cor}(x_0,x_1,x_2)+\textup{Cor}(x_0,x_2,x_3)-\textup{Cor}(x_0,x_1,x_3)\\ 
	&\quad-\textup{Cor}(x_0,x_3,x_3)-\textup{Cor}(x_1,x_2,x_3)-\textup{Cor}(x_1,x_1,x_2)+\textup{Cor}(x_1,x_1,x_3)+\textup{Cor}(x_1,x_3,x_3)\\
	&=\Li^{\mathcal{L}}_{2}([\infty,x_0,x_1,x_2])+\Li^{\mathcal{L}}_{2}([\infty,x_0,x_2, x_{3}])-\Li^{\mathcal{L}}_{2}([\infty,x_0,x_1, x_{3}])-\Li^{\mathcal{L}}_{2}([\infty,x_1,x_2, x_{3}])\\
	&=-\Li^{\mathcal{L}}_{2}([x_0,x_1,x_2, x_{3}])=\Li^{\mathcal{L}}_{1;1}([x_0,x_1,x_2, x_{3}]).\\
	\end{split}
\]
\end{example}

The coproduct of  a quadrangular polylogarithm can be computed inductively. We assemble quadrangular polylogarithms with fixed $n$ and different $k$ into a series 
\[
\textup{QLi}_{n}(x_0,\dots, x_{2n+1})=\sum_{k\geq 0} \textup{QLi}_{n,k}(x_0,\dots, x_{2n+1})\in \L.
\]
We introduce the following notation:
\[
\textup{QLi}_{n}^{(-)^s}(x_0,x_1,\dots, x_{2n},x_{2n+1})=
\begin{cases}
\textup{QLi}_{n}(x_0,x_1,\dots, x_{2n},x_{2n+1})	& \text{\ if \ } s  \text{\ is even, \ }\\ 
-\textup{QLi}_{n}(x_1,x_2,\dots, x_{2n+1},x_0)	& \text{\ if \ } s  \text{\ is odd. \ }\\
\end{cases}
\]

\begin{theorem}\label{TheoremClusterPolylogarithmCoproduct} We have the following formula for the coproduct of quadrangular polylogarithms:
\be \label{ClusterPolylogarithmCoproduct}
\begin{split}
	&\Delta^{\mathcal{L}}\textup{QLi}_{n}(x_0,\dots, x_{2n+1})=\\
	&\sum_{\substack{0\leq i<j\leq2n+1\\j-i=2s+1}} \textup{QLi}_{n-s}(x_0,\dots,x_{i},x_{j},\dots,x_{2n+1})\wedge \textup{QLi}_{s}^{(-)^{i}}(x_{i},\dots,x_{j}).\\
\end{split}
\ee
\end{theorem}

\begin{proof} 
For a sequence $\bar{s}=(i_0,\dots,i_{n+k})$ we put
\[
\textup{Cor}(\bar{s})=\textup{Cor}(x_{i_0},\ldots,x_{i_{n+k}})\in \L_{n+k}.
\]
After replacing the quadrangular polylogarithms in (\ref{ClusterPolylogarithmCoproduct}) with  corresponding sums of correlators and applying (\ref{FormulaCoproductCorrelators}) both sides of  (\ref{ClusterPolylogarithmCoproduct}) contain only terms of the form
\[
\pm \textup{Cor}(\bar{s}_1)\wedge \textup{Cor}(\bar{s}_2)
\]
for certain nondecreasing admissible sequences $\bar{s}_1, \bar{s}_2\in \mathcal{C}_{n,\bullet}.$  Below we show that after all cancellations the terms in the left-hand side (LHS) and in the right-hand side (RHS) of (\ref{ClusterPolylogarithmCoproduct}) are the same. For that we will use only (\ref{FormulaCorrelatorsSymmetry}).

For each term $\textup{Cor}(\bar{s}_1)\wedge \textup{Cor}(\bar{s}_2)$ appearing in (\ref{ClusterPolylogarithmCoproduct}) both sequences $\bar{s}_1$ and  $\bar{s}_2$ contain each even index from $0$ to $2n$ at most once. It is easy to see that the set-theoretic intersection $\bar{s}_1\cap \bar{s}_2$ contains at most two indices. We look at all possibilities for the cardinality of the set $\bar{s}_1\cap \bar{s}_2$ and show that in each case the corresponding terms in the LHS and the RHS coincide after all the cancellations are performed.
\begin{enumerate}
\item[{\bf Case 1}] {\it The intersection $\bar{s}_1\cap \bar{s}_2$ is empty.} 
 In the LHS such terms do not appear since in the coproduct of a correlator $\textup{Cor}(x_{i_0},\dots,x_{i_{n+k}})$ the sequences $\bar{s}_1$ and $\bar{s}_2$ contain at least one index in common. Thus we need to show that in the RHS such terms cancel out. Consider a term coming from 
\[
\textup{QLi}_{n-s}(x_0,\dots,x_{i},x_{j},\dots,x_{2n+1})\wedge \textup{QLi}_{s}^{(-)^{i}}(x_{i},\dots,x_{j})
\]
for $0\leq i<j\leq 2n+1, \ j-i=2s+1.$ This term has a form $\textup{Cor}(\bar{s}_1)\wedge \textup{Cor}(\bar{s}_2)$ for a subsequence $s_1$ of  $0, 1,\dots, i,j,\dots, 2n+1$ and a  subsequence $s_2$ of  $i, \dots, j. $ Indices $i$ and $j$ have different parity. Assume that $i$ is even and $j$ is odd; the other case is similar. By the second condition in the definition of $\mathcal{C}_{n,k},$ sequence $s_1$ contains $i$ or $j$.

\begin{enumerate} 
\item {\bf Case 1a: $i \in \bar{s}_1$ and $j \notin \bar{s}_1$.} In this case the same term appears in exactly one other place in the RHS coming from 
\[
\textup{QLi}_{n-s}(x_0,\dots,x_{i+1},x_{j+1},\dots,x_{2n+1})\wedge \textup{QLi}_{s}^{-}(x_{i+1},\dots,x_{j+1})
\] 
with the opposite sign.

\item {\bf Case 1b: $i \notin \bar{s}_1$ and $j \in \bar{s}_1$.} In this case the same term appears in 
\[
\textup{QLi}_{n-s}(x_0,\dots,x_{i-1},x_{j-1},\dots,x_{2n+1})\wedge \textup{QLi}_{s}^{-}(x_{i-1},\dots,x_{j-1})
\] 
with the opposite sign.

\item {\bf Case 1c: $i,j \in \bar{s}_1$.} In this case the same term appears in 
\[
\textup{QLi}_{n-s+1}(x_0,\dots,x_{i+1},x_{j-1},\dots,x_{2n+1})\wedge \textup{QLi}_{s-1}^{-}(x_{i+1},\dots,x_{j-1})
\]
with the opposite sign.
\end{enumerate}

\item[{\bf Case 2}] {\it The intersection $\bar{s}_1\cap \bar{s}_2$ contains two even indices.}
Such terms do not appear. Indeed, in the LHS, all the terms come from cutting a sequence  $\bar{s}\in \mathcal{C}_{n,k}$. One even index in $\bar{s}_1\cap \bar{s}_2$  has to come from the beginning point of the cut. Any even index appears in $\bar{s}$ at most once, so the second index in the intersection must be odd.  Next, in in a term in the RHS coming from 
\[ 
\textup{QLi}_{n-s}(x_0,\dots,x_{i},x_{j},\dots,x_{2n+1})\wedge \textup{QLi}_{s}^{(-)^{i}}(x_{i},\dots,x_{j})
\]
the indices in the intersection $\bar{s}_1\cap \bar{s}_2$  have to be equal to $i$ or $j$ where $j-i=2s+1.$ Thus, one of them has to be odd.

\item[{\bf Case 3}] {\it The intersection $\bar{s}_1\cap \bar{s}_2$ consists of an even index $i$ or a pair of an even index $i$ and an odd index $j.$} Consider a term $\textup{Cor}(\bar{s}_1)\wedge \textup{Cor}(\bar{s}_2)$ in the LHS. Let $j$ be the  largest odd number such that $\bar{s}_1$ contains $j-1$ or $j.$ The same term (with the same sign) comes from 
\[ 
\textup{QLi}_{n-s}(x_0,\dots,x_{i},x_{j},\dots,x_{2n+1})\wedge \textup{QLi}_{s}^{(-)^{i}}(x_{i},\dots,x_{j})
\]
in the RHS.

\item[{\bf Case 4}] {\it The intersection $\bar{s}_1\cap \bar{s}_2$ consists of an index $i$, which is odd.}
 Consider a term $\textup{Cor}(\bar{s}_1)\wedge \textup{Cor}(\bar{s}_2)$ in the LHS coming from a cut beginning with $i.$ Since $\bar{s}_1\cap \bar{s}_2=\{i\},$ the end point of the cut is located between two distinct indices, which we call $j_1$ and $j_2.$ Such term comes from
\[
\textup{QLi}_{n-s}(x_0,\dots,x_{i},x_{j},\dots,x_{2n+1})\wedge \textup{QLi}_{s}^{(-)^{i}}(x_{i},\dots,x_{j}),
\]
where $j$ is the largest even number less than $j_2.$

\item[{\bf Case 5}] {\it The intersection $\bar{s}_1\cap \bar{s}_2$ consists of two odd indices $i$ and $j$.}
Such terms do not appear in the RHS by the same reason as in the Case 2. We need to show that in the LHS such terms cancel out. Consider a term $\textup{Cor}(\bar{s}_1)\wedge \textup{Cor}(\bar{s}_2)$ of this type obtained from a cut  of a sequence $\bar{s}\in \mathcal{C}_{n,k}$ beginning at $i$ and ending between $j$'s. The term $ \textup{Cor}(\bar{s}_2)\wedge \textup{Cor}(\bar{s}_1)$ appears in the coproduct of another admissible sequence $\bar{s}'$ obtained from $s$  by adding an extra $i$ and deleting $j.$ For this, one should take a cut beginning at $j$ and ending between $i$'s. Since 
\[
\textup{Cor}(\bar{s}_1)\wedge \textup{Cor}(\bar{s}_2)+\textup{Cor}(\bar{s}_2)\wedge \textup{Cor}(\bar{s}_1)=0,
\]
the statement follows.
\end{enumerate}
This finishes the proof of  Theorem  \ref{TheoremClusterPolylogarithmCoproduct}.	
\end{proof}

\subsection{Some properties of quadrangular polylogarithms}\label{SectionQuadrangularProp}

It is easy to see that 
\be \label{FormulaQuadrangularTwo}
\textup{QLi}_{0}(x_0, x_{1})= \log^{\mathcal{L}}(x_0-x_1),
\ee 
because
\[
\textup{Cor}(x_0,\underbrace{x_1,x_1,\dots,x_1}_{m})=0 \text{\ for \ } m\geq 2.
\] 
Here is a more interesting computation.

\begin{proposition} \label{FormulaClassicalPolylog} 
We have 
\begin{align*}
\textup{QLi}_{1,k}(x_0, x_{1},x_2,x_3)&=(-1)^{k}\Li^{\mathcal{L}}_{k+1}([x_0,x_1,x_2,x_3])\\
&=-\Li^{\mathcal{L}}_{k;1}([x_0,x_1,x_2,x_3]).
\end{align*}
\end{proposition}
\begin{proof}
	 We have 
\[
\begin{split}
&\Delta^{\mathcal{L}}\textup{QLi}_{1}(x_0, x_{1},x_2,x_3)\\
&=\textup{QLi}_{1}(x_0, x_1,x_2,x_3)\wedge \left (\textup{QLi}_{0}(x_0, x_1)+\textup{QLi}^{-}_{0}(x_1, x_2)+\textup{QLi}_{0}(x_2, x_3)	\right )\\
&\quad+\textup{QLi}_{0}(x_0, x_3)\wedge \textup{QLi}_1(x_0, x_1,x_2,x_3)\\
&=\textup{QLi}_{1}(x_0, x_1,x_2,x_3) \wedge\log^{\mathcal{L}}\left(\frac{(x_0-x_1)(x_2-x_3)}{(x_1-x_2)(x_3-x_0)}\right ),\\
\end{split}
\]
so for $k\geq 1$ we have 
\[
\begin{split}
&\Delta^{\mathcal{L}}\textup{QLi}_{1,k}(x_0, x_{1},x_2,x_3)=\textup{QLi}_{1,k-1}(x_0, x_{1},x_2,x_3)\wedge \log^{\mathcal{L}}[x_0,x_1,x_2,x_3].
\end{split}
\]
It is easy to see that for $k\geq 1$ we have
\[ 
\begin{split}
&\Delta^\L\Li^{\mathcal{L}}_{k+1}(a)= \log^\L(a)\wedge \Li^{\mathcal{L}}_{k}(a),\\
&\Delta^\L\Li^{\mathcal{L}}_{k;1}(a)= \Li^{\mathcal{L}}_{k-1;1}(a)\wedge \log^\L(a).
\end{split}
\]
From here and the fact that $\textup{QLi}_{1,0}(x_0, x_{1},x_2,x_3)=\log^\L(1-[x_0,x_1,x_2,x_3])$ both the equality 
\[
\textup{QLi}_{1,k}(x_0, x_{1},x_2,x_3)=(-1)^{k}\Li^{\mathcal{L}}_{k+1}([x_0,x_1,x_2,x_3])
\] 
and the equality
\[
\textup{QLi}_{1,k}(x_0, x_{1},x_2,x_3)=-\Li^{\mathcal{L}}_{k;1}([x_0,x_1,x_2,x_3])
\] follow by induction.
\end{proof}

\begin{proposition}
Quadrangular polylogarithms $\textup{QLi}_{n,k}$ for $n\geq 1$ are invariant  under projective transformations, i.e., for $\psi\in \textup{PGL}_2(\C)$ we have
\be \label{FormulaProjective}
\textup{QLi}_{n,k}(\psi(x_0),\ldots,\psi(x_{2n+1}))=
\textup{QLi}_{n,k}(x_0,\ldots,x_{2n+1}).
\ee
\end{proposition}

\begin{proof}
From (\ref{FormulaQuadrangularTwo}) we see that for $n\geq 1$ the formula (\ref{ClusterPolylogarithmCoproduct}) can be rewritten in the following way:
\be \label{ClusterPolylogarithmCoproductVersion2}
\begin{split}
	&\Delta^{\mathcal{L}}\textup{QLi}_{n}(x_0,\dots, x_{2n+1})=\textup{QLi}_{n}(x_0,\dots, x_{2n+1})\wedge \log^\L  \left ( \prod_{i=1}^{n} [x_{0},x_{2i-1},x_{2i},x_{2i+1}]\right )\\
	&\quad+\sum_{\substack{0\leq i<j\leq2n+1\\j-i=2s+1\\s\geq 1}} \textup{QLi}_{n-s}(x_0,\dots,x_{i},x_{j},\dots,x_{2n+1})\wedge \textup{QLi}_{s}^{(-)^{i}}(x_{i},\dots,x_{j}).\\
\end{split}
\ee
To show the projective invariance, we proceed by induction.  For $n=1$ the statement follows from Proposition \ref{FormulaClassicalPolylog}. From (\ref{ClusterPolylogarithmCoproductVersion2}) we see that the coproduct $\Delta^\L \textup{QLi}_{n,k}\in \Lambda^2 \L$ is invariant under projective transformations. Now the statement follows by the rigidity argument (\S  \ref{SectionRigidity}), because (\ref{FormulaProjective}) is true  for $\psi=\textup{Id}\in \textup{PGL}_2(\C)$. 
\end{proof}

\subsection{The formula for quadrangular polylogarithms via multiple polylogarithms}

Let $\P=(p_0,\dots,p_{2n+1})$ be an alternating polygon defined in \S \ref{SectionAlternatingPolygons}. Consider the following element of the Lie coalgebra $\L_\P$ of framed unipotent variations of mixed Hodge-Tate structures on $\mathfrak{M}_{\P}:$
\[
\textup{QLi}_{k}(\P)=
\begin{cases}
\textup{QLi}_{n,k}(x_{p_0},\dots,x_{p_{2n+1}}) & \text{\ if \ }	\P \text{\ is even,}\\
\textup{QLi}_{n,k}^{-}(x_{p_0},\dots,x_{p_{2n+1}}) & \text{\ if \ }	\P \text{\ is odd.}
\end{cases}
\]
We define 
\[
\textup{QLi}(\P)=\sum_{k\geq 0} \textup{QLi}_{k}(\P)\in \L_\P.
\]

Next, consider a map $\Li_{\bullet}^\L\colon \F_\P \lra \L_\P$ defined by formula
\[
\Li_{\bullet}^\L[\varphi_1,n_1|\dots|\varphi_k,n_k]=
\Li^{\L}_{\bullet;n_1,\dots,n_k}(\varphi_1,\dots,\varphi_k)
\]
for  functions $\varphi_1,\dots,\varphi_k\in \C(\mathfrak{M}_\P)^\times$  and $n_1,\dots, n_k \geq 1.$ Also, in \S \ref{SectionFormalCluster} we constructed an element $
 \textup{T}_\P\in \F_\P$ called formal quadrangular polylogarithm.

\begin{theorem} \label{TheoremQuadrangulationClusterPolylog}
The following equality holds for an alternating $(2n+2)-$gon $\P$ and $k\geq 0$:
\be \label{FormulaQuadrangulationClusterPolylog}
\textup{QLi}_{k}(\P)=\Li_{k}^\L(\textup{T}_\P).
\ee
\end{theorem}

\begin{proof}To show that  $\textup{QLi}(\P)=\Li_{\bullet}^\L(\textup{T}_\P)$ we proceed by induction on $n$. For $n=1$ the statement follows from (\ref{FormulaFormalPolylogWeight1}) and Proposition \ref{FormulaClassicalPolylog}. 

Assume that $n>1.$ Recall that $\textup{alt}(\P)$ is the set of all alternating subpolygons  $(p_{i_0},p_{i_1},\dots, p_{i_{2r+1}})$ of $\P$ with $i_0=0$ and $i_{2r+1}=2n+1$. Let $\textup{alt}_0(\P)$ be the subset of $\textup{alt}(\P)$ consisting of alternating subpolygons 
$S=(p_{0},\dots,p_{i},p_{j},\dots, p_{2n+1}).$  Then for $S'=(p_{i},\dots,p_{j})$ we have a decomposition $\P=\mathrm{S}\sqcup\mathrm{S}'$.  

For an even $\P$ formula (\ref{ClusterPolylogarithmCoproductVersion2}) is equivalent to the following statement:
\be \label{ClusterPolylogarithmCoproductVersion3}
\begin{split}
	&\Delta^{\mathcal{L}}\textup{QLi}(\P)=\textup{QLi}(\P)\wedge \log^\L(\textup{cr}(\P))+\sum_{\mathrm{S}\in \textup{alt}_0(\P)} \textup{QLi}(\mathrm{S})\wedge \textup{QLi}(\mathrm{S}').\\
\end{split}
\ee
Both parts of the formula (\ref{ClusterPolylogarithmCoproductVersion3}) change sign after a cyclic shift, so the same formula holds for an odd polygon $\P$ as well.

On the other hand, recall that in the proof of Proposition \ref{PropositionMapToGeneralizedPolylogarithms} we found that
\begin{align*}
\begin{split}
&\Delta^\L\Li^{\L}_{\bullet;n_1,\dots,n_k}(\varphi_1,\dots, \varphi_k)\\
&=\Li^{\L}_{\bullet;n_1,\dots,n_k}(\varphi_1,\dots, \varphi_k)\wedge \log^\L(\varphi_1\dots \varphi_k)\\
&\quad+\sum_{s=0}^k\Li^{\L}_{\bullet;n_1,\dots,n_s}(\varphi_1,\dots, \varphi_s)\wedge \Li^{\L}_{\bullet;n_{s+1},\dots,n_k}(\varphi_{s+1},\dots, \varphi_k)\\
&\quad+(1\wedge \Li_\bullet^\L)\Delta^{\H\F}[\varphi_{1},n_{1}|\dots|\varphi_k,n_k].\\
\end{split}
\end{align*}
It follows that 
\[
\begin{split}
&\Delta^\L\Li_{\bullet}^\L(\textup{T}_\P)=\Li_{\bullet}^\L(\textup{T}_\P)\wedge \log^\L(\textup{cr}(\P))+(\Li_\bullet^\L\wedge \Li_\bullet^\L)\Delta^{\mathcal{FF}}\textup{T}_\P+(1\wedge \Li_\bullet^\L)\Delta^{\H\F}\textup{T}_\P.\\
\end{split}
\]
By Proposition \ref{TheoremFormalSSR} we have $\Delta^{\H\F}\textup{T}_\P=0,$ so
\be \label{FormulaCoproductLi0}
\begin{split}
&\Delta^\L\Li_{\bullet}^\L(\textup{T}_\P)=\Li_{\bullet}^\L(\textup{T}_\P)\wedge \log^\L(\textup{cr}(\P))+(\Li_\bullet^\L\wedge \Li_\bullet^\L)\Delta^{\mathcal{FF}}\textup{T}_\P.\\
\end{split}
\ee
By formula $(\ref{FormulaCoproductFormalPolylogarithm})$ we have
\[
(\Li_\bullet^\L\wedge \Li_\bullet^\L)\Delta^{\mathcal{FF}}\textup{T}_\P= \sum_{S\in\textup{alt}(\P)} \Li_\bullet(\textup{T}_{\mathrm{S}})\wedge \Li_\bullet\left (\prod_{i=0}^{2r}\textup{T}_{\mathrm{S}^i} \right)=
\sum_{S\in\textup{alt}_0(\P)}\Li_\bullet(\textup{T}_{\mathrm{S}})\wedge \Li_\bullet(\textup{T}_{\mathrm{S}'}).
\]
The second equality holds because for polygons $S$ in $\textup{alt}(\P)$ but not in  $\textup{alt}_0(\P)$ the term 
$\Li_\bullet\left (\prod_{i=0}^{2r}\textup{T}_{\mathrm{S}^i} \right)$
vanishes in  $\L_P$ by the generalized quasi-shuffle relation (Proposition \ref{PropositionMapToGeneralizedPolylogarithms}). Now, from (\ref{FormulaCoproductLi0}) we have
\be\label{FormulaCoproductInProof}
\begin{split}
&\Delta^\L\Li_{\bullet}^\L(\textup{T}_\P)=\Li_{\bullet}^\L(\textup{T}_\P)\wedge \log^\L(\textup{cr}(\P))+\sum_{S\in\textup{alt}_0(\P)}\Li_\bullet(\textup{T}_{\mathrm{S}})\wedge \Li_\bullet(\textup{T}_{\mathrm{S}'}).\\
\end{split}
\ee
Comparing formulas (\ref{ClusterPolylogarithmCoproductVersion3}) and (\ref{FormulaCoproductInProof}) we conclude by induction that 
\[
\Delta^{\mathcal{L}}(\textup{QLi}(\P))=\Delta^{\mathcal{L}}(\Li_{\bullet}^\L(\textup{T}_\P)).
\]
It is easy to see that for both framed variations, the specialization to a point with 
\[
x_{p_0}=x_{p_1}, x_{p_2}=x_{p_3}, \dots, x_{p_{2n}}=x_{p_{2n+1}}
\] 
vanishes, so the conclusion follows by the rigidity argument (\S \ref{SectionRigidity}).
\end{proof}

Identities between variations of framed mixed Hodge-Tate structures imply identities between the corresponding multivalued functions, see \S \ref{SectionRigidity}. In view of that, Theorem \ref{TheoremQuadrangulationClusterPolylog} implies Theorem \ref{MainTheoremQuadrangulationCluster}.

\begin{corollary}\label{CorollaryDepthOfClusterPolylogarithm}
	The depth of the quadrangular polylogarithm $\textup{QLi}_{n,k}$ is less than or equal to $n.$
\end{corollary}
\begin{proof}
We need to show that for an alternating $(2n+2)-$gon $\P$ the depth of $\Li_{k}^\L(\textup{T}_\P)$ is less than or equal to $n.$ From Lemma \ref{LemmaRecursiveTP} we conclude that $\textup{T}_\P$ is the sum of words of length less than or equal to $n.$ Since the map $\Li_{k}^\L$ sends a word of length $r$ to a multiple polylogarithm of depth $r,$ the statement follows.
\end{proof}

\subsection{Universality of quadrangular polylogarithms}
\label{SectionQuadrangularUniversality}

\begin{proposition}\label{PropositionUniversality}
	The following formula holds:
	\be	\label{FormulaCorToClustereven}
	\textup{Cor}_a(x_0,\dots,x_{2n})=\sum_{s=0}^{2n+2} \sum_{0\leq i_1<\dots<i_s\leq 2n}(-1)^s\textup{QLi}_{n,n}(x_0,\dots,a,\dots,a,\dots,x_{2n},a).
	\ee
	Here the $s$-th term is $(-1)^s$ times a sum over  $0\leq i_1<\dots<i_s\leq 2n$ of quadrangular polylogarithms obtained from the quadrangular polylogarithm $\textup{QLi}_{n,n}(x_0,\dots,x_{2n},a)$ by substituting the point $a$  instead of points $x_{i_1},\dots,x_{i_s}.$
	Similarly,
	
	\be	\label{FormulaCorToClusterodd}
	\textup{Cor}_a(x_0,\dots,x_{2n+1})=\sum_{s=0}^{2n+2} \sum_{0\leq i_1<\dots<i_s\leq 2n+1}(-1)^s\textup{QLi}_{n,n+1}(x_0,\dots,a,\dots,a,\dots,x_{2n+1}).
	\ee
	Here the $s$-th term is $(-1)^s$ times a sum over  $0\leq i_1<\dots<i_s\leq 2n+1$ of quadrangular polylogarithms obtained from the quadrangular polylogarithm $\textup{QLi}_{n,n+1}(x_0,\dots,x_{2n+1})$ by substituting the point $a$  instead of points $x_{i_1},\dots,x_{i_s}.$
\end{proposition}

\begin{proof}
We prove (\ref{FormulaCorToClusterodd}); the proof of (\ref{FormulaCorToClustereven}) is similar.  Consider a vector space of functions depending on variables $(x_0,\dots,x_{2n+1}).$ Consider  an endomorphism of this space given by the formula 
\[
\textup{T}(f)(x_0,\dots,x_{2n+1})=
\sum_{s=0}^{2n+2} \sum_{0\leq i_1<\dots<i_s\leq 2n+1}(-1)^s f(x_0,\dots,a,\dots,a,\dots,x_{2n+1}).
\]
Any function which does not depend on at least one variable $x_i$ lies in the kernel of $\textup{T}.$ According to Definition \ref{DefClusterPolylog} 
\[
\textup{QLi}_{n,n+1}(x_0,\dots,x_{2n+1})=\textup{Cor}(x_0,\dots,x_{2n+1})+(\text{correlators with repeating argument}),
\]
so 
\[
\textup{T}(\textup{QLi}_{n,n+1}(x_0,\dots,x_{2n+1}))=\textup{T}(\textup{Cor}(x_0,\dots,x_{2n+1})).
\]
On the other hand, by (\ref{CorrelatorChangeBasepoint}) we have
\[
\textup{T}(\textup{Cor}(x_0,\dots,x_{2n+1}))=\textup{Cor}_a(x_0,\dots,x_{2n+1}).
\]
\end{proof}

\begin{corollary}\label{CorollaryUniverality}
Iterated integrals $\textup{I}^{\L}$ and multiple polylogarithms $\Li^{\L}$ 	can be expressed as linear combinations of quadrangular polylogarithms $\textup{QLi}_{n,n}$ and $\textup{QLi}_{n,n+1}$ for $n\geq 0.$ 
\end{corollary}
\begin{proof}
The statement follows from \ref{FormulaIIviaCor} and \ref{FormulaMPviaII}.	
\end{proof}

Corollaries \ref{CorollaryDepthOfClusterPolylogarithm} and \ref{CorollaryUniverality}  imply Theorem \ref{MainTheoremDepth}, see \S\ref{SectionRigidity}.

\section{Volume of non-Euclidean orthoschemes}\label{SecOrt}

\subsection{Projective simplices and scissors congruence}\label{SectionProjectiveSimplex}
Let $E$ be a vector space over $\C$ of dimension $m$ and consider a quadratic form $q\in \mathbb{S}^2E^{\vee}.$   The set of zeros of $q$ in the projectivization $\mathbb{P}(E)$ is a quadric, which we denote by $Q$. The quadric $Q$ is smooth if and only if the quadratic form $q$ is nondegenerate. A smooth quadric defines a duality between subspaces of $\mathbb{P}(E),$ known as the polar duality. For a subspace $\pi\subseteq \mathbb{P}(E)$ the polar subspace $\pi^{\perp}$ has dimension $m-2-\textup{dim}(\pi).$ Two hyperplanes $H_1$ and $H_2$ in $\mathbb{P}(E)$ are called orthogonal if $H_1^{\perp}\in H_2$ (equivalently, $ H_2^{\perp}\in H_1.$) 

A smooth quadric $Q$ contains projective subspaces of dimension $\left \lfloor \dfrac{m}{2} \right \rfloor-1$ parametrized by a certain Fano variety. This variety is irreducible if $\textup{dim}(Q)$ is odd and  has two connected components if $\textup{dim}(Q)$ is even. We call the choice of an irreducible component of this variety an {\it orientation} of  the quadric. 

\begin{definition}
 A projective simplex $S=(Q;\H)$ in $\mathbb{P}(E)$ is a configuration, consisting of a quadric $Q$ and an ordered set of hyperplanes $\H=\{H_1, H_2, \dots, H_{m}\}$ in $\mathbb{P}(E).$
\end{definition}

Two projective simplices $S_1=(Q_1; \H_1)$ and $S_2=(Q_2; \H_2)$  are called isometric if there exists a projective transformation sending  $Q_1$ to $Q_2$ and hyperplanes in $\H_1$ to the corresponding hyperplanes in $\H_2.$ Denote by $h_i$ hyperplanes in $E$ with projectivization $H_i$ and for any subset $I\subseteq \{1,2,\dots, m\}$  define 
$h_I=\bigcap_{i\in I} h_i\subseteq E$ and $H_I=\bigcap_{i\in I}H_i\subseteq \mathbb{P}(E).$ A projective simplex $S$ is called {\it nondegenerate} if intersections $Q_I=H_I\cap Q$ are smooth for every subset $I\subseteq \{1,\dots,m\}$ (in particular, $Q$ itself is smooth). An orientation of $S$ is an orientation of $Q_I$ for every $I\subseteq \{1,\dots,m\}$. A nondegenerate projective simplex has $2^{(2^{m-1}-1)}$ orientations.

 Let $S$ be a nondegenerate simplex. For every subset $I\subseteq \{1,\dots, m\}$ we define the projective simplex $S_I$, called the $I$-face of $S.$  Simplex $S_I$ is a configuration of the quadric $Q\cap H_I$ in the projective space $H_I$  and hyperplanes $H_{I}\cap H_j$ for $j\notin I.$ Next, we define projective simplex $S^I$, called $I$-angle of $S.$ It is a configuration of the quadric $Q\cap (H_I)^{\perp}$ in the projective space $(H_I)^{\perp}$ and a collection of hyperplanes $H_{I\fgebackslash \{i\}} \cap (H_I)^{\perp}$ for $i\in I.$ It is easy to see that for an odd-dimensional simplex $S$ an orientation of $S$ induces an orientation of all its faces and angles.
 
 Let $S_1$ and $S_2$ be a pair of projective simplices in spaces $\mathbb{P}(E)$ and $\mathbb{P}(E')$ defined by quadratic forms $q$ and $q'$ and hyperplanes $h_1,\dots, h_m$ and  $h_1',\dots, h_{m'}'.$  Their join $S_1\cdot S_2$ is a projective simplex in $\mathbb{P}(E\oplus E')$ defined by hyperplanes $h_i\oplus E', 1\leq i \leq m$  and $E\oplus h_j', 1\leq j \leq m'.$ If $S_1$ and $S_2$ are nondegenerate, then $S_1\cdot S_2$ is nondegenerate. If $S_1$ and $S_2$  are oriented, then $S_1\cdot S_2$ has a canonical orientation.

Next, we recall a construction of the Hopf algebra of polytopes modulo scissors congruence, following Goncharov \cite[\S 3.4]{Gon99}. We define a commutative graded Hopf algebra 
 $\mathcal{G}$  over $\mathbb{Q}$ by generators and relations in the following way. We have $\mathcal{G}_0=\Q$ and for $n\geq 1$ the component $\mathcal{G}_{n}$  is generated  by isometry classes $[S]$ of oriented nondegenerate $(2n-1)$-dimensional projective simplices, subject to the following relations:

\begin{enumerate}
\item	For the class $[S]$ of a simplex $S=(Q;H_1,\dots, H_{2n})$ in $\mathbb{P}^{2n-1}$ and $\sigma \in \mathbb{S}_{2n}$ we have 
\[
[(Q;H_{\sigma(1)},\dots, H_{\sigma(2n)})]=(-1)^\sigma[(Q;H_1,\dots, H_{2n})].
\]
The class $[S]$ does not depend on orientations of the quadrics $Q_I$ for $I\neq \{1,\ldots,m\}.$ If $\bar{S}$ is the same simplex with the opposite orientation of   $Q$ then $[\bar{S}]=-[S].$
If $Q$ hyperplanes $H_i$ are not in  general position,  we put $[S]=0.$

\item Consider a collection of hyperplanes $H_0, H_1, H_2 \dots, H_{2n}$ and an oriented quadric $Q$ in $\mathbb{P}^{2n-1}.$   Then
\[
\sum_{i=0}^{2n} (-1)^{i}[(Q;H_0,\dots \widehat{H_i}, \dots,H_{2n})]=0.
\]
\end{enumerate}

The product in $\mathcal{G}$ is defined on generators as the join: 
\[
[S_1][S_2]=[S_1\cdot S_2].
\]
Clearly, it is commutative. The coproduct is defined by the formula
\be \label{FormulaDehn}
\Delta^{\mathcal{G}}[S]=\sum_{\substack{I\subseteq\{1,\dots,2n\}\\ |I| \text{ is even}}}[S_I] \otimes \left[S^I \right].
\ee
The coproduct $\Delta^\mathcal{G}$ is a projective counterpart of the Dehn invariant, used by Dehn to show that a regular tetrahedron is not scissors congruent to a cube of the same volume. 

In \cite[\S 3]{Gon99} Goncharov constructed the Hodge realization map  
\[
h\colon\mathcal{G} \lra \H,
\]
assigning a framed mixed Hodge structure to a projective tetrahedron. For an oriented projective tetrahedron $S$, the framed mixed Hodge-Tate structure $h(S)$ is equal to 
\[
H^{2n-1}\left(\mathbb{P}^{2n-1}\fgebackslash Q, \left(\bigcup_{i=1}^{2n} H_i \right)\fgebackslash Q \right)
\]
with a certain framing. Goncharov proved  \cite[Theorem 3.11]{Gon99} that $h$ is a homomorphism of commutative Hopf algebras.

\begin{example}\label{ExampleOneDimensionalSimplex} In weight one we have $\H_1\cong \mathbb{C}^{\times}_\Q.$ For a $1$-dimensional simplex $S=(Q;H_1,H_2)$ a quadric $Q$ is a pair of points $Q_1,Q_2\in \mathbb{P}^1.$ An orientation of the quadric is the choice of an ordering of points $Q_1$ and $Q_2.$ We have
	\[
h(S)=-\frac{1}{2}\log^\H\left([Q_1, H_1,Q_2, H_2] \right)\in \H_1 .
\]
\end{example}

For every hyperbolic simplex $\mathcal{S}\subseteq \mathbb{H}^{2n-1}$ of dimension $2n-1$ one can construct the corresponding projective simplex $S$, see  \cite[\S 1.5]{Gon99}.  In \cite[\S 4]{Gon99} Goncharov defined a real period map 
\be \label{FormulaRealPeriod}
\mathrm{per}_{\mathbb{R}}\colon \H_n \lra \mathbb{R}(n-1)
\ee
where $\mathbb{R}(n-1)=\mathbb{R}(2\pi i)^{n-1}\subseteq \C.$  Goncharov proved in \cite[Theorem 4.7]{Gon99} that the following equality holds:  
\[
\textup{Vol}(\mathcal{S})=\mathrm{per}_{\mathbb{R}}(h(S))\in \mathbb{R}(n-1).
\]
Notice that our normalization of the hyperbolic volume is different from the standard one by a factor of $(2\pi i)^ni/(2n-2)!.$


\subsection{Orthoschemes and Maslov index} \label{SecOrt1}

\begin{definition}\label{DefOrt}
A projective simplex $S=(Q;H_1,\dots,H_m)$ in $\mathbb{P}^{m-1}$ is called an orthoscheme if $Q$ is smooth, the hyperplanes $H_1,\dots, H_m$ are in general position and the hyperplanes $H_i$ and $H_j$ are orthogonal for $|i-j|>1.$ 
\end{definition}

From the point of view of the theory of scissors congruence, orthoschemes are important because 
they  generate the scissors congruence Hopf algebra $\mathcal{G}$ as a vector space. One can easily see that by considering an orthogonal version of the barycentric subdivision. Every simplex in $\mathbb{P}^{m-1}$ is scissors congruent to a sum of $m!$ orthoschemes. A related question of presenting a simplex as a union of disjoint orthoschemes is much harder and is related to a conjecture of Hadwiger, see \cite{Had56}.

For a projective orthoscheme $S=(Q;H_1,\dots,H_m)$ in $\mathbb{P}^{m-1}$ we define two more hyperplanes in $\mathbb{P}^{m-1}$ 
\be \label{FormulaCorners}
H_0=\left ( \bigcap_{i=2}^{m}H_i\right )^\perp \text{ \ and \ }
H_{m+1}=\left ( \bigcap_{i=1}^{m-1}H_i\right )^\perp.
\ee
We label hyperplanes $H_0,\dots, H_{m+1}$ by an index $i\in \mathbb{Z}/(m+2)\mathbb{Z}$ and denote by $h_0,\dots,h_{m+1}$ the corresponding hyperplanes in $E.$ Distinct indices $i,j \in \mathbb{Z}/(m+2)\mathbb{Z}$ are called adjacent if $(i-j)\equiv \pm 1 \pmod{m+2}.$  Clearly, the hyperplanes $H_i$ and $H_j$ are orthogonal for any pair of nonadjacent indices.
\begin{definition}
An orthoscheme $S$ is called generic if hyperplanes  $H_0,\dots, H_{m+1}$ are in general position (the intersection of any $m$ of them is empty).
\end{definition}

We will show that a generic orthoscheme is a nondegenerate projective simplex, see Corollary \ref{CorollaryGenericOrthoscheme}. If $S$ is a generic orthoscheme, then the formula (\ref{FormulaCorners}) generalizes to
\[
H_r=\left ( \bigcap_{i\not \in \{r-1,r,r+1\}} H_i\right )^\perp \text{ \ for \ } r\in \mathbb{Z}/(m+2)\mathbb{Z}.
\]

\begin{remark}
If $S$ is a generic orthoscheme, then simplices 
\[
S^{[i]}=(Q;H_{i+1},H_{i+2}, \dots , H_{i+m}),\ \ i \in \mathbb{Z}/(m+2)\mathbb{Z},
\]
are also generic orthoschemes. This sequence of orthoschemes is classically known as the {\it Napier cycle}. 
\end{remark}

Our first result is a bijection between isometry classes of generic orthoschemes and points of $\mathfrak{M}_{0,m+2}.$ This result is inspired by  Coxeter's spectacular work \cite{Cox36}, but our construction seems to be new. It is based on an algebraic approach to the Maslov index from an unpublished work by Kashiwara, see also \cite{LV80}. 

Consider a configuration $(x_0,\dots,x_{m+1})\in \mathfrak{M}_{0,m+2}$ of points in $\mathbb{P}^1=\mathbb{P}(V)$ and let $l_0,\dots, l_{m+1}$ be lines in $V$ corresponding to points  $x_0,\dots,x_{m+1}\in \mathbb{P}^1.$ Let $\omega\in \Lambda^2 V^{\vee}$ be a symplectic form. The vector space 
\be\label{FormulaMaslov}
E=\textup{Ker}\left (\bigoplus_{i=0}^{m+1}l_i\stackrel{\sum}{\lra} V \right)
\ee
has dimension $m$ and carries a non-degenerate quadratic form defined on a vector \[v=(v_0,\dots,v_{m+1})\in E\] by the formula
\[
q(v)=\sum_{0\leq i<j\leq m+1}\omega(v_i,v_{j}).
\]
We denote by $q(v_1,v_2)$ the symmetric bilinear form, associated to $q.$

\begin{example}\label{ExampleQuadricTriangle} Consider the case $m=1.$ Fix nonzero vectors $e_0\in l_0, e_1\in l_1, e_2\in l_2.$ The vector space $E$ is spanned by the vector 
\[
v=\left(\omega(e_1,e_2)e_0, \omega(e_2,e_0)e_1, \omega(e_0,e_1)e_2\right).
\]
Then 
\be \label{FormulaQuad3}
q(v,v)=\omega(e_1,e_2)\omega(e_2,e_0)\omega(e_0,e_1).
\ee	
\end{example}

For any subset $I=\{i_0,i_1\dots,i_{r+1}\}\subseteq \{0,1,\dots,m+1\}$ we define a vector space 
\[
E_I=\textup{Ker}\left (\bigoplus_{j=0}^{r+1}l_{i_j}\stackrel{\sum}{\lra} V \right),
\]
which is a subspace of $E.$ As a quadratic space it is isometric to the space obtained from the configuration $(x_{i_0},\dots,x_{i_{r+1}})\in \mathfrak{M}_{0,m+2}$ by the same construction as above. 

\begin{lemma}\label{LemmaOrtSum}
For $0\leq i<j\leq m+1$ consider subsets $I=\{0,\dots,i,j,\dots, m+1\}$ and $I'=\{i, i+1,\dots, j\}$ of $\{0,\dots,m+1\}.$ Then we have an orthogonal decomposition $E=E_I\oplus E_{I'}.$ 
\end{lemma}
\begin{proof}
We have $\textup{dim}(E)=\textup{dim}(E_I)+\textup{dim}(E_{I'})$ and $E_I\cap E_{I'}=0,$ so $E=E_I\oplus E_{I'}.$ Also, consider  $v\in E_I$ and $v'\in E_{I'}$ such that
$v=(v_0,\dots, v_{m+1}), v'=(v_0',\dots, v_{m+1}')\in E.$ Then we have $v_r=0$ for $i<r<j$ and $v_r'=0$ for $r<i$ and for $r>j.$    We have
\be
\begin{split}
&q(v,v')=\sum_{r_1<r_2}\omega(v_{r_1},v_{r_2}')=\sum_{r_1\leq r_2}\omega(v_{r_1},v_{r_2}')=\sum_{r_1 \leq i\leq r_2\leq j}\omega(v_{r_1},v_{r_2}')\\
&=\omega \left (\sum_{ r_1\leq i} v_{r_1}, \sum_{i\leq r_2 \leq j} v_{r_2}'\right )=\omega \left (\sum_{ r_1\leq i} v_{r_1}, 0\right )=0.
\end{split}
\ee
This implies that $E_I$ and $E_{I'}$ are orthogonal subspaces of $E,$ so the decomposition  $E=E_I\oplus E_{I'}$ is orthogonal.
\end{proof}

\begin{theorem}\label{TheoremOrthoschemesConfigurations}
For a configuration $(x_0,\dots, x_{m+1})$ consider a projective simplex 
\[
\textup{ort}(x_0,\dots, x_{m+1})=(Q;H_1,\dots,H_{m+1})
\]	
in $\mathbb{P}(E),$ where $Q$ is the quadric defined by a quadratic form $q$ and $H_i=\mathbb{P}\left(E_{i-1,i,i+1}^{\perp}\right)$ for $1\leq i \leq m.$ Then $\textup{ort}(x_0,\dots, x_{m+1})$ is a generic projective orthoscheme. This gives a bijection between points of $\mathfrak{M}_{0,m+2}$ and isometry classes of generic projective orthoschemes.
\end{theorem}
\begin{proof}
To prove that $S=\textup{ort}(x_0,\dots, x_{m+1})$ is an orthoscheme first notice that by Lemma \ref{LemmaOrtSum} we have 
\[
E_{i-1,i,i+1}^{\perp}=E_{\{0,\dots, i-1,i+1,\dots, m+1\}}.
\] 
From here it is obvious that for $|i-j|>1$ we have $E_{j-1,j,j+1}\subseteq E_{i-1,i,i+1}^{\perp},$ so $S$ is an orthoscheme. By the same argument, $h_0=E_{0,1,m+1}^{\perp}$ and $h_{m+1}=E_{0,m,m+1}^{\perp}.$
It is easy to see that
\be \label{FormulaGeneric}
\bigcap_{r\notin\{i,j\}} h_r=0
\ee
 for any $0 \leq i\neq j\leq m+1.$ Indeed, by Lemma \ref{LemmaOrtSum}  if $v=(v_0,\dots,v_{m+1})$ is a vector in $\bigcap_{r\notin\{i,j\}} h_r$ then we have an equality $v_r=0$ for any $r\notin \{i,j\}.$ Since $\sum_{r=0}^{m+1}v_r=0,$ we conclude that $v_i+v_j=0,$ which contradicts our assumption that the lines $l_0,\dots, l_{m+1}$ are distinct. The formula (\ref{FormulaGeneric}) implies that $S$ is generic.

Next, we give a construction of the map in the other direction. For a generic orthoscheme 
\[
S=(Q;H_1,\dots,H_m)
\] 
in $\mathbb{P}(E)$ we consider the sum of projection maps  
\be \label{FormulaSumOfProjections}
p\colon E\lra\bigoplus_{i=0}^{m+1}(E/h_i).
\ee
 Since $H_0,\dots, H_{m+1}$ are in general position, $p$ is injective and images of the lines $E/h_i$ in the space $\textup{Coker}(p)$ are distinct. We define $\textup{conf}(S)$ to be the configuration of these lines in $\mathbb{P}^1=\mathbb{P}(\textup{Coker}(p)).$

To see that $\textup{conf}(\textup{ort}(x_0,\dots,x_{m+1}))=(x_0,\dots,x_{m+1})$ notice that  for a configuration $(x_0,\dots,x_{m+1})$ we have an exact sequence 
\be
0\lra E\lra \bigoplus_{i=0}^{m+1} l_i \lra V \lra 0.
\ee
By Lemma \ref{LemmaOrtSum} we have a canonical isomorphism ${l_i\cong E_{i-1, i,i+1}\cong E/h_i}$ for $0\leq i\leq m+1.$ Moreover, it is easy to see that the embedding $E\hookrightarrow \oplus_{i=0}^{m+1} l_i$ is identified with the sum of projections (\ref{FormulaSumOfProjections}). 

Finally, we need to show that the generic orthoschemes 
\[
S=(Q;H_1,\dots,H_m)
\]
and 
\[
{\textup{ort}(\textup{conf}(S))=(Q';H_1',\dots,H_m')}
\]  
are isometric. By the same argument as above,  the configurations of hyperplanes $H_0,\dots, H_{m+1}$ and  $H_0',\dots, H_{m+1}'$ are projectively equivalent. The fact that $Q$ is uniquely  determined by the hyperplanes $H_0,\dots, H_{m+1}$ can be easily checked in coordinates and is left to the reader.
\end{proof}

The faces and the angles of a generic orthoscheme are joins of generic orthoschemes.

\begin{proposition}\label{PropositionFacesAngles}
Let $S=\textup{ort}(x_0,\dots,x_{m+1})$ be a generic orthoscheme. For a subset $I=\{i_1,\dots,i_r\}	\subseteq \{1,\dots,m\}$ we define $i_0=0, i_{r+1}=m+1.$ Then we have
\be\label{FormulaFaceOrthoscheme} 
S_I=\textup{ort}(x_{i_0},x_{i_1},\dots,x_{i_{r+1}}),
\ee
and
\be \label{FormulaAngleOrthoscheme} 
S^I=\prod_{p=0}^r\textup{ort}(x_{i_p},x_{i_p+1},\dots,x_{i_{p+1}-1},x_{i_{p+1}}).
\ee

\end{proposition}
\begin{proof}
	Let $J=I\cup\{0, m+1\}\subseteq \{0,\dots,m+1\}$ and $J_p=\{i_p,i_p+1,\dots,i_{p+1}-1,i_{p+1}\}$ for $0\leq p\leq r.$ Applying Lemma \ref{LemmaOrtSum} consequently we deduce that we have an orthogonal decomposition
\[
E=E_J\oplus \bigoplus_{p=0}^r E_{J_p}. 
\]
	From here, the statement can be deduced easily; we leave the details to the reader.
\end{proof}

\begin{corollary} \label{CorollaryGenericOrthoscheme}
A generic orthoscheme is a nondegenerate projective simplex.	
\end{corollary}
\begin{proof}
For a generic orthoscheme $S=(Q,H_1,\dots,H_m)$ the quadric $Q$ is smooth, and every face $S_I$ is a generic orthoscheme, so quadrics $H_I\cap Q$ are also smooth.
\end{proof}

Theorem \ref{TheoremOrthoschemesConfigurations} and Proposition \ref{PropositionFacesAngles} imply Theorem \ref{MainTheoremOrthoschemes}.

\subsection{Hyperbolic orthoschemes}
In this section, we discuss hyperbolic orthoschemes, i.e., projective orthoschemes coming from hyperbolic geometry.

\begin{definition}
The hyperbolic locus $\mathfrak{M}_{0,m+2}^h\subseteq \mathfrak{M}_{0,m+2}$ is the connected component of the set of real points of the variety $\mathfrak{M}_{0,m+2}$ consisting of configurations projectively equivalent to 
\[
x=(x_{0},\dots,x_{m+1})
\]
with $x_{0},\dots, x_{m+1}\in \mathbb{R}$ and
\be \label{FormulaPointsHyperbolicOrthoscheme}
x_{m+1}<x_{1}<x_{2}<\dots<x_{m-1}<x_{m}<x_{0}.
\ee
\end{definition}

\begin{proposition}
For  $x\in \mathfrak{M}_{0,m+2}^h$  the orthoscheme $\textup{ort}(x)$ is the projectivization of a hyperbolic orthoscheme. This gives a bijection between points of $\mathfrak{M}_{0,m+2}^h$ and isometry classes of hyperbolic orthoschemes.
\end{proposition}
\begin{proof}
For $x\in \mathfrak{M}_{0,m+2}^h$ the vector space $V$ is a complexification of a real vector space. Also, the quadratic space  $E$ defined by (\ref{FormulaMaslov}) is a complexification of a real quadratic space. In both cases, we denote the corresponding real vector spaces by the same letters. We choose the symplectic form $\omega\in \Lambda^2 V^{\vee}$ to be the usual area form. Denote the index of a quadratic space $E$ by $\textup{Ind}(E).$  Since index is additive, Lemma \ref{LemmaOrtSum} implies that  
\[
\textup{Ind}(E)=\sum_{i=1}^{m}\textup{Ind}(E_{0,i,i+1}).
\]
From the condition (\ref{FormulaPointsHyperbolicOrthoscheme}) and the formula (\ref{FormulaQuad3}) we see that 
\[
\textup{Ind}(E_{0,i,i+1})=
\begin{cases}
1 &\text{if\ } i=m,\\	
-1 &\text{if\ } i\neq m,\\	
\end{cases}
\]
so $\textup{Ind}(E)=m-2$ and  $q$ has   signature $(1,m-1).$ The hyperboloid $\{v\in E \mid q(v)=1\}$ has two connected components and we identify one of them with $\mathbb{H}^{m-1}.$ Since 
\[
\textup{Ind}(E_{0,i,m+1})=1 \text{\ for \ } 1\leq i \leq m,
\]
we can define a hyperbolic tetrahedron with vertices 
\[
A_i=E_{0,i,m+1}\cap \mathbb{H}^{m-1}
\]  
for $1\leq i \leq m.$ In this way we get a hyperbolic orthoscheme with projectivization $\textup{ort}(x).$

To prove the implication in the other direction, consider a projective orthoscheme $S$, obtained from a hyperbolic orthoscheme by projectivization. The orthoscheme $S$ is generic, because hyperplanes $H_0,H_{m+1}$  defined by (\ref{FormulaCorners}) do not intersect $\mathbb{H}^{m-1}.$ The map $(\ref{FormulaSumOfProjections})$ is a complexification of the corresponding map of real vector spaces, so the configuration $\textup{conf}(S)$ is equivalent to a configuration $(x_{0},\dots,x_{m+1})$  with $x_{0},\dots, x_{m+1}\in \mathbb{R}.$ Without loss of generality, $x_{m+1}<x_0.$  Since $\textup{ort}(\textup{conf}(S))$ is isometric to $S$ by Theorem~\ref{TheoremOrthoschemesConfigurations}, we know that 
\[
\textup{Ind}(E_{0,i,m+1})=1 \text{\ for \ } 1\leq i\leq m,
\]
so $x_{m+1}<x_i<x_0$ for $1\leq i\leq m.$ On the other hand, for any $1\leq i<j \leq m$ the quadratic space $E_{0,i,j,m+1}$ must have signature $(1,1),$ so 
\[
\textup{Ind}(E_{0,i,j})=\textup{Ind}(E_{0,i,j,m+1})-\textup{Ind}(E_{0,j,m+1})=-1 \text{\ for \ } 1\leq i\leq m,
\]
thus $x_i<x_j$ by (\ref{FormulaQuad3}).
\end{proof}

\subsection{Orientations of orthoschemes}\label{SecOrOrt}

Let 
\[
\P=(p_0,\dots,p_{2n+1})
\]
be an alternating polygon. The variety $\mathfrak{M}_{\P}$ has an  {\'e}tale covering $\mathfrak{M}_\P^{s}.$ This covering can be characterized by the fact that for an alternating subpolygon $\mathrm{P}'$ of $\P$ the square root of the cross-ratio $\sqrt{\textup{cr}(\P')}$  is a regular function on $\mathfrak{M}_\P^{s}.$ 

By Theorem \ref{TheoremOrthoschemesConfigurations} we have a generic orthoscheme $\textup{ort}(x)$ for every $x\in \mathfrak{M}_{\P};$ let $Q_x$ be the corresponding quadric. Since $Q_x$ has two orientations, we have an action of the spherical pure braid group $\pi_1(\mathfrak{M}_\P)$ on the set with two elements. Consider the corresponding cohomology class 
\[
\textup{Or}_\P\in H^1\left(\mathfrak{M}_{\P},\mathbb{Z}/2\mathbb{Z}\right).
\]

For an alternating  subpolygon $\P'$ of $\P$ we have the forgetful morphism $f_{\P,\P'}\colon \mathfrak{M}_{\P}\lra \mathfrak{M}_{\P'}.$ 

\begin{lemma}\label{LemmaOrientation} For every decomposition $\P=\P_1 \sqcup P_2$ the following equality holds:
\[
\textup{Or}_\P=(f_{\P,\P_1})^*\textup{Or}_{\P_1}+(f_{\P,\P_2})^*\textup{Or}_{\P_2}.
\]
\end{lemma}
\begin{proof}
The statement follows immediately from the following observations. By Lemma \ref{LemmaOrtSum} for a point $x\in \mathfrak{M}_\P$  the quadratic space $E_x$ defined by (\ref{FormulaMaslov}) is the orthogonal sum of the corresponding spaces $E_{f_{\P,\P_1}(x)}$ and $E_{f_{\P,\P_2}(x)}.$ It is easy to see that a choice of orientations of quadrics $Q_{f_{\P,\P_1}(x)}$ and $Q_{f_{\P,\P_1}(x)}$ defines a choice of an orientation of the quadric  $Q_x,$ from where the statement can be easily deduced.
\end{proof}

\begin{definition}\label{DefinitionOrientedConfigurations}
For an alternating polygon $\P$ consider a subgroup 
\[
\bigcap_{\P'\subseteq \P}\textup{Ker}(\textup{Or}_{\P'})\subseteq \pi_1(\mathfrak{M}_{\P}).
\]	
We denote by $\mathfrak{M}_{\P}^s$ the corresponding {\'e}tale cover of $\mathfrak{M}_{\P}.$
\end{definition}

\begin{remark}
	The group $\bigcap_{\P'\subseteq \P}\textup{Ker}(\textup{Or}_{\P'})$ contains the level two congruence subgroup of the pure spherical braid group, so the volume of an orthoscheme is a function on the pro-unipotent completion of this group. It would be interesting to interpret this function in the language of \cite{KM19}.
\end{remark}

By  Definition \ref{DefinitionOrientedConfigurations}  in order to define an orientation of $\textup{ort}(x)$ for each point $x^s\in \mathfrak{M}_\P^s$ over $x\in \mathfrak{M}_\P$ it is sufficient to fix a choice of an orientation of $\textup{ort}(x_0)$ for just one ``base point'' $x_0^s\in  \mathfrak{M}_\P^s.$   The hyperbolic locus $\mathfrak{M}_\P^h\subseteq \mathfrak{M}_\P$ is simply connected, and hyperbolic simplices have a canonical orientation, so any point of $\mathfrak{M}_\P^h$ can be taken as the base point. If $\P$ is odd, we define $\textup{ort}_P(x^s)$ as the scissors congruence class of the oriented orthoscheme $\textup{ort}(x^s)$ with the canonical orientation of the hyperbolic orthoscheme. If $\P$ is even, we use the orientation opposite to the canonical. 

Our next goal is to compute the ``Dehn invariant'' of an oriented orthoscheme, namely the coproduct $\Delta^{\mathcal{G}}[\textup{ort}(x^s)]$ in $\mathcal{G}.$ Recall that we have defined a set of alternating subpolygons $\textup{alt}(\P)$ in \S \ref{SectionFormalCluster}.

\begin{proposition}\label{PropositionOrthoschemeDehn}
 For an alternating polygon $\P$ the following equality holds
\[
\Delta^{\mathcal{G}}\textup{ort}_\P=\sum_{\mathrm{S}\in\textup{alt}(\P)} \textup{ort}_{\mathrm{S}} \otimes \left (\prod_{i=0}^{2r}\textup{ort}_{\mathrm{S}^i} \right).
\]
\end{proposition}
\begin{proof}
By (\ref{FormulaDehn}), we have
\[
\Delta^{\mathcal{G}}\textup{ort}_\P=\sum_{\substack{I\subseteq\{1,\dots,2n\}\\ |I| \text{ is even}}}[(\textup{ort}_\P)_I] \otimes \left[(\textup{ort}_\P)^I \right].
\]
It is easy to see that $\left[(\textup{ort}_\P)^I \right]=0$ if $I$ is not alternating. Indeed,
by Proposition \ref{PropositionFacesAngles}, if $I$ is even and not alternating, then $\left[(\textup{ort}_\P)^I \right]$ is a product of classes of orthoschemes, among which at least two are even-dimensional. The product of classes of even-dimensional simplices vanishes, see \cite[Lemma 3.10]{Gon99}. From here the statement follows.
\end{proof}

\subsection{Alternating polylogarithms} \label{SectionAlternatingPolylogarithms}
For $a_1,\dots,a_n\in \mathbb{C}$ fix the choice of square roots $\sqrt{a_1},\dots, \sqrt{a_n}\in \mathbb{C}.$ An alternating polylogarithm is a framed mixed Hodge structure defined by the following formula:
\[
\begin{split}
&\textup{ALi}^{\H}_{n_1,\dots,n_k}(a_1,\dots,a_k)=\\
&\frac{1}{2^{k}}\sum_{\epsilon_1,\dots, \epsilon_{k}\in \{-1,1\} } \left ( \prod_{i=1}^k \epsilon_i\right ) \Li^{\H}_{n_1,\dots,n_k}(\epsilon_1 \sqrt{a_1},\dots,\epsilon_k  \sqrt{a_k}).
\end{split}
\]
Similar functions appeared in \cite[Definition 3.12]{Cha21} under the name {\it multiple $t$-polylogarithms}.
\begin{example} Here are some examples in weights $1$ and $2:$
\[
\begin{split}
&\textup{ALi}^{\H}_{1}(a_1)=\frac{1}{2}(\Li^{\H}_{1}(\sqrt{a_1})-\Li^{\H}_{1}(-\sqrt{a_1}))=\frac{1}{2}\log^\H\left (\frac{1+\sqrt{a_1}}{1-\sqrt{a_1}}\right);\\
&\textup{ALi}^{\H}_{2}(a_1)=\frac{1}{2}(\Li^{\H}_{2}(\sqrt{a_1})-\Li^{\H}_{2}(-\sqrt{a_1}));\\
&\textup{ALi}^{\H}_{1,1}(a_1,a_2)\\
&=\frac{1}{4}\left (\Li^{\H}_{1,1}(\sqrt{a_1},\sqrt{a_2})-\Li^{\H}_{1,1}(-\sqrt{a_1},\sqrt{a_2})-\Li^{\H}_{1,1}(\sqrt{a_1},-\sqrt{a_2})+\Li^{\H}_{1,1}(-\sqrt{a_1},-\sqrt{a_2})\right ).\\
\end{split}
\]
\end{example}

Assume that $S$ is an irreducible algebraic variety over $\C.$ For $\varphi_1,\dots,\varphi_k\in \C(S)^\times$ consider a covering $\widetilde{S}$ of $S$ such that $\sqrt{\varphi_1},\dots, \sqrt{\varphi_k}$ are regular on $\widetilde{S}.$  We define
\[
\textup{ALi}^{\H}[\varphi_1,n_1|\varphi_2,n_2|\dots|\varphi_k,n_k]=
\textup{ALi}^{\H}_{n_1,\dots,n_k}(\varphi_1,\varphi_2,\dots,\varphi_k) \in \H_{\widetilde{S}}.
\]
\begin{proposition} \label{LemmaAlternatingCoproduct} 
The coproduct of alternating polylogarithms can be computed by the following formula:
\[
\begin{split}
&\Delta^{\H\H}\textup{ALi}^{\H}([\varphi_1,n_1|\dots|\varphi_k,n_k])\\
&=\sum_{s=0}^k \left ((\textup{ALi}^{\H}[\varphi_1,n_1|\dots|\varphi_s,n_s]\otimes 1) \left ((1\otimes \textup{ALi}^{\H})\Delta^{\H\F}[\varphi_{s+1},n_{s+1}|\dots|\varphi_k,n_k]\right )\right).
\end{split}
\]
\end{proposition}
\begin{proof}

It is sufficient to show that
\begin{equation}\label{EquationAlternatingProp}
\begin{split}
&(1\otimes \Li^{\H})\left(\frac{1}{2^{k}}\sum_{\epsilon_{1},\dots, \epsilon_{k}\in \{-1,1\} } \left ( \prod_{i=1}^k \epsilon_i\right ) \Delta^{\H\F}[\epsilon_{1} \sqrt{\varphi_{1}},n_{1}|\dots|\epsilon_k\sqrt{\varphi_k},n_k]\right )\\
&=(1\otimes \textup{ALi}^{\H})\Delta^{\H\F}[\varphi_{1},n_{1}|\dots|\varphi_k,n_k];\\
\end{split}
\end{equation}
after that the statement would follow from (\ref{FormulaCoproductMultiplePolylogarithm}).

For $\epsilon=(\epsilon_1,\dots,\epsilon_k)$ consider a sequence
\[
\begin{split}
	&x^{\epsilon}=(x_0^{\epsilon},\dots,x_{n+1}^{\epsilon})\\
&=(0,1,\underbrace{0,\dots,0, \epsilon_1 \sqrt{\varphi_1}}_{n_1},\underbrace{0,\dots,0, \epsilon_1\epsilon_2 \sqrt{\varphi_1} \sqrt{\varphi_2}}_{n_2},\dots,\underbrace{0,\dots,0,\epsilon_1\epsilon_2\dots\epsilon_k \sqrt{\varphi_1} \sqrt{\varphi_2}\dots \sqrt{\varphi_{k}}}_{n_k}).
\end{split}
\]
By (\ref{FormulaCoaction}) we have
\[
\begin{split}
&\Delta^{\mathcal{HF}}[\epsilon_{1} \sqrt{\varphi_{1}},n_{1}|\dots|\epsilon_k\sqrt{\varphi_k},n_k]\\
&=(-1)^{l(x^{\epsilon})-l(x_I^{\epsilon})}\sum_{I=(i_0,\dots,i_{r+1})} \left (\prod_{p=1}^{r} \textup{I}^{\H}(x_{i_p}^{\epsilon};x_{i_p+1}^{\epsilon},\dots,x_{i_{p+1}-1}^{\epsilon};x_{i_{p+1}^{\epsilon}})\right )\otimes  x_I^{\epsilon}
\end{split}
\]
where the summation goes over all sequences $0=i_0<i_1<\dots<i_r<i_{r+1}=n+1$ with $i_1=1.$ 
From (\ref{EqualityRoots}) it follows that for any sequence $I$ we have
\[
\begin{split}
&\frac{1}{2^{k}}\sum_{\epsilon_{1},\dots, \epsilon_{k}\in \{-1,1\} } \left ( \prod_{i=1}^k \epsilon_i\right )  \left (\prod_{p=1}^{r} \textup{I}^{\H}\left(x_{i_p}^{\epsilon};x_{i_p+1}^{\epsilon},\dots,x_{i_{p+1}-1}^{\epsilon};x_{i_{p+1}}^{\epsilon}\right)\right )\otimes  \Li^{\H}(x_I^{\epsilon})\\
&=\left (\prod_{p=1}^{r} \textup{I}^{\H}\left(x_{i_p}^{2};x_{i_p+1}^{2},\dots,x_{i_{p+1}-1}^{2};x_{i_{p+1}}^{2}\right)\right )\otimes  \textup{ALi}^{\H}(x_I^{\epsilon}),
\end{split}
\]
which implies (\ref{EquationAlternatingProp}).
\end{proof}

\begin{example} We have 

\begin{align*}
\Delta^\H \textup{ALi}^{\H}_{1,1}(a_1,a_2)&=1 \otimes \textup{ALi}^{\H}_{1,1}(a_1,a_2)+\textup{ALi}^{\H}_{1}(a_1)\otimes \textup{ALi}^{\H}_{1}(a_2)+\textup{ALi}^{\H}_{1,1}(a_1,a_2)\otimes 1\\
&\quad-\textup{I}^\H(1,a_1,a_1 a_2)\otimes\textup{ALi}^{\H}_{1}(a_1a_2).\\
\end{align*}
\end{example}

\begin{corollary}\label{LemmaAlternatingQShuffle}
	Alternating polylogarithms satisfy the quasi-shuffle relation
\[
\textup{ALi}^{\H}(x\star y)=\textup{ALi}^{\H}(x)\textup{ALi}^{\H}(y)
\]
for $x,y \in \F_{\tilde{S}}.$
\end{corollary}
\begin{proof}
	 The statement follows from Proposition \ref{LemmaAlternatingCoproduct}, see the proof of Proposition \ref{PropositionMapToPolylogarithms}. 
\end{proof}

\subsection{Volumes of orthoscheme}\label{SectionVolumeOrthoschemes}
Our goal in this section is to give an explicit formula for the unipotent variation of framed mixed Hodge-Tate structures $h(\textup{ort}_P)$ over $\mathfrak{M}_{\P}^s$ in terms of alternating polylogarithms. For that, we first construct  a collection of functions on $\mathfrak{M}_\P^s,$ which are square roots of cross-ratios. 

We start with the case $n=1.$ For an alternating polygon $\P=(p_0,p_1,p_2,p_3)$ consider a point $x^{s}\in \mathfrak{M}_{\P}^s$ over 
\[
x=(x_{p_0}, x_{p_1}, x_{p_2},x_{p_3})\in \mathfrak{M}_\P.
\] 
Then $\textup{ort}(x^{s})=(Q;H_1,H_2),$ where $Q$ is a quadric of dimension zero, thus a  pair of points $Q=\{Q_1,Q_2\}$ in $\mathbb{P}^1.$ Since $Q$ is oriented, the pair $\{Q_1,Q_2\}$ is ordered. It follows from the proof of Theorem \ref{TheoremOrthoschemesConfigurations}
that the configuration $(x_{p_0},x_{p_1},x_{p_2},x_{p_3})$ is projectively equivalent to the configuration $(H_0,H_1,H_2,H_3),$ where $H_0=H_2^{\perp}$  and $H_3=H_1^{\perp}.$ A direct computation shows that 
\[
[x_{p_0}, x_{p_1}, x_{p_2},x_{p_3}]=[H_0,H_1,H_2,H_3]=\left(\frac{[Q_1, H_1,Q_2, H_2]+1}{[Q_1, H_1,Q_2, H_2]-1}\right)^2.
\]
We put 
\be\label{FormulaSquareRoot}
\sqrt{\textup{cr}(\P)}=
\begin{cases}
	\dfrac{[Q_1, H_1,Q_2, H_2]+1}{[Q_1, H_1,Q_2, H_2]-1}& \text{\ if\ } \P \text{\ is even, }\\
	\\
		\dfrac{[Q_2, H_1,Q_1, H_2]-1}{[Q_2, H_1,Q_1, H_2]+1} &\text{\ if\ } \P \text{\ is odd}.
\end{cases}
\ee

\begin{lemma} Consider an  alternating polygon $\P.$ The product 
\[
\sqrt{\textup{cr}(\P)}=\prod_{i=1}^n  \sqrt{\textup{cr}(\mathrm{Q}_i)}
\]
does not depend on the choice of a quadrangulation $Q=\{\mathrm{Q}_1,\dots,\mathrm{Q}_n\}\in \mathcal{Q}(\P).$  \end{lemma}
\begin{proof} Since $\textup{cr}(\P)$ does not depend on the choice of a quadrangulation, the product $\prod_{i=1}^n  \sqrt{\textup{cr}(\mathrm{Q}_i)}$ can depend on the quadrangulation only up to a sign. Thus, it is enough to check the statement for any particular oriented orthoscheme, e.g., hyperbolic. Assume that $\P$ is odd, in which case  the orthoscheme has the canonical orientation. By (\ref{FormulaSquareRoot}) $\sqrt{\textup{cr}(\mathrm{Q}_i)}$ is real and positive if $\mathrm{Q}_i$ corresponds to the edge length. Similarly, $\sqrt{\textup{cr}(\mathrm{Q}_i)}$ equals to $\lambda i$ for positive $\lambda$ if  $\mathrm{Q}_i$ corresponds to an angle. The statement follows because every quadrangulation has one quadrangle corresponding to an edge of the orthoscheme and $n-1$ .quadrangles corresponding to angles.
\end{proof}

For an alternating polygon $\P$ we put $\H_\P^s$ to be $\H_S$ for $S=\mathfrak{M}_{\P}^s.$ Since square roots of cross-ratios are regular functions on  $\mathfrak{M}_{\P}^s,$ we have an element
\[
\textup{ALi}^{\H}(\textup{T}_\P)\in \H_\P^s.
\]

\begin{theorem} \label{TheoremVolumeOrt}
We have the following equality of framed variations on $\mathfrak{M}_{\P}^s:$ 
\be \label{FormulaOrthoschemesPolylogarithms}
h(\textup{ort}_\P^s)=\textup{ALi}^{\H}(\textup{T}_\P).
\ee
\end{theorem}
\begin{proof} 

We prove the statement by induction on $n.$ We start with the case $n=1.$
Consider a point 
$
x^{s}=(x_{p_0}, x_{p_1}, x_{p_2},x_{p_3})\in \mathfrak{M}_{\P}^s.
$
From (\ref{FormulaSquareRoot}) we have an equality
\[
[Q_1, H_1,Q_2, H_2]=
\begin{cases}
	\dfrac{\sqrt{\textup{cr}(\P)}+1}{\sqrt{\textup{cr}(\P)}-1}& \text{\ if\ } \P \text{\ is even, }\\
	\\
		\dfrac{1-\sqrt{\textup{cr}(\P)}}{1+\sqrt{\textup{cr}(\P)}}&\text{\ if\ } \P \text{\ is odd}.
\end{cases}
\]
Recall that by Example \ref{ExampleOneDimensionalSimplex} we have an equality
$
h[\textup{ort}(x^{s})]=-\frac{1}{2}\log^\H\left([Q_1, H_1,Q_2, H_2] \right).
$
For an even polygon $\P=(0,1,2,3)$
\[
\begin{split}
&\textup{ALi}^{\H}_{1}(\textup{T}_\P)=-\textup{ALi}^{\H}_{1}([(0,1,2,3),1])=-\frac{1}{2}\log^\H\frac{1+\sqrt{\textup{cr}(0,1,2,3)}}{1-\sqrt{\textup{cr}(0,1,2,3)}},\\
\end{split}
\]
so $h(\textup{ort}_\P)=\textup{ALi}^{\H}(\textup{T}_\P).$ For an odd polygon $\P=(0,1,2,3)$ we have 
\[
\begin{split}
&\textup{ALi}^{\H}_{1}(\textup{T}_\P)=\textup{ALi}^{\H}_{1}([(1,2,3,4),1])\\
&=\frac{1}{2}\log^\H\frac{1+\sqrt{\textup{cr}(1,2,3,4)}}{1-\sqrt{\textup{cr}(1,2,3,4)}}=-\frac{1}{2}\log^\H\frac{1-\sqrt{\textup{cr}(1,2,3,4)}}{1+\sqrt{\textup{cr}(1,2,3,4)}},\\
\end{split}
\]
so again $h(\textup{ort}_\P)=\textup{ALi}^{\H}(\textup{T}_\P).$ This finishes the proof for $n=1.$

Next, assume that $n>1.$ From Theorem \ref{TheoremFormalQF}	we know that $\Delta^{\H\F}\textup{T}_\P=0.$ Thus Proposition~\ref{LemmaAlternatingCoproduct} and Corollary~\ref{LemmaAlternatingQShuffle} imply that
\[
\Delta^{\H\H}\textup{ALi}^{\H}(\textup{T}_\P)=(\textup{ALi}^{\H}\otimes \textup{ALi}^{\H})\Delta^{\F\F}\textup{T}_\P=\sum_{S\in\textup{Alt}(\P)}\textup{ALi}^{\H}(\textup{T}_{\mathrm{S}})\otimes \left (\prod_{i=0}^{2r}\textup{ALi}^{\H}(\textup{T}_{\mathrm{S}_i})\right). 
\]
  Comparing it with Proposition \ref{PropositionOrthoschemeDehn}  we see  that by induction we have
\[
\Delta^{\H\H}\textup{ALi}^{\H}(\textup{T}_\P)=\Delta^{\H\H}h(\textup{ort}_\P),
\]
   so the variation $\textup{ALi}^{\H}(\textup{T}_\P)-h(\textup{ort}_\P)$ is constant on $\mathfrak{M}_{\P}^s.$ On the divisor $x_{p_0}=x_{p_1}$ both sides equal to zero. Indeed, this is obvious for alternating polylogarithms. For $h(\textup{ort}_\P)$ it follows from the fact that in this specialization the quadric $Q$ becomes singular, which can be easily deduced from  (\ref{FormulaQuad3}). This finishes the proof of the theorem.
\end{proof}

Applying the real period map (\ref{FormulaRealPeriod}) to (\ref{FormulaOrthoschemesPolylogarithms}) we obtain Theorem \ref{MainTheoremQuadrangulationVolume}. 

\bibliographystyle{alpha}      
\bibliography{Polylogarithms_Bibliography}

\begin{thebibliography}{CGR19b}

\bibitem[Aom77]{Aom77}
Kazuhiko Aomoto.
\newblock Analytic structure of {{S}chl\"{a}fli} function.
\newblock {\em Nagoya Math. J.}, 68:1--16, 1977.

\bibitem[BD20]{BD20}
Francis C.~S. Brown and Claude Duhr.
\newblock A double integral of dlog forms which is not polylogarithmic.
\newblock arXiv:2006.09413, 2020.

\bibitem[BG47]{BG47}
Andr\'{e} Bloch and Gustave Guillaumin.
\newblock Sur le volume des poly\`edres non euclidiens.
\newblock {\em C. R. Acad. Sci. Paris}, 224:1690--1692, 1947.

\bibitem[BGSV90]{BGSV90}
A.~Beilinson, A.~Goncharov, V.~Shekhtman, and A.~Varchenko.
\newblock Projective geometry and {$K$}-theory.
\newblock {\em Algebra i Analiz}, 2(3):78--130, 1990.

\bibitem[BH80]{BH80}
Johannes B\"{o}hm and Eike Hertel.
\newblock {\em Polyedergeometrie in {$n$}-dimensionalen {{R}\"{a}umen}
  konstanter {K}r\"{u}mmung}, volume~14 of {\em Mathematische Monographien
  [Mathematical Monographs]}.
\newblock VEB Deutscher Verlag der Wissenschaften, Berlin, 1980.

\bibitem[BMS87]{BMS87}
A.~Beilinson, R.~MacPherson, and V.~Schechtman.
\newblock Notes on motivic cohomology.
\newblock {\em Duke Math. J.}, 54(2):679--710, 1987.

\bibitem[B{\"o}h60]{Boh60}
Johannes B{\"o}hm.
\newblock Inhaltsmessung im {$R_{5}$} konstanter {K}r\"{u}mmung.
\newblock {\em Arch. Math. (Basel)}, 11:298--309, 1960.

\bibitem[B{\"o}h64]{Boh63}
Johannes B{\"o}hm.
\newblock Zu {C}oxeters {I}ntegrationsmethode in gekr\"{u}mmten {R}\"{a}umen.
\newblock {\em Math. Nachr.}, 27:179--214, 1963/64.

\bibitem[Bro17]{Bro17}
Francis Brown.
\newblock Notes on motivic periods.
\newblock {\em Commun. Number Theory Phys.}, 11(3):557--655, 2017.

\bibitem[CGR19a]{CGR19}
Steven Charlton, Herbert Gangl, and Danylo Radchenko.
\newblock Explicit formulas for {G}rassmannian polylogarithms.
\newblock arXiv:1909.13869, 2019.

\bibitem[CGR19b]{CGR19b}
Steven Charlton, Herbert Gangl, and Danylo Radchenko.
\newblock On functional equations for {N}ielsen polylogarithms.
\newblock arXiv:1908.04770, 2019.

\bibitem[Cha17]{Cha17}
Steven Charlton.
\newblock A review of {D}an's reduction method for multiple polylogarithms.
\newblock arXiv:1703.03961, 2017.

\bibitem[Cha21]{Cha21}
Steven Charlton.
\newblock On motivic multiple $t$ values, {S}aha's basis conjecture, and
  generators of alternating {MZV}'s.
\newblock arXiv:2112.14613, 2021.

\bibitem[CK99]{CK99}
Alain Connes and Dirk Kreimer.
\newblock Hopf algebras, renormalization and noncommutative geometry.
\newblock In {\em Quantum field theory: perspective and prospective ({L}es
  {H}ouches, 1998)}, volume 530 of {\em NATO Sci. Ser. C Math. Phys. Sci.},
  pages 59--108. Kluwer Acad. Publ., Dordrecht, 1999.

\bibitem[Cox35]{Cox35}
H.~S.~M. Coxeter.
\newblock {The functions of {S}chl{\"a}fli and {L}obatschefsky}.
\newblock {\em The Quarterly Journal of Mathematics}, os-6(1):13--29, 01 1935.

\bibitem[Cox36]{Cox36}
H.~S.~M. Coxeter.
\newblock On {S}chl{\"a}fli's generalization of {N}apier's pentagramma
  mirificum.
\newblock {\em Bull. Calcutta Math. Soc.}, 28:123--144, 1936.

\bibitem[Del71a]{Del71}
Pierre Deligne.
\newblock Th\'{e}orie de {H}odge. {I}.
\newblock In {\em Actes du {C}ongr\`es {I}nternational des {M}ath\'{e}maticiens
  ({N}ice, 1970), {T}ome 1}, pages 425--430. 1971.

\bibitem[Del71b]{Del71b}
Pierre Deligne.
\newblock Th\'{e}orie de {H}odge. {II}.
\newblock {\em Inst. Hautes \'{E}tudes Sci. Publ. Math.}, (40):5--57, 1971.

\bibitem[Del74]{Del74}
Pierre Deligne.
\newblock Th\'{e}orie de {H}odge. {III}.
\newblock {\em Inst. Hautes \'{E}tudes Sci. Publ. Math.}, (44):5--77, 1974.

\bibitem[DG05]{DG05}
Pierre Deligne and Alexander Goncharov.
\newblock Groupes fondamentaux motiviques de {T}ate mixte.
\newblock {\em Ann. Sci. \'{E}cole Norm. Sup. (4)}, 38(1):1--56, 2005.

\bibitem[Gan16]{Gan16}
Herbert Gangl.
\newblock Multiple polylogarithms in weight 4.
\newblock arXiv:1609.05557, 2016.

\bibitem[Gau66]{Gau66}
C.~F. Gauss.
\newblock Pentagramma {M}irificum.
\newblock In {\em Werke, Band III: Analysis.}, volume BD. III, pages 481--490.
  G{\"o}ttingen: K{\"o}nigliche Gesellschaft der Wissenschaften., 1866.

\bibitem[Gon95a]{Gon95B}
A.~B. Goncharov.
\newblock Geometry of configurations, polylogarithms, and motivic cohomology.
\newblock {\em Adv. Math.}, 114(2):197--318, 1995.

\bibitem[Gon95b]{Gon95}
Alexander Goncharov.
\newblock Polylogarithms in arithmetic and geometry.
\newblock In {\em Proceedings of the {I}nternational {C}ongress of
  {M}athematicians, {V}ol. 1, 2 ({Z}\"{u}rich, 1994)}, pages 374--387.
  Birkh\"{a}user, Basel, 1995.

\bibitem[Gon99]{Gon99}
Alexander Goncharov.
\newblock Volumes of hyperbolic manifolds and mixed {T}ate motives.
\newblock {\em J. Amer. Math. Soc.}, 12(2):569--618, 1999.

\bibitem[Gon01]{Gon01}
Alexander Goncharov.
\newblock Multiple polylogarithms and mixed {T}ate motives.
\newblock arXiv:math/0103059, 2001.

\bibitem[Gon02]{Gon02}
Alexander Goncharov.
\newblock Periods and mixed motives.
\newblock arXiv:math/0202154, 2002.

\bibitem[Gon05]{Gon05}
Alexander Goncharov.
\newblock Galois symmetries of fundamental groupoids and noncommutative
  geometry.
\newblock {\em Duke Math. J.}, 128(2):209--284, 2005.

\bibitem[Gon19]{Gon19}
Alexander Goncharov.
\newblock Hodge correlators.
\newblock {\em J. Reine Angew. Math.}, 748:1--138, 2019.

\bibitem[GR18]{GR18}
Alexander Goncharov and Daniil Rudenko.
\newblock Motivic correlators, cluster varieties and {Z}agier's conjecture on
  $\zeta_\mathrm{F}(4)$.
\newblock arXiv:1803.08585, 2018.

\bibitem[GZ18]{GZ18}
Alexander Goncharov and Guangyu Zhu.
\newblock The {G}alois group of the category of mixed {H}odge-{T}ate
  structures.
\newblock {\em Selecta Math. (N.S.)}, 24(1):303--358, 2018.

\bibitem[Had56]{Had56}
Hugo Hadwiger.
\newblock Ungel\"{o}ste {P}robleme.
\newblock {\em Elemente der Mathematik}, 11:109--110, 1956.

\bibitem[Hai87]{Hai87}
Richard Hain.
\newblock The geometry of the mixed hodge structure on the fundamental group.
\newblock {\em Proc. Symp. Pure Math}, pages 247--281, 1987.

\bibitem[HI17]{HI17}
Michael~E. Hoffman and Kentaro Ihara.
\newblock Quasi-shuffle products revisited.
\newblock {\em J. Algebra}, 481:293--326, 2017.

\bibitem[Hof00]{Hof00}
Michael~E. Hoffman.
\newblock Quasi-shuffle products.
\newblock {\em J. Algebraic Combin.}, 11(1):49--68, 2000.

\bibitem[Kel92]{Kel92}
Ruth Kellerhals.
\newblock On the volumes of hyperbolic {$5$}-orthoschemes and the trilogarithm.
\newblock {\em Comment. Math. Helv.}, 67(4):648--663, 1992.

\bibitem[Kel95]{Kel95}
Ruth Kellerhals.
\newblock Volumes in hyperbolic {$5$}-space.
\newblock {\em Geom. Funct. Anal.}, 5(4):640--667, 1995.

\bibitem[KM19]{KM19}
Kevin Kordek and Dan Margalit.
\newblock Representation stability in the level 4 braid group.
\newblock arXiv:1903.03119, 03 2019.

\bibitem[KZ01]{KZ01}
Maxim Kontsevich and Don Zagier.
\newblock Periods.
\newblock In {\em Mathematics unlimited---2001 and beyond}, pages 771--808.
  Springer, Berlin, 2001.

\bibitem[Lob36]{Lob36}
Nikolai~I. Lobachevsky.
\newblock Application of {I}maginary {G}eometry to certain integrals.
\newblock {\em Uchenye zapiski kazanskogo imperatorskogo universiteta},
  I:3--166, 1836.

\bibitem[LV80]{LV80}
G\'{e}rard Lion and Mich\`ele Vergne.
\newblock {\em The {W}eil representation, {M}aslov index and theta series},
  volume~6 of {\em Progress in Mathematics}.
\newblock Birkh\"{a}user, Boston, Mass., 1980.

\bibitem[Mal20]{Mal20}
Nikolay Malkin.
\newblock Shuffle relations for {H}odge and motivic correlators.
\newblock arXiv:2003.06521, 2020.

\bibitem[Mol77]{Mol77}
Richard~K. Molnar.
\newblock Semi-direct products of {H}opf algebras.
\newblock {\em J. Algebra}, 47(1):29--51, 1977.

\bibitem[M{\"u}l54]{Mul54}
Paul~F. M{\"u}ller.
\newblock {\em {\"U}ber Simplexinhalte in nichteuklidischen R{\"a}umen.}
\newblock PhD thesis, Universit{"a}t Bonn, 1954.

\bibitem[Tho06]{Tho06}
Teruji Thomas.
\newblock The {M}aslov index as a quadratic space.
\newblock {\em Math. Res. Lett.}, 13(5-6):985--999, 2006.

\bibitem[Wec91]{Wec91}
Gerd Wechsung.
\newblock Functional equations of hyperlogarithms.
\newblock In {\em Structural properties of polylogarithms}, volume~37 of {\em
  Math. Surveys Monogr.}, pages 171--184. Amer. Math. Soc., Providence, RI,
  1991.

\end{thebibliography}

\end{document}